\documentclass[11pt]{scrartcl}
\usepackage[english]{babel}
\usepackage{mathtools}
\usepackage{amssymb}
\usepackage{amscd}
\usepackage{amsthm}
\usepackage{eucal}
\usepackage{tikz-cd}
\usepackage{geometry}
\usepackage{abstract}
\usepackage{setspace}
\usepackage[utf8]{inputenc}
\usepackage{comment}
\usepackage{blkarray}
\usepackage{hyperref}
\usepackage{xcolor}
\usepackage{mathdots}
\hypersetup{
   colorlinks,
   linkcolor={blue!50!black},
   citecolor={green!50!black},
   urlcolor={blue!80!black}
}
\usepackage{graphicx}
\graphicspath{ {Images/} }

\usetikzlibrary{decorations.markings,arrows.meta}
\tikzset
  {midarrow/.style={decoration={markings,mark=at position 0.5 with
     {\arrow[xshift=2pt]{Latex[length=4pt,width=5pt,#1]}}},postaction={decorate}}
  }

\theoremstyle{theorem}
\newtheorem{theorem}{Theorem}[section]
\newtheorem*{theorem*}{Theorem}
\newtheorem{corollary}[theorem]{Corollary}

\newtheorem{proposition}[theorem]{Proposition}
\newtheorem{definition}[theorem]{Definition}

\theoremstyle{definition}
\newtheorem{remark}[theorem]{Remark}

\newtheorem{examplex}[theorem]{Example}
\newenvironment{example}
  {\pushQED{\qed}\examplex}
  {\popQED\endexamplex}
  
\newcommand{\verteq}{\rotatebox{270}{$\,\cong$}}

\makeatletter
\DeclareRobustCommand{\cev}[1]{%
  \mathpalette\do@cev{#1}%
}
\newcommand{\do@cev}[2]{%
  \fix@cev{#1}{+}%
  \reflectbox{$\m@th#1\vec{\reflectbox{$\fix@cev{#1}{-}\m@th#1#2\fix@cev{#1}{+}$}}$}%
  \fix@cev{#1}{-}%
}
\newcommand{\fix@cev}[2]{%
  \ifx#1\displaystyle
    \mkern#23mu
  \else
    \ifx#1\textstyle
      \mkern#23mu
    \else
      \ifx#1\scriptstyle
        \mkern#22mu
      \else
        \mkern#22mu
      \fi
    \fi
  \fi
}

\title{Persistent sheaf cohomology}
\author{Florian Russold \\ Institute of Geometry, Graz University of Technology, Austria}
 
\begin{document}

\maketitle

\begin{abstract}
\noindent
We expand the toolbox of (co)homological methods in computational topology by applying the concept of persistence to sheaf cohomology. Since sheaves (of modules) combine topological information with algebraic information, they allow for variation along an algebraic dimension and along a topological dimension. Consequently, we introduce two different constructions of sheaf cohomology (co)persistence modules. One of them can be viewed as a natural generalization of the construction of simplicial or singular cohomology copersistence modules. We discuss how both constructions relate to each other and show that, in some cases, we can reduce one of them to the other. Moreover, we show that we can combine both constructions to obtain two-dimensional (co)persistence modules with a topological and an algebraic dimension. We also show that some classical results and methods from persistence theory can be generalized to sheaves. Our results open up a new perspective on persistent cohomology of filtrations of simplicial complexes.   
\end{abstract}

\section{Introduction} \label{986}

In recent years, the field of computational topology, especially the field of topological data
analysis, gained a lot of popularity. One of its most prominent tools is persistent (co)homology \cite{perscoho,edelsbrunner,carlsson}. The motivation for persistent homology is to investigate the homology of data sets or, more precisely, the homology of the support of the underlying data distribution. To extract homological information from a point set $P\subseteq \mathbb{R}^d$, one constructs a parameter-dependent simplicial complex model $K_\epsilon$ of the subspace from which the data set is sampled using for example the \v{C}ech or the Vietoris-Rips construction. The problem is that the homology of these models depends on the scale parameter $\epsilon$ and that there is not necessarily a ``right choice" for this parameter. This problem can be tackled by constructing a nested sequence of simplicial complex models (upper row of (\ref{672})) 
\begin{equation} \label{672}
\begin{tikzcd}
K_{\epsilon_0} \arrow[r,hook] & K_{\epsilon_1} \arrow[r,hook] & \cdots \arrow[r,hook] & K_{\epsilon_{n-2}} \arrow[r,hook] & K_{\epsilon_{n-1}} \\[-15pt]
H_k(K_{\epsilon_0}) \arrow[r] & H_k(K_{\epsilon_1}) \arrow[r] & \cdots \arrow[r] & H_k(K_{\epsilon_{n-2}}) \arrow[r] & H_k(K_{\epsilon_{n-1}})
\end{tikzcd}
\end{equation}
with respect to increasing scale parameters $\epsilon_0<\ldots<\epsilon_{n-1}$, computing the $k$-dimensional simplicial homology $H_k(K_{\epsilon_i})$ (lower row of (\ref{672})) and investigating the stability or persistence of homological features along a range of parameter values. The information about the stability of homological features in such a filtered simplicical complex is provided by its persistent homology and is commonly represented as a so-called persistence barcode \cite{cohensteiner,oudot}.

Even more recently, applied sheaf theory started to gain momentum in the computational topology community. From an applied mathematics perspective, sheaves can be viewed as models of (local) information attached to a space. The idea is to model objects of interest as sheaves that can be stored and manipulated on a computer and investigate these models by the means of sheaf theory. Foundational work in this field was done by Shepard and Curry in their respective theses \cite{curry,shepard}, where they developed the theory of cellular sheaves and cosheaves. Cellular sheaves are a kind of ``discrete" sheaves on cell complexes that allow for practical computations in finite cases. Cellular sheaves, or more generally, sheaves on finite posets, are used, for example, to describe information flows in networks \cite{ghrist,hansen,robinson4}, for sensor integration and data fusion \cite{robinson2,robinson3} or for stratification learning \cite{brown}. Much more about applications of sheaves can be found in \cite{curry,ghristbook,robinsonbook}. One of the main tools of sheaf theory and, in particular, applied sheaf theory is sheaf cohomology. Sheaf cohomology can be used to investigate local to global inference problems. For example, sheaf cohomology (or cosheaf homology) is used to compute global persistent (co)homology from local persistent (co)homology \cite{curry,morse,yoonghrist,casas}.

\subsection{Our contribution}

The goal of this work is to extend the theory of persistence to sheaf cohomology. Since sheaves (of modules) combine topological information with algebraic information, they allow for variation along an algebraic dimension and along a topological dimension. Therefore, sheaf theory allows for two different constructions of sheaf cohomology (co)persistence modules and, consequently, for two different versions of persistent sheaf cohomology. 

In the applications discussed above, a common theme is to model an object of interest as a simplicial complex or a sheaf and investigate this model by the means of (co)homology. The principles of the construction of persistent homology, namely functoriality (of homology) and unique decomposability into interval modules, can also be applied in the context of sheaf theory. Instead of considering parameter-dependent simplicial complex models of topological spaces or data sets, we can consider parameter-dependent sheaf models of some object of interest. Given a linear diagram of sheaves and sheaf morphisms on some topological space $X$ (upper row of (\ref{856})),
\begin{equation} \label{856}
\begin{tikzcd}
\mathcal{F}_0 \arrow[r] & \mathcal{F}_1 \arrow[r] & \cdots \arrow[r] & \mathcal{F}_{n-2} \arrow[r] & \mathcal{F}_{n-1} \\[-15pt]
H^k(X,\mathcal{F}_0) \arrow[r] & H^k(X,\mathcal{F}_1) \arrow[r] & \cdots \arrow[r] & H^k(X,\mathcal{F}_{n-2}) \arrow[r] & H^k(X,\mathcal{F}_{n-1})
\end{tikzcd}
\end{equation}
we can track the $k$-dimensional sheaf cohomology $H^k(X,\mathcal{F}_i)$ along this diagram (lower row of (\ref{856})) in the same way as we track simplicial homology along a linear diagram of simplicial complexes (\ref{672}). If we consider vector space valued sheaves, we obtain an interval decomposable persistence module and, moreover, a persistence barcode where the bars correspond to persistent sheaf cohomology classes along the diagram. In this way, we extend the cohomological investigation of sheaf models to families of parameter-dependent models. In this approach, the topological space $X$ is fixed, whereas the algebraic information attached to $X$ varies along the upper diagram in (\ref{856}). Therefore, we call the obtained persistence module a \emph{sheaf persistence module of algebraic type} or a \emph{persistence module of type $A$}. 

Sheaf theory also offers the following alternative way of generalizing persistent (co)homology to persistent sheaf cohomology. Given a linear diagram of topological spaces and continuous maps, as depicted in the middle row of (\ref{987}), we can pull back a fixed sheaf $\mathcal{F}_{n-1}$ on $X_{n-1}$ along the continuous maps to obtain sheaves on all the spaces in the diagram. These pullbacks along continuous maps induce morphisms in cohomology and lead to the linear diagram of cohomology modules depicted in the lower row of (\ref{987}).
\begin{equation} \label{987}
\begin{tikzcd}
\mathcal{F}^0 & \arrow[l,swap,"f_{0}^*",maps to] \mathcal{F}^1 & \arrow[l,swap,"f_{1}^*",maps to] \cdots & \arrow[l,swap,"f_{n-3}^*",maps to] \mathcal{F}^{n-2} & \arrow[l,swap,"f_{n-2}^*",maps to] \mathcal{F}^{n-1} \\[-20pt]
X_0 \arrow[r,"f_0"] & X_1 \arrow[r,"f_1"] & \cdots \arrow[r,"f_{n-2}"] & X_{n-2} \arrow[r,"f_{n-2}"] & X_{n-1} \\[-15pt]
H^k(X_0,\mathcal{F}^0) & \arrow[l] H^k(X_1,\mathcal{F}^1) & \arrow[l] \cdots & \arrow[l] H^k(X_{n-2},\mathcal{F}^{n-2}) & \arrow[l] H^k(X_{n-1},\mathcal{F}^{n-1}) 
\end{tikzcd}
\end{equation}
If we consider a sheaf of vector spaces, the obtained copersistence module is interval decomposable and we get a barcode describing the persistence of sheaf cohomology classes along the informal diagram of pairs of topological spaces and sheaves. In this approach, the algebraic information is provided by a fixed sheaf, whereas the topological spaces vary. Therefore, we call the obtained copersistence module a \emph{sheaf copersistence module of topological type} or a \emph{copersistence module of type $T$}. One can identify a topological space with the constant ($R$-valued) sheaf on this space. Moreover, (in many cases) the singular (or simplicial) cohomology of a space is the sheaf cohomology of the constant sheaf on this space. From this point of view, the construction of the cohomology copersistence module of a linear diagram of topological spaces and continuous maps can be phrased in the following way (see (\ref{672}) for the homological case). Put the constant sheaf on every space, compute the respective sheaf cohomology vector spaces and connect them by the morphisms induced by continuous maps. Hence, this construction of sheaf copersistence modules of topological type generalizes the construction of ordinary cohomology copersistence modules from diagrams of topological spaces by allowing arbitrary sheaves instead of the constant sheaf.  

To our knowledge, until now, there is no general definition, systematic construction, or investigation of any version of persistent sheaf cohomology. We fill this gap by establishing a rigorous basis of both theories of persistent sheaf cohomology discussed above. We develop both theories in a systematic way and discuss how they relate to each other. 

We show that there is a correspondence between linear diagrams of sheaves and sheaves of graded modules generalizing the well-known correspondence between persistence modules and graded modules over polynomial rings. Moreover, we show that there is also a sheaf version of the Zomorodian-Carlsson representation theorem \cite{corbet,carlsson} that contains the original theorem as a special case. We use this correspondence to prove that the persistent cohomology of a linear diagram of sheaves can be computed by the ordinary cohomology of the corresponding sheaf of graded modules. In a computational context this result enables us to compute persistent sheaf cohomology via matrix reduction, a method familiar from usual persistent homology. 

We show that there is a connection between both constructions of sheaf (co)persistence modules and that, in some cases, we can reduce the second construction to the first one. This result opens up a new perspective on simplicial persistent cohomology of filtrations of simplicial complexes as the (co)homology of a (co)sheaf of graded modules. Finally, we show that one can combine both constructions to obtain two-dimensional sheaf cohomology (co)persistence modules. This combined approach is different from the straightforward generalization of each construction to the multi-dimensional case. In our combined construction, along one direction the algebraic information over a fixed topological space changes, whereas along another direction the topological spaces change and the algebraic information over these spaces is, in some sense, fixed. 

Last but not least, we demonstrate both constructions by the means of two examples where we can provide interpretations of the meaning of the persistent sheaf cohomology classes. 

\subsection{Related work}

Persistent sheaf cohomology in our sense is hardly covered in the current literature. In \cite{yoon,yoonghrist} Yoon and Ghrist use a form of ``persistent cosheaf homology" to compute global persistent homology from local persistent homology subordinate to a cover. Given a point set $P$ and a cover of $P$ with one-dimensional nerve complex $N$, they show that the homology persistence module of a Vietoris-Rips filtration on $P$ is isomorphic to a persistence module
\begin{equation*} \label{322}
\begin{tikzcd}[column sep=large]
H_0(N,\mathcal{G}^0_k)\oplus H_1(N,\mathcal{G}^0_{k-1}) \arrow[r,"\Psi_0"] & \cdots \arrow[r,"\Psi_{n-2}"] & H_0(N,\mathcal{G}^{n-1}_k)\oplus H_1(N,\mathcal{G}^{n-1}_{k-1})
\end{tikzcd}
\end{equation*}
where $\mathcal{G}^i_k$ is a cosheaf gathering the $k$-dimensional ``local" homology at the $i$-th filtration step. Unfortunately, in general, the above persistence module is not a direct sum of persistence modules 
\begin{equation*} \label{977}
\begin{tikzcd}[column sep=huge]
H_j(N,\mathcal{G}^0_k) \arrow[r,"H_j(N\text{,}\phi^0_k)"] & \cdots \arrow[r,"H_j(N\text{,}\phi^{n-2}_k)"] & H_j(N,\mathcal{G}^{n-1}_k) \quad .
\end{tikzcd}
\end{equation*}
Hence, it is not exactly a cosheaf persistence module of algebraic type in our sense. They use spectral sequences to obtain the appropriate morphisms $\Psi_i$. This spectral sequence approach is generalized to higher-dimensional nerve complexes by Casas \cite{casas}. 

The idea of constructing a persistence module from a cellular sheaf on a finite filtered space by restricting the sheaves to the subspaces in the filtration and computing cellular sheaf cohomology was first proposed in a short note by Karthik Yegnesh \cite{yegnesh}. This construction can be viewed as a special case of our construction of sheaf copersistence modules of topological type.

Lots of work has been done on a different approach of connecting sheaf theory with the theory of persistence. One can view persistence modules as sheaves on the underlying poset and hence investigate persistence modules by the means of sheaf theory. This approach is discussed in \cite{berkouk,bubenik,curry,kashiwara,kim}.     

\subsection{Outline}

\textbf{Main part:}
Sections \ref{402}, \ref{260} and \ref{182} provide some preliminaries of sheaves, sheaf cohomology and persistence, respectively. In Section \ref{551} we introduce sheaf persistence modules of algebraic type. In Section \ref{608} we show a correspondence between linear diagrams of sheaves and sheaves of graded modules and a sheaf version of the Z-C representation theorem. In Section \ref{819} we show that we can relate the persistent cohomology of linear diagrams of sheaves and the cohomology of the corresponding sheaf of graded modules. In Section \ref{769} we introduce sheaf copersistence modules of topological type. In Section \ref{959} we discuss how they relate to sheaf persistence modules of algebraic type. In Section \ref{911} we show how to combine both constructions. In Section \ref{334} and \ref{741} we give examples of the first and second construction of sheaf (co)persistence modules, respectively. \newline

\noindent
\textbf{Appendix:}
Throughout this work we will use cellular sheaves of vector spaces on simplicial complexes as a running example. Section \ref{684} provides a brief introduction to cellular sheaves on simplicial complexes. For the reader not familiar with derived functors and \v{C}ech cohomology, Sections \ref{658}, \ref{710}, \ref{967} and \ref{809} provide a short accessible discussion of these constructions. In Section \ref{233} we use \v{C}ech cohomology to derive explicit cochain complexes whose cohomology compute the cohomology of cellular sheaves on simplicial complexes. Moreover, in Section \ref{342} we derive explicit cochain morphisms whose induced morphisms in cohomology compute the morphisms induced by simplicial maps.  \newline

\noindent
\textbf{Notation:}
Throughout this work we assume that $R$ is a unital commutative ring and the category $\mathbf{M}$ is either the category of $R$-modules $\mathbf{Mod}_R$ or the category of $\mathbb{N}_0$-graded $R$-modules $\mathbf{grMod}_R$. For a field $\mathbb{F}$, we denote by $\mathbf{Vec}_\mathbb{F}$ the category of $\mathbb{F}$-vector spaces. By $\mathbf{vec}_\mathbb{F}$, $\mathbf{mod}_R$ or $\mathbf{grmod}_R$ we denote the categories of finitely generated (graded) modules. For a general reference on category theory see \cite{maclane}. Given $(f_i\colon X\rightarrow Y_i)_{i\in I}$, we denote by $\prod_{i\in I}f_i\colon X\rightarrow \prod_{i\in I}Y_i$ the map defined by $x\in X\mapsto \prod_{i\in I}f_i(x)$ and by $p_j\colon\prod_{i\in I}Y_i\rightarrow Y_j$ the projection from the product to the $j$-th factor. We use similar notations for direct sums.

\section{Background} \label{217}

In this section, we discuss basic notions of sheaves, sheaf cohomology and persistence. For more information on sheaves and sheaf cohomology see \cite{bredon,cohomology,iversen}. For a survey on persistence theory see~\cite{oudot}. Readers familiar with these concepts can safely skip this section.

\subsection{Sheaves} \label{402}

A sheaf can be viewed as a structure that gathers algebraic information on the open sets of a topological space in such a way that consistent local information can be lifted to global information. 

\begin{definition}[Sheaf] \label{430}
A \emph{presheaf} $\mathcal{F}$ on a topological space $X$ with values in the category $\mathbf{M}$ is a functor  $\mathcal{F}\colon\mathbf{Open}(X)^{op}\rightarrow \mathbf{M}$ from the opposite category of open subsets of $X$ to $\mathbf{M}$. A presheaf $\mathcal{F}$ is a \emph{sheaf} if for every open subset $U\subseteq X$ and every open cover $\mathcal{U}=(U_i)_{i\in I}$ of $U$ the following sequence is exact
\begin{equation*} \label{951}
\begin{tikzcd}
0 \arrow[r] & \mathcal{F}(U) \arrow[r,"\gamma"] & \underset{i\in I}{\prod} \mathcal{F}(U_i) \arrow[r,"\alpha-\beta"] & \underset{i,j\in I}{\prod} \mathcal{F}(U_i\cap U_j) 
\end{tikzcd}
\end{equation*}
where the morphisms $\alpha$, $\beta$ and $\gamma$ are defined in the following way:
\begin{align*} \label{637}
& \alpha\coloneqq \prod_{i,j\in I} \mathcal{F}(U_i\cap U_j\xhookrightarrow{} U_i)\circ p_i \\
& \beta\coloneqq \prod_{i,j\in I} \mathcal{F}(U_i\cap U_j\xhookrightarrow{} U_j)\circ p_j \\
& \gamma\coloneqq \prod_{i\in I}\mathcal{F}(U_i\xhookrightarrow{} U) \quad .
\end{align*} 
\end{definition}

\noindent
The structure preserving maps between sheaves are called sheaf morphisms.

\begin{definition}[Sheaf morphism] \label{944} 
Let $\mathcal{F}$ and $\mathcal{G}$ be (pre)sheaves with values in $\mathbf{M}$ on a topological space $X$. A morphism of (pre)sheaves $\phi\colon\mathcal{F}\rightarrow \mathcal{G}$ is a natural transformation between the functors $\mathcal{F}$ and $\mathcal{G}$, i.e.\ a family $\phi=(\phi_U)_{U\in \mathbf{Open}(X)}$ of morphisms $\phi_U\colon\mathcal{F}(U)\rightarrow \mathcal{G}(U)$ in the category $\mathbf{M}$ such that for every pair of open sets $V\subseteq U \in \mathbf{Open}(X)$ the following diagram commutes
\begin{equation*} 
\begin{tikzcd}
\mathcal{F}(U) \arrow[r,"\phi_U"] \arrow[d,swap,"\mathcal{F}(V\xhookrightarrow{} U)"] & \mathcal{G}(U) \arrow[d,"\mathcal{G}(V\xhookrightarrow{} U)"] \\
\mathcal{F}(V) \arrow[r,"\phi_V"] & \mathcal{G}(V)
\end{tikzcd} \quad .
\end{equation*}
\end{definition}

\noindent
(Pre)sheaves and (pre)sheaf morphisms form categories which we will denote by $\mathbf{pShv}(X,\mathbf{M})$ and $\mathbf{Shv}(X,\mathbf{M})$, respectively. Since the categories $\mathbf{Mod}_R$ and $\mathbf{grMod}_R$ are abelian categories \cite[Page 20]{gradedrings}, the categories $\mathbf{pShv}(X,\mathbf{M})$ and $\mathbf{Shv}(X,\mathbf{M})$ are also abelian categories \cite[Page 25/26]{weibel}. There is also the notion of a cosheaf, dual to the notion of a sheaf \cite[Page 281]{bredon}. The definitions of cosheaves and cosheaf morphisms are dual to the definitions of sheaves and a sheaf morphisms. We denote by $\mathbf{coShv}(X,\mathbf{M})$ the category of cosheaves with values in $\mathbf{M}$ on a topological space $X$.

\begin{example}[Sheaf of continuous functions] \label{700} 
The prototypical example of a sheaf on a topological space $X$ is the \emph{sheaf of continuous functions} $\mathcal{O}_X\colon\mathbf{Open}(X)^{op}\rightarrow \mathbf{Vec}_\mathbb{R}$ which is defined by $\mathcal{O}_X(U)\coloneqq\{f\colon U\rightarrow \mathbb{R}\text{ continuous}\}$ and $\mathcal{O}_X(U\xhookrightarrow{} V)(f)\coloneqq f|_U$ for all $U\subseteq V\subseteq X$ open. In this case, the sheaf property can be interpreted in the following way: Given an open cover $\mathcal{U}=(U_i)_{i\in I}$ of $X$ and a family of continuous functions $(f_i\colon U_i\rightarrow \mathbb{R})_{i\in I}$ such that $f_i|_{U_i\cap U_j}=f_j|_{U_i\cap U_j}$, then we can glue the maps $(f_i)_{i\in I}$ together to obtain a unique continuous map $f\colon X\rightarrow \mathbb{R}$ such that $f|_{U_i}=f_i$. The general sheaf condition is an abstraction of this property.
\end{example}

\noindent
From a computational perspective, general topological spaces and sheaves like the sheaf of continuous functions are a little bit unhandy as they might involve infinitely many open sets and infinite-dimensional vector spaces. The so-called cellular sheaves on finite cell complexes \cite[Definition 4.1.6]{curry} are a kind of sheaves well-suited for computations. We consider the following special case of cellular sheaves of finite-dimensional vector spaces on finite simplicial complexes as a running example. More details on the following example can be found in Section \ref{684}.

\begin{example}[Cellular sheaves on simplicial complexes] \label{558}
Let $X$ be a finite abstract simplicial complex. On the one hand, $X$ can be viewed as a finite topological space equipped with the so-called Alexandrov topology (Section \ref{913}) and a sheaf on $X$ is a sheaf by Definition \ref{430}. On the other hand, $X$ can be viewed as the category $\mathbf{X}$ corresponding to the poset of its face relations. One can show that a sheaf on $X$ corresponds to a functor $F\colon\mathbf{X}\rightarrow \mathbf{M}$ \cite[Theorem 4.2.10]{curry}, i.e.\ a cellular sheaf on a simplicial complex can be identified with an assignment of an object $F(\sigma)$ to each simplex $\sigma\in X$ and an assignment of a morphism $F(\sigma\rightarrow \tau)$ to each pair of incident simplices $\sigma\leq\tau$, subject to commutativity requirements. Moreover, a sheaf morphism $\phi\colon F\rightarrow G$ between cellular sheaves on $X$ is a natural transformation between the functors $F$ and $G$, i.e.\ an assignment of a morphism $\phi_\sigma\colon F(\sigma)\rightarrow G(\sigma)$ to each $\sigma\in X$ satisfying the conditions of a natural transformation. Figure \ref{706} shows an example of cellular sheaves and a sheaf morphism on a $1$-simplex.
\begin{figure}[h] 
\centering
\begin{tikzcd}[every label/.append style = {font = \tiny}]
\bullet \arrow[dd,dash,thick,shorten <= -.5em,shorten >= -.5em] & & \mathbb{F}^2 \arrow[d,swap,"\begin{pmatrix} 1 \quad 0 \\ 1 \quad 0 \\ 0 \quad 1 \end{pmatrix}"] \arrow[rrr,red,"\begin{pmatrix} 1 \quad 0 \\ 0 \quad 1 \\ 0 \quad 0 \end{pmatrix}"] & & & \mathbb{F}^3 \arrow[d,"\begin{pmatrix} 1 \quad 0 \quad 0 \\ 1 \quad 0 \quad 1 \\ 0 \quad 1 \quad 0 \end{pmatrix}"] \\ 
& & \mathbb{F}^3 \arrow[rrr,red,"\begin{pmatrix} 1 \quad 0 \quad 0 \\ 0 \quad 1 \quad 0 \\ 0 \quad 0 \quad 1\end{pmatrix}"] & & & \mathbb{F}^3 \\
\bullet & & \mathbb{F} \arrow[u,"\begin{pmatrix} 0 \\ 0 \\ 1\end{pmatrix}"] \arrow[rrr,red,"\begin{pmatrix} 1 \\ 0  \end{pmatrix}"] & & & \mathbb{F}^2 \arrow[u,swap,"\begin{pmatrix} 0 \quad 0 \\ 0 \quad 1 \\ 1 \quad 0 \end{pmatrix}"] \\[-15 pt]
X & & F \arrow[rrr,"\phi"{font = \small},red] & & & G
\end{tikzcd}
\caption{Cellular sheaves and a sheaf morphism on a $1$-simplex.}
\label{706}
\end{figure}
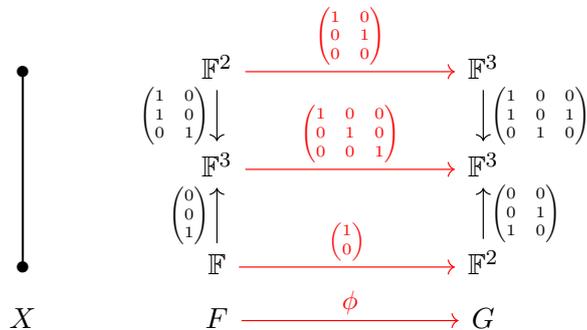
\end{example}

\noindent
Sheaf morphisms $\phi\colon \mathcal{F}\rightarrow \mathcal{G}$ connect different sheaves on a fixed topological space $X$ in the category $\mathbf{Shv}(X,\mathbf{M})$. Given a continuous map $f\colon X\rightarrow Y$ between topological spaces $X$ and $Y$, there are also notions of ``maps" between the space of sheaves on $X$ and the space of sheaves on $Y$ describing how one can translate a sheaf on $X$ into a sheaf on $Y$ or vice versa. 

\begin{definition}[Direct image functor] \label{434}
Let $X$ and $Y$ be topological spaces and $f\colon X\rightarrow Y$ a continuous map. Then there is a functor $f_*\colon \mathbf{Shv}(X,\mathbf{M})\rightarrow \mathbf{Shv}(Y,\mathbf{M})$ called the \emph{direct image functor} with respect to $f$. For $\phi\colon\mathcal{F}\rightarrow \mathcal{G}$ a morphism in $\mathbf{Shv}(X,\mathbf{M})$, the sheaf $f_*\mathcal{F}$ is defined by
\begin{equation*} 
f_*\mathcal{F}(U)\coloneqq \mathcal{F}(f^{-1}(U)) \quad \forall U\in\mathbf{Open}(Y) 
\end{equation*}
and the sheaf morphism $f_*\phi\colon f_* \mathcal{F}\rightarrow f_*\mathcal{G}$ is defined by 
\begin{equation*} 
(f_*\phi)_U\coloneqq \phi_{f^{-1}(U)} \quad \forall U\in\mathbf{Open}(Y)  \quad .
\end{equation*}
\end{definition}

\noindent 
To define the inverse image functors we need a process called sheafification that turns a presheaf into a sheaf \cite[page 85/96]{iversen}. 

\begin{definition}[Inverse image functor] \label{894}
Let $X$ and $Y$ be topological spaces and $f\colon X\rightarrow Y$ a continuous map. Then there is a functor $f^{*}\colon \mathbf{Shv}(Y,\mathbf{M})\rightarrow \mathbf{Shv}(X,\mathbf{M})$ called the \emph{inverse image functor} with respect to $f$. For $\phi\colon\mathcal{F}\rightarrow \mathcal{G}$ a morphism in $\mathbf{Shv}(X,\mathbf{M})$, the sheaf $f^*\mathcal{F}$ is the sheafification of the presheaf defined by 
\begin{equation*} 
U\mapsto \underset{{f(U)\subseteq V}}{\text{colim}} \mathcal{F}(V) \quad \forall U\in \mathbf{Open}(X) 
\end{equation*}
and $f^*\phi\colon f^*\mathcal{F}\rightarrow f^*\mathcal{G}$ is the morphism induced by $\phi$ on colimits.
\end{definition}

\begin{example} \label{575}
Let $\iota\colon X\xhookrightarrow{} Y$ be an inclusion of simplicial complexes. Then $\iota$ is continuous if $X$ and $Y$ are viewed as topological spaces with the Alexandrov topology (Section \ref{684}) and $\iota_*$ and $\iota^*$ act in the following way on cellular sheaves $F$ on $X$ and $G$ on $Y$ \cite[Definition 5.1.3, 5.1.8]{curry}
\begin{equation*} 
\begin{aligned}
& \iota_*F(\tau)=\begin{cases} F(\tau) &\text{ if } \tau\in X \\ 0 &\text{ else} \end{cases} \quad \forall\tau\in Y  \\[10pt]
& \iota^*G(\sigma)=G(\iota(\sigma))=G(\sigma) \qquad \forall\sigma\in X
\end{aligned} \hspace{.5cm} .
\end{equation*} 
Figure \ref{398} shows the action of the direct and inverse image functors with respect to the inclusion of an edge into a triangle.
\begin{figure}[h]
\centering 
\begin{tikzcd}[every label/.append style = {font = \tiny}]
\bullet & & \bullet \arrow[ddrr,dash,thick,shorten <= -.7em,shorten >= -.7em] & & & & \mathbb{F} \arrow[d,swap,"\begin{pmatrix} 0 \\ 1  \end{pmatrix}"] & & \mathbb{F} \arrow[d,swap,"\begin{pmatrix} 0 \\ 1  \end{pmatrix}"] \arrow[dr,"0"] \\
 & & & & & & \mathbb{F}^2 & & \mathbb{F}^2 & 0 \\
\bullet \arrow[uu,dash,thick,shorten <= -.5em,shorten >= -.5em] & & \bullet \arrow[uu,dash,thick,shorten <= -.5em,shorten >= -.5em] \arrow[rr,dash,thick,shorten <= -.5em,shorten >= -.5em] & & \bullet & & \mathbb{F} \arrow[u,"\begin{pmatrix} 1 \\ 0  \end{pmatrix}"] & & \mathbb{F} \arrow[r,"0"] \arrow[u,"\begin{pmatrix} 1 \\ 0  \end{pmatrix}"] & 0 & 0 \arrow[l,swap,"0"] \arrow[ul,swap,"0"] \\[-12pt]
X \arrow[rrr,hook,"\iota"{font = \small}] & & & Y & & & F \arrow[rrr,"\iota_*"{font = \small},shorten <= 1em,shorten >= 1em,maps to] & &  & \iota_*F \\[-20pt]
 & & & & & & \iota^*G & & & \arrow[lll,swap,"\iota^*"{font = \small},shorten <= 1em,shorten >= 1em, maps to] G
\end{tikzcd}
\caption{Pushforward and pullback of cellular sheaves along an inclusion of simplicial complexes.}
\label{398}
\end{figure}
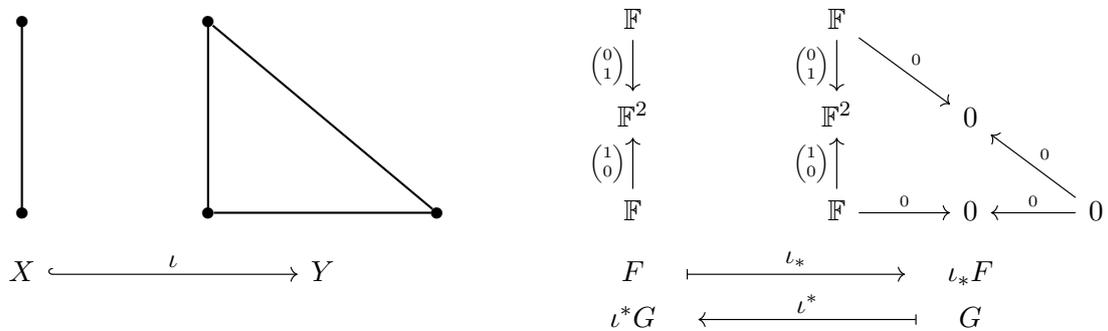
\end{example}

\noindent
The direct and inverse image functors are also called pushforward and pullback. We will need the following properties of these functors for the construction of morphisms induced by continuous maps in cohomology (see Section \ref{710} and \ref{809}). 

\begin{theorem} \label{573} 
Let $X$ and $Y$ be topological spaces and $f\colon X\rightarrow Y$ a continuous map. 
\begin{enumerate}
\item \cite[Page 97]{iversen} The functor $f_*\colon \mathbf{Shv}(X,\mathbf{M})\rightarrow \mathbf{Shv}(Y,\mathbf{M})$ is additive and left exact.
\item \cite[Page 97]{iversen} The functor $f^{*}\colon \mathbf{Shv}(Y,\mathbf{M})\rightarrow \mathbf{Shv}(X,\mathbf{M})$ is additive and exact.
\item \cite[Theorem 4.8]{iversen} The functor $f^*$ is left adjoint to the functor $f_*$. 
\end{enumerate}
\end{theorem}

\subsection{Sheaf cohomology} \label{260}

Sheaf cohomology gives us some information about the internal structure of a sheaf. In general, sheaf cohomology is defined as the right derived functors \cite[Chapter I.7]{iversen} of the so-called global section functor.

\begin{definition}[Section functor] \label{661}
Let $U\subseteq X$ be an open subset of a topological space $X$. Then $\Gamma(U,-)\colon \mathbf{Shv}(X,\mathbf{M})\rightarrow \mathbf{M}$ is the functor defined by $\mathcal{F}\mapsto \mathcal{F}(U)$ for all $\mathcal{F}\in \mathbf{Shv}(X,\mathbf{M})$ and $(\phi\colon\mathcal{F}\rightarrow \mathcal{G}) \mapsto \big(\phi_U\colon\mathcal{F}(U)\rightarrow\mathcal{G}(U)\big)$ for all $\phi\in \text{Hom}(\mathcal{F},\mathcal{G})$. If $U=X$, then $\Gamma(X,-)$ is called the \emph{global section functor}.
\end{definition} 

\noindent
A construction of sheaf cohomology can be found in Section \ref{658}. This is needed for the construction of morphisms induced by continuous maps and for some proofs. For the rest of this work, one can consider sheaf cohomology as a family of black box functors.  

\begin{definition}[Sheaf cohomology] \label{613}
Let $X$ be a topological space. The \emph{$k$-th sheaf cohomology functor} $H^k(X,-)\colon \mathbf{Shv(X,\mathbf{M})}\rightarrow \mathbf{M}$ is defined as the $k$-th right derived functor (see Section \ref{656}) of the global section functor $\Gamma(X,-)$, i.e.\ $H^k(X,-)\coloneqq R^k\Gamma(X,-)$ for all $k\in\mathbb{N}_0$. 
\end{definition}  

\noindent
If $k=0$, we obtain $H^0(X,-)\cong \Gamma(X,-)$, i.e.\ $H^0(X,\mathcal{F})\cong \mathcal{F}(X)$ for all $\mathcal{F}\in\mathbf{Shv}(X,\mathbf{M})$. Hence sheaf cohomology in dimension zero yields the space of global sections of a sheaf. For $k>0$ sheaf cohomology is more difficult to interpret. A common proverb is that sheaf cohomology describes the obstructions to solve a geometric problem globally when it can be solved locally. The general definition of sheaf cohomology is very abstract but there is the alternative somewhat more concrete construction of \v{C}ech cohomology, described in Section \ref{967}, that agrees with sheaf cohomology in many situations. In the case of a cellular sheaf on a simplicial complex, we use \v{C}ech cohomology to derive an explicit cochain complex such that the cohomology of this complex computes sheaf cohomology.

\begin{example} \label{374}
Let $X$ be a finite abstract simplicial complex and $F\colon \mathbf{X}\rightarrow \mathbf{M}$ a cellular sheaf on $X$. Let $X^k$ denote the set of $k$-simplices of $X$. Then the sheaf cohomology of $F$ is given by the cohomology of the cochain complex $(C^\bullet(X,F),\delta^\bullet)$ defined by 
\begin{equation} \label{604}
\begin{aligned}
& C^k(X,F)\coloneqq \bigoplus_{\sigma\in X^k} F(\sigma) \\
& \delta^k\coloneqq\bigoplus_{\tau\in X^{k+1}}\big(\sum_{\tau \geq \sigma \in X^k} [\sigma:\tau] F(\sigma\rightarrow \tau)\circ p_\sigma\big)
\end{aligned}
\end{equation}
where $[\sigma:\tau]$ is an orientation coefficient on the simplices (as in simplicial (co)homology). The cohomology of $(C^\bullet(X,F),\delta^\bullet)$ is usually called the cellular sheaf cohomology of $F$ \cite[Definition 6.2.1]{curry}. In Section \ref{233} we show how to derive (\ref{604}) from \v{C}ech cohomology, and that for cellular sheaves on simplicial complexes sheaf and \v{C}ech cohomology agree. Hence, the cohomology of (\ref{604}) indeed computes sheaf cohomology. This is already proved in \cite{curry}. We provide this alternative proof as a basis for the derivation of morphisms induced in cohomology by simplicial maps using \v{C}ech cohomology (Section \ref{342}). Figure \ref{433} shows an example of the constant $\mathbb{F}$-valued sheaf $\mathbb{F}_X$ on a simplicial triangle where the direction of the edges denotes the orientation.
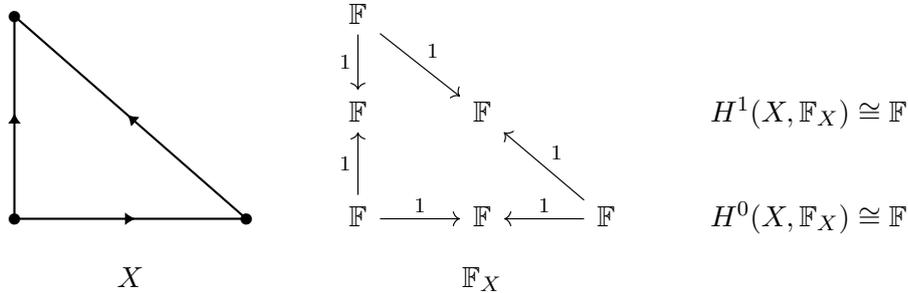
\begin{figure}[h]
\centering
\begin{tikzcd}
\bullet  & & & \mathbb{F} \arrow[d,swap,"1"] \arrow[dr,"1"] \\
& & & \mathbb{F} & \mathbb{F} & & H^1(X,\mathbb{F}_X)\cong\mathbb{F} \\
\bullet \arrow[rr,dash,thick,shorten <= -.55em,shorten >= -.55em,midarrow] \arrow[uu,dash,thick,shorten <= -.55em,shorten >= -.55em,midarrow] & & \bullet \arrow[uull,dash,thick,shorten <= -.8em,shorten >= -.8em,midarrow] & \mathbb{F} \arrow[r,"1"] \arrow[u,"1"] & \mathbb{F} & \arrow[l,swap,"1"] \mathbb{F} \arrow[ul,swap,"1"] & H^0(X,\mathbb{F}_X)\cong\mathbb{F} \\[-15pt]
& X & & & \mathbb{F}_X
\end{tikzcd}
\caption{The constant $\mathbb{F}$-valued cellular sheaf on a simplicial triangle and its sheaf cohomology.}
\label{433}
\end{figure} 
By computing the cochain complex (\ref{604}) for $\mathbb{F}_X$, one can see that it is exactly the cochain complex of simplicial cohomology. Hence, the cohohomology of this complex is the simplicial cohomology of $X$ with $\mathbb{F}$-coefficients. 
\end{example}  

\noindent
In Example \ref{374}, we saw that we can derive simplicial cohomology from sheaf cohomology. This raises the question if we can also derive the morphisms induced by a simplicial map in simplicial cohomology from sheaf theory. This question is answered by a general construction of morphisms induced by continuous maps between the sheaf cohomology modules of a sheaf and its inverse image. Let $f:X\rightarrow Y$ be a continuous map and $\mathcal{F}$ a sheaf on $Y$. Then, for every $k\in\mathbb{N}_0$, there is an induced morphism \cite[page 100]{iversen}  
\begin{equation*} 
H^k(f)\colon H^k(Y,\mathcal{F})\rightarrow H^k(X,f^*\mathcal{F}) \quad .
\end{equation*}
In Section \ref{710}, we show how to construct these induced morphisms.

\begin{example} \label{548}
Let $X$ and $Y$ be finite abstract simplicial complexes, $f\colon X\rightarrow Y$ a simplicial map and $F$ a cellular sheaf on $Y$. Then the induced morphism $H^k(f)\colon H^k(Y,F)\rightarrow H^k(X,f^*F)$ is given by $H^k(f^\bullet)$ where $f^\bullet\colon C^\bullet(Y,F)\rightarrow C^\bullet(X,f^*F)$ is a cochain morphism on (\ref{604}) such that $f^k\colon C^k(Y,F) \rightarrow C^k(X,f^*F)$ is defined by
\begin{equation*} 
f^k\coloneqq \underset{\sigma\in X}{\bigoplus}\text{ } p_{f(\sigma)}
\end{equation*}
for all $k\in\mathbb{N}_0$. Note that $C^k(X,f^*F)\cong\underset{\sigma\in X}{\bigoplus} F\big(f(\sigma)\big)$. In Section \ref{342} we show how to derive this formula using \v{C}ech cohomology. If $F$ is the constant cellular sheaf on $Y$, then $H^k(f)$ agrees with the morphism induced by the simplicial cohomology functor.  
\end{example}

\subsection{Persistence} \label{182}

The idea of persistent homology is to study the evolution of homological features along a linear diagram of topological spaces. On an abstract level, the two main principles on which the construction of persistent homology or the construction of a persistence barcode is based are the functoriality of homology and the unique decomposability of persistence modules into interval modules. Let $\mathbf{A}$ be some category and $\vec{S}\colon\mathbb{N}_0\rightarrow \mathbf{A}$ be a functor on the category corresponding to the poset $(\mathbb{N}_0,\leq)$. We use the notation $S_i\coloneqq\vec{S}(i)$ and $S_i^{i+1}\coloneqq\vec{S}(i\rightarrow i+1)$ for all $i\in\mathbb{N}_0$. Let $T\colon\mathbf{A}\rightarrow\mathbf{Mod}_R$ be a functor to the category of $R$-modules. The composition $T\circ \vec{S}\colon\mathbb{N}_0\rightarrow\mathbf{Mod}_R$ depicted in the following diagram is a functor from $\mathbb{N}_0$ to $\mathbf{Mod}_R$
\begin{equation*} 
\begin{tikzcd}
\mathbb{N}_0 \arrow[d,"\vec{S}"] & 0 \arrow[r] &[15pt] 1 \arrow[r] &[15pt] 2 \arrow[r] &[15pt] \cdots \\
\mathbf{A} \arrow[d,"T"] & S_0 \arrow[r,"S_0^1"] & S_1 \arrow[r,"S_1^2"] & S_2 \arrow[r,"S_2^3"] & \cdots \\
\mathbf{Mod}_R & T(S_0) \arrow[r,"T(S_0^1)"] & T(S_1) \arrow[r,"T(S_1^2)"] & T(S_2) \arrow[r,"T(S_2^3)"] & \cdots
\end{tikzcd} \quad .
\end{equation*}

\noindent
We will call such a functor a \emph{persistence module}. Functors of the form $\vec{D}\colon\mathbb{N}_0\rightarrow\mathbf{vec}_\mathbb{F}$, i.e.\ pointwise finite-dimensional persistence modules, are composed of so-called \emph{interval modules} 
$\vec{I}_{[a,b]}\colon\mathbb{N}_0\rightarrow \mathbf{vec}_\mathbb{F}$ 
\begin{equation*} 
\begin{tikzcd}
\mathbb{N}_0 \arrow[d,"\vec{I}_{[a,b]}"] & 0 \arrow[r] & \cdots \arrow[r] & a-1 \arrow[r] & a \arrow[r] & \cdots \arrow[r] & b \arrow[r] & b+1 \arrow[r] & \cdots \\
\mathbf{vec}_\mathbb{F} & 0 \arrow[r,"0"] & \cdots \arrow[r,"0"]  & 0 \arrow[r,"0"] & \mathbb{F} \arrow[r,"\text{id}"] & \cdots \arrow[r,"\text{id}"] & \mathbb{F} \arrow[r,"0"] & 0 \arrow[r,"0"] & \cdots 
\end{tikzcd}
\end{equation*}
with $a\leq b\in\mathbb{N}_0\cup\{\infty\}$, as stated by the following theorem.
\begin{theorem} \label{817} \cite[Theorem 1.9]{oudot}
For all $\vec{D}\in\mathbf{Fun}(\mathbb{N}_0,\mathbf{vec}_\mathbb{F})$ we have $\vec{D}\cong \bigoplus_{i\in I} \vec{I}_{[a_i,b_i]}$ with a uniquely determined collection of \emph{intervals} $[a_i,b_i]$, i.e.\ $\vec{D}$ is uniquely determined up to isomorphism by the multiset of intervals $\{\{[a_i,b_i]\text{ } i\in I\}\}$. 
\end{theorem}

\begin{remark} \label{151}
We call a functor $\cev{D}\in\mathbf{Fun}(\mathbb{N}_0^\text{op},\mathbf{Mod}_R)$ a \emph{copersistence module}. If $\cev{D}\in\mathbf{Fun}(\mathbb{N}_0^\text{op},\mathbf{vec}_\mathbb{F})$, by applying the functor $\text{dual}\colon\mathbf{vec}_\mathbb{F}^\text{op}\xrightarrow{\cong}\mathbf{vec}_\mathbb{F}$ that sends a vector space to its dual, we obtain a functor $\text{dual}\circ \cev{D}\in\mathbf{Fun}(\mathbb{N}_0,\mathbf{vec}_\mathbb{F})$. By Theorem \ref{817}, we have $\text{dual}\circ \cev{D}\cong \bigoplus_{i\in I} \vec{I}_{[a_i,b_i]}$, thus $\cev{D}\cong \bigoplus_{i\in I} \text{dual}\circ \vec{I}_{[a_i,b_i]}\cong\bigoplus_{i\in I} \cev{I}_{[a_i,b_i]}$, where $\cev{I}_{[a_i,b_i]}\in\mathbf{Fun}(\mathbb{N}_0^\text{op},\mathbf{vec}_\mathbb{F})$ is an interval module with inverted arrow orientations. Therefore, every copersistence module $\cev{D}\in\mathbf{Fun}(\mathbb{N}_0^\text{op},\mathbf{vec}_\mathbb{F})$ has a decomposition into dual interval modules.
\end{remark}

\begin{definition}[Persistence barcode] \label{611}
Let $\vec{D}\in\mathbf{Fun}(\mathbb{N}_0,\mathbf{vec}_\mathbb{F})$ be a persistence module with interval decomposition $\vec{D}\cong \bigoplus_{i\in I} \vec{I}_{[a_i,b_i]}$. Then the \emph{persistence barcode} of $\vec{D}$ is defined as the  corresponding multiset of intervals
\begin{equation*} \label{601}
\text{BC}(\vec{D})\coloneqq \{\{[a_i,b_i]\text{ }|\text{ }i\in I\}\} \quad .
\end{equation*}
Analogously we define the barcode of a copersistence module.
\end{definition}

\noindent
Now assume that $\mathbf{A}=\mathbf{Top}$, $\vec{S}=\vec{X}\colon\mathbb{N}_0\rightarrow\mathbf{Top}$ is a linear diagram of topological spaces and continuous maps, $T=H^{\text{sing}}_k(-,\mathbb{F})$ is the singular homology functor and $H^{\text{sing}}_k(X_i,\mathbb{F})$ is finite-dimensional for all $i\in\mathbb{N}_0$. Then, by Theorem \ref{817}, we have $H^{\text{sing}}_k(-,\mathbb{F})\circ \vec{X}\cong \bigoplus_{i\in I} \vec{I}_{[a_i,b_i]}$ and the intervals in this decomposition correspond to homology classes in the image of $\vec{X}$ that are born at index $a_i$ and die at index $b_i$. One can define the $k$-dimensional \emph{persistent homology} of $\vec{X}$ as the collection of interval modules corresponding to $\text{BC}(H^{\text{sing}}_k(-,\mathbb{F})\circ \vec{X})$.

In a computational context, as described in the introduction, one considers simplicial complexes and simplicial homology instead of topological spaces and singular homology. We obtain this theory of \emph{simplicial persistent homology} for $\mathbf{A}=\mathbf{simp}$ the category of finite simplicial complexes, $\vec{S}=\vec{K}\colon\mathbb{N}_0\rightarrow\mathbf{simp}$ and $T=H^{\text{simp}}_k(-,\mathbb{F})$ the simplicial homology functor. 

The following Theorem states that persistence modules correspond to graded modules over polynomial rings. The second part of the theorem is sometimes called the Zomorodian-Carlsson representation theorem.  

\begin{definition} \label{813}
Define $\eta\colon\mathbf{Fun}(\mathbb{N}_0,\mathbf{Mod}_R)\rightarrow \mathbf{grMod}_{R[t]}$ to be the functor that sends a persistence module $\vec{D}\colon\mathbb{N}_0\rightarrow \mathbf{Mod}_R$ to the graded module $\eta(\vec{D})\coloneqq \bigoplus_{n\in\mathbb{N}_0}D_n$ with $t\cdot \coloneqq D_n^{n+1}\colon D_n\rightarrow D_{n+1}$, with the obvious action on morphisms. Define $\epsilon\colon\mathbf{grMod}_{R[t]}\rightarrow \mathbf{Fun}(\mathbb{N}_0,\mathbf{Mod}_R)$ to be the functor that sends a graded $R[t]$-module $M=\bigoplus_{n\in\mathbb{N}_0}M_n$ to the persistence module $\epsilon(M)\colon\mathbb{N}_0\rightarrow \mathbf{Mod}_R$ such that $\epsilon(M)_n\coloneqq M_n$ and $\epsilon(M)_n^{n+1}\coloneqq t \cdot\colon M_n\rightarrow M_{n+1}$, with the obvious action on morphisms.
\end{definition}

\begin{theorem} \label{960} \cite{corbet,carlsson}
There are natural isomorphisms $\epsilon\circ\eta\cong \text{id}$ and $\eta\circ\epsilon\cong\text{id}$, i.e.\ the categories $\mathbf{Fun}(\mathbb{N}_0,\mathbf{Mod}_R)$ and $\mathbf{grMod}_{R[t]}$ are equivalent.

Moreover, the functors $\eta$ and $\epsilon$ restrict to equivalences between the subcategories of persistence modules of finite type\footnote{A persistence module $\vec{D}\colon\mathbb{N}_0\rightarrow \mathbf{Mod}_R$ is of finite type if $D_i$ is finitely generated for all $i\in\mathbb{N}_0$ and there exists an $m\in\mathbb{N}_0$ such that $D_i^{i+1}$ is an isomorphism for all $i\geq m$.} and finitely generated graded $R[t]$-modules.
\end{theorem}

\noindent
The following proposition states that the functors $\eta$ and $\epsilon$ are in some sense compatible with the cohomology functor. Although this is a well-known result, for the sake of completeness we provide a proof in the appendix. Given a functor $F\colon\mathbf{A}\rightarrow \mathbf{B}$, we denote by $\underline{F}\colon \mathbf{Fun}(\mathbf{I},\mathbf{A})\rightarrow \mathbf{Fun}(\mathbf{I},\mathbf{B})$ the functor defined by post-composition with $F$, i.e.\ $\underline{F}(-)\coloneqq F\circ -$. This can also be viewed as componentwise application of $F$.

\begin{proposition} \label{635}
The following diagram commutes 
\begin{equation*} \label{102}
\begin{tikzcd}
\mathbf{coCh}(\mathbf{grMod}_{R[t]}) \arrow[r,shift left=1,"\underline{\epsilon}"] \arrow[d,swap,"H^k"] &[10pt] \arrow[l,shift left=1,"\underline{\eta}"] \mathbf{coCh}\big(\mathbf{Fun}(\mathbb{N}_0,\mathbf{Mod}_R)\big) \arrow[r,shift left=1,"cr^{-1}"{xshift=3pt}] & \arrow[l,shift left=1,"cr"] \mathbf{Fun}\big(\mathbb{N}_0,\mathbf{coCh}(\mathbf{Mod}_R)\big) \arrow[d,"\underline{H}^k"]  \\
\mathbf{grMod}_{R[t]} \arrow[rr,shift left=1,"\epsilon"] && \arrow[ll,shift left=1,"\eta"] \mathbf{Fun}(\mathbb{N}_0,\mathbf{Mod}_R)
\end{tikzcd} 
\end{equation*}
i.e.\ there are natural isomorphisms $\eta\circ \underline{H}^k\cong H^k\circ \underline{\eta}\circ cr$ and $\epsilon\circ H^k\cong \underline{H}^k\circ cr^{-1}\circ\underline{\epsilon}$ where $cr$ and $cr^{-1}$ are obtained by currying \footnote{They take a cochain complex of linear diagrams to a linear diagram of cochain complexes and vice versa.} \cite{nlab:currying}. 
\end{proposition}

\begin{proof}
See Section \ref{702}.
\end{proof}

\section{Persistent sheaf cohomology type A} \label{129}

\subsection{Sheaf persistence modules of algebraic type} \label{551}

We apply the concept of persistence, discussed in the previous section, to the category of sheaves and the sheaf cohomology functor. In the notation of the previous section, we set $\mathbf{A}=\mathbf{Shv}(X,\mathbf{Mod}_R)$, $\vec{S}=\vec{\mathcal{F}}\colon\mathbb{N}_0\rightarrow \mathbf{Shv}(X,\mathbf{Mod}_R)$ and $T=H^k(X,-)\colon\mathbf{Shv}(X,\mathbf{Mod}_R)\rightarrow \mathbf{Mod}_R$ the $k$-th sheaf cohomology functor. In this way, we obtain the persistence module $H^k(X,-)\circ \vec{\mathcal{F}}\in\mathbf{Fun}(\mathbb{N}_0,\mathbf{Mod}_R)$ depicted in the following diagram
\begin{equation} \label{681} 
\begin{tikzcd}
\mathbb{N}_0 \arrow[d,"\vec{\mathcal{F}}"{yshift=3pt}] &[-20pt] 0 \arrow[r] &[20pt] 1 \arrow[r] &[20pt] 2 \arrow[r] &[20pt] \cdots \\
\mathbf{Shv}(X,\mathbf{Mod}_R) \arrow[d,"H^k(X\text{,}-)"] & \mathcal{F}_0 \arrow[r,"\mathcal{F}_0^1"] & \mathcal{F}_1 \arrow[r,"\mathcal{F}_1^2"] & \mathcal{F}_{2} \arrow[r,"\mathcal{F}_2^3"] & \cdots \\
\mathbf{Mod}_R & H^k(X,\mathcal{F}_0) \arrow[r,"H^k(X\text{,}\mathcal{F}_0^1)"] & H^k(X,\mathcal{F}_1) \arrow[r,"H^k(X\text{,}\mathcal{F}_1^2)"] & H^k(X,\mathcal{F}_2) \arrow[r,"H^k(X\text{,}\mathcal{F}_2^3)"] & \cdots
\end{tikzcd} \quad .
\end{equation}  

\noindent
In the linear diagram of sheaves $\vec{\mathcal{F}}$, depicted in the middle row of (\ref{681}), we have variable algebraic information attached to the fixed topological space $X$. Therefore, we call $H^k(X,-)\circ \vec{\mathcal{F}}$ a \emph{sheaf persistence module of algebraic type} or, for brevity, a \emph{persistence module of type A}.

It is straightforward to generalize this construction of persistence modules from diagrams of sheaves to the zigzag (arbitrary orientations of arrows in the indexing category) or the multi-dimensional case. Given a functor $\vec{\mathcal{F}}\colon\mathbb{N}_0^d\rightarrow\mathbf{Shv}(X,\mathbf{Mod}_R)$, composition with the sheaf cohomology functor yields a $d$-dimensional persistence module $H^k(X,-)\circ \vec{\mathcal{F}}\in\mathbf{Fun}(\mathbb{N}_0^d,\mathbf{Mod}_R)$. For the sake of simplicity of presentation we just consider the oriented one-dimensional case here.

If $R=\mathbb{F}$ and if we assume that $H^k(X,-)\circ \vec{\mathcal{F}}\in\mathbf{Fun}(\mathbb{N}_0,\mathbf{vec}_\mathbb{F})$, then, by Theorem \ref{817}, there exists an interval decomposition $H^k(X,-)\circ \vec{\mathcal{F}}\cong\bigoplus_{i\in I} \vec{I}_{[a_i,b_i]}$. An interval module $\vec{I}_{[a_i,b_i]}$ in this decomposition corresponds to a sheaf cohomology class in $H^k(X,-)\circ \vec{\mathcal{F}}$ that is born at index $a_i$ and dies at index $b_i$. The corresponding persistence barcode $\text{BC}\big(H^k(X,-)\circ \vec{\mathcal{F}}\big)$ describes the evolution of all $k$-dimensional sheaf cohomology classes along $\vec{\mathcal{F}}$. In analogy to persistent homology, we call the collection of interval modules corresponding to a sheaf persistence module of algebraic type, represented by $\text{BC}\big(H^k(X,-)\circ \vec{\mathcal{F}}\big)$, the \emph{(type A) persistent sheaf cohomology} of $\vec{\mathcal{F}}$.

For example, in the case $k=0$, the intervals in $\text{BC}\big(H^0(X,-)\circ \vec{\mathcal{F}}\big)$ correspond to generators of global sections in $\vec{\mathcal{F}}$ that are born at some index $a_i$ and persist until they die at index $b_i$. Figure \ref{257} shows an example of a functor $\vec{F}$ (of cellular sheaves) on a $1$-simplex and the corresponding zero-dimensional persistence barcode. Note that we can identify a functor $\vec{F}\in\mathbf{Fun}\big(\mathbf{[n]},\mathbf{Shv}(X,\mathbf{vec}_\mathbb{F})\big)$ on the category corresponding to the poset $[n]=\{0,\ldots,n-1\}$, as depicted in Figure \ref{257}, with a functor $\vec{F}'\in\mathbf{Fun}\big(\mathbb{N}_0,\mathbf{Shv}(X,\mathbf{vec}_\mathbb{F})\big)$ by extending $\vec{F}$ on the right by infinitely many copies of $F_4$ and identity maps. 
\begin{figure} 
\begin{tikzcd} [every label/.append style = {font = \tiny}]
\bullet_{\sigma_1} &[-35pt] \mathbb{F} \arrow[r,"\begin{pmatrix}1\\0\end{pmatrix}"{xshift=5pt}] \arrow[d,"\begin{pmatrix}1\\0\end{pmatrix}"] & \mathbb{F}^2 \arrow[r,"\begin{pmatrix}1\quad 0\\0\quad 1\end{pmatrix}"{xshift=5pt}] \arrow[d,"\begin{pmatrix}1\quad 0\\0\quad 1\end{pmatrix}"] & \mathbb{F}^2 \arrow[r,"\begin{pmatrix}1\quad 0\\0\quad 1\end{pmatrix}"{xshift=5pt}] \arrow[d,"\begin{pmatrix}1\quad 0\\0\quad 1\end{pmatrix}"] & \mathbb{F}^2 \arrow[r,"\begin{pmatrix}1\quad 0\\0\quad 1\end{pmatrix}"{xshift=5pt}] \arrow[d,"\begin{pmatrix}1\quad 0\\0\quad 1\\0\quad 0\end{pmatrix}"] & \mathbb{F}^2 \arrow[d,"\begin{pmatrix}1\quad 0\\0\quad 1\\0\quad 0\end{pmatrix}"] \\
 &[-35pt]  \mathbb{F}^2 \arrow[r,"\begin{pmatrix}1\quad 0\\0\quad 1\end{pmatrix}"{xshift=5pt}] & \mathbb{F}^2 \arrow[r,"\begin{pmatrix}1\quad 0\\0\quad 1\end{pmatrix}"{xshift=5pt}] & \mathbb{F}^2 \arrow[r,"\begin{pmatrix}1\quad 0\\0\quad 1\\0\quad 0\end{pmatrix}"{xshift=5pt}] & \mathbb{F}^3 \arrow[r,"\begin{pmatrix}1\quad 0\quad 0\\0\quad 1\quad 0\\0\quad 0\quad 1\end{pmatrix}"{xshift=5pt}] & \mathbb{F}^3  \\
\bullet_{\sigma_2} \arrow[uu,dash,swap,"\tau"{font=\small},shorten <= -.6em,shorten >= -.7em,thick,shift left=1.7] &[-35pt]  \mathbb{F} \arrow[r,"(1)"{xshift=5pt}] \arrow[u,swap,"\begin{pmatrix}0\\1\end{pmatrix}"] & \mathbb{F} \arrow[r,"\begin{pmatrix}1\\0\end{pmatrix}"{xshift=5pt}] \arrow[u,swap,"\begin{pmatrix}0\\1\end{pmatrix}"] & \mathbb{F}^2 \arrow[r,"\begin{pmatrix}1\quad 0\\0\quad 1\end{pmatrix}"{xshift=5pt}] \arrow[u,swap,"\begin{pmatrix}0\quad 0\\1\quad 0\end{pmatrix}"] & \mathbb{F}^2 \arrow[r,"\begin{pmatrix}1\quad 0\\0\quad 1\\0\quad 0\end{pmatrix}"{xshift=5pt}] \arrow[u,swap,"\begin{pmatrix}0\quad 0\\1\quad 0\\0\quad 0\end{pmatrix}"] & \mathbb{F}^3 \arrow[u,swap,"\begin{pmatrix}0\quad 0\quad 1\\1\quad 0\quad 0\\0\quad 0\quad 0\end{pmatrix}"] \\[-15pt]
X &[-35pt]  F_0 \arrow[r,"F_0^1"] & F_1 \arrow[r,"F_1^2"] & F_2 \arrow[r,"F_2^3"] & F_3 \arrow[r,"F_3^4"] & F_4 \\
&[-35pt]  0 \arrow[r,"0"] & \mathbb{F} \arrow[r,"\begin{pmatrix}1\\0\end{pmatrix}"] & \mathbb{F}^2 \arrow[r,"\begin{pmatrix}1\quad 0\\0\quad 1\end{pmatrix}"] & \mathbb{F}^2 \arrow[r,"\begin{pmatrix}1\quad 0\\0\quad 1\\0\quad 0\end{pmatrix}"] & \mathbb{F}^3 \\[-15pt]
 &[-35pt]  & \bullet \arrow[rrr,dash,shorten <= -.55em,shorten >= -.5em,thick] & & & \bullet \\[-25pt]
 &[-35pt]  & & \bullet \arrow[rr,dash,shorten <= -.55em,shorten >= -.5em,thick] & & \bullet \\[-25pt]
&[-35pt] & & & & \bullet \\[-20pt]
 &[-35pt]  H^0(X,F_0)\arrow[r] & H^0(X,F_1) \arrow[r] &  H^0(X,F_2) \arrow[r] &  H^0(X,F_3) \arrow[r] & H^0(X,F_4) 
\end{tikzcd}
\caption{A functor $\vec{F}\in\mathbf{Fun}\big(\mathbf{[5]},\mathbf{Shv}(X,\mathbf{vec}_\mathbb{F})\big)$ (of cellular sheaves) on a $1$-simplex and the corresponding zero-dimensional persistence barcode.}
\label{257}
\end{figure}

\subsection{Correspondence of linear diagrams of sheaves and sheaves of graded modules} \label{608}

A possible but inefficient method to compute the persistent cohomology of $\vec{\mathcal{F}}\in\mathbf{Fun}\big(\mathbf{[n]},\mathbf{Shv}(X,\mathbf{vec}_\mathbb{F})\big)$ is to compute the cohomology of each sheaf $\mathcal{F}_i$ and every morphism induced in cohomology by $\mathcal{F}_i^{i+1}$ individually and then compute the decomposition of the obtained persistence module. In classical persistence theory this inefficiency is avoided by computing persistent homology from the ordinary homology of a complex of graded modules. This method of computation is based on the correspondence between persistence modules and $\mathbb{N}_0$-graded $\mathbb{F}[t]$-modules (Theorem \ref{960}) and Proposition \ref{635}. We want to establish a similar correspondence theorem for sheaves.

With this goal in mind, we begin by considering presheaves. We define functors that convert linear diagrams of presheaves into presheaves of graded modules and vice versa. We denote by 
\begin{equation*} \label{937}
\begin{tikzcd}
\mathbf{Fun}\big(\mathbb{N}_0,\mathbf{Fun}(\mathbf{Open}(X)^\text{op},\mathbf{Mod}_R)\big) \arrow[r,yshift=3pt,"cr"] & \arrow[l,yshift=-3pt,"cr^{-1}"{xshift=3pt}] \mathbf{Fun}\big(\mathbf{Open}(X)^\text{op},\mathbf{Fun}(\mathbb{N}_0,\mathbf{Mod}_R)\big)
\end{tikzcd}
\end{equation*}
the functors obtained by currying \cite{nlab:currying}, i.e.\ they convert linear diagrams of presheaves into a presheaves of linear diagrams and vice versa. This is a slight abuse of notation because we also denoted the functors in Proposition \ref{635} by $cr$, but it is clear from the context which of them are meant. 

\begin{definition}[Correspondence] \label{492}
Define functors $\mathcal{C}\coloneqq \underline{\eta}\circ cr$ and $\mathcal{E}\coloneqq cr^{-1}\circ \underline{\epsilon}$ as depicted in the following diagram 
\begin{equation*}
\begin{tikzcd}
\mathbf{Fun}\big(\mathbb{N}_0,\mathbf{pShv}(X,\mathbf{Mod}_R)\big) \arrow[r,yshift=3pt,"cr"] \arrow[dr,xshift=-15pt,yshift=-5pt,shorten >= .6em,shorten <= -1em,"\mathcal{C}"] &[10pt] \arrow[l,yshift=-3pt,"cr^{-1}"] \mathbf{pShv}\big(X,\mathbf{Fun}(\mathbb{N}_0,\mathbf{Mod}_R)\big) \arrow[d,xshift=3pt,"\underline{\eta}"] \\[10pt]
& \mathbf{pShv}\big(X,\mathbf{grMod}_{R[t]}) \arrow[u,xshift=-3pt,"\underline{\epsilon}"] \arrow[ul,yshift=-10pt,xshift=-15pt,shorten >= -1.2em,shorten <= .8em,"\mathcal{E}"] 
\end{tikzcd} \quad .
\end{equation*} 
\end{definition} 

\noindent
In the following we explain the action of these functors in more detail. Let $\vec{\mathcal{F}}\in\mathbf{Fun}\big(\mathbb{N}_0,\mathbf{pShv}(X,\mathbf{Mod}_R)\big)$, i.e.\ we have a diagram of presheaves on $X$ of the form: 
\begin{equation*} 
\begin{tikzcd} [column sep=large]
\mathcal{F}_0 \arrow[r,"\mathcal{F}_0^1"] & \mathcal{F}_1 \arrow[r,"\mathcal{F}_1^2"] & \mathcal{F}_2 \arrow[r,"\mathcal{F}_2^3"] & \cdots \quad .
\end{tikzcd}
\end{equation*} 
For every open subset $U\subseteq X$, by applying the section functor $\Gamma(U,-)$ on this diagram, we obtain a diagram of $R$-modules
\begin{equation*} 
\begin{tikzcd} [column sep=large]
\mathcal{F}_0(U) \arrow[r,"(\mathcal{F}_0^1)_U"] & \mathcal{F}_1(U) \arrow[r,"(\mathcal{F}_1^2)_U"] & \mathcal{F}_2(U) \arrow[r,"(\mathcal{F}_2^3)_U"] & \cdots \quad .
\end{tikzcd}
\end{equation*} 
Hence, we have an assignment of a persistence module $\vec{\mathcal{F}}(U)\in\mathbf{Fun}(\mathbb{N}_0,\mathbf{Mod}_R)$ to every open subset $U\subseteq X$. This assignment defines a presheaf of persistence modules $cr(\vec{\mathcal{F}})$. Moreover, by the correspondence of persistence modules and graded $R[t]$-modules, we can associate to $\vec{\mathcal{F}}(U)$ the graded $R[t]$-module
\begin{equation} \label{127}
\eta(\vec{\mathcal{F}}(U))=\bigoplus_{n\in\mathbb{N}_0}\mathcal{F}_n(U)
\end{equation}
where the $t$-multiplication is defined by $t\cdot \coloneqq (\mathcal{F}_n^{n+1})_U\colon \mathcal{F}_n(U)\rightarrow \mathcal{F}_{n+1}(U)$. The assignment $U\mapsto \eta(\vec{\mathcal{F}}(U))$, for all $U\subseteq X$ open, defines $\mathcal{C}(\vec{\mathcal{F}})\in\mathbf{pShv}(X,\mathbf{grMod}_{R[t]})$. In the case of a cellular sheaf on a simplicial complex $X$, we only need to consider the basic open sets corresponding to the simplices of $X$. Consider, for example, the functor $\vec{F}$ in Figure \ref{257}. Here we obtain the persistence modules $\vec{F}(\sigma_1)$, $\vec{F}(\tau)$ and $\vec{F}(\sigma_2)$ given by the first three rows. The presheaf of graded $\mathbb{F}[t]$-modules $\mathcal{C}(\vec{F})$, corresponding to the functor $\vec{F}$, has the following form
\begin{equation} \label{907}
\begin{tikzcd}[column sep=huge,every label/.append style = {font = \tiny}]
 & \mathbb{F}[t]_3 \arrow[d,phantom,"\oplus"] & \mathbb{F}[t]_4 \arrow[d,phantom,"\oplus"] \\[-15pt]
\mathbb{F}[t]_1 \arrow[d,phantom,"\oplus"] \arrow[r,"\begin{pmatrix}1\quad 0\\0\quad t\\0\quad 0 \end{pmatrix}",shift left=-3.5] & \mathbb{F}[t]_0 \arrow[d,phantom,"\oplus"] & \arrow[l,swap,"\begin{pmatrix}0\quad 0\quad t^4\\1\quad 0\quad 0\\0\quad 0\quad 0\end{pmatrix}",shift left=3.5] \mathbb{F}[t]_2 \arrow[d,phantom,"\oplus"] \\[-15pt]
\mathbb{F}[t]_0 & \mathbb{F}[t]_0  & \mathbb{F}[t]_0 \\[-10pt]
_{\sigma_1}\bullet \arrow[rr,dash,shorten <= -.5em,shorten >= -.5em,thick,swap,"\tau"{font=\footnotesize}] & & \bullet_{\sigma_2}
\end{tikzcd}
\end{equation}
where $\mathbb{F}[t]_n$ denotes the graded $\mathbb{F}[t]$-module that is generated by a single generator on the $n$-th level and the matrix representations of the restriction maps are with respect to these generators ordered from bottom to top. 

Conversely, suppose we have a presheaf of graded $R[t]$-modules $M\in\mathbf{pShv}(X,\mathbf{grMod}_{R[t]})$ such that
\begin{equation*}
M(U)=\bigoplus_{n\in\mathbb{N}_0}M_n(U)\text{ }, \quad M(U\xhookrightarrow{} V)=\bigoplus_{n\in\mathbb{N}_0}M_n(U\xhookrightarrow{} V) \quad \forall U\subseteq V \subseteq X \text{ open} \, .
\end{equation*}
Then, applying the functor $\epsilon$ yields the persistence module 
\begin{equation*} 
\begin{tikzcd} [column sep=large]
M_0(U) \arrow[r,"t \cdot"] & M_1(U) \arrow[r,"t \cdot"] & M_2(U) \arrow[r,"t \cdot"] & \cdots
\end{tikzcd}
\end{equation*} 
for every open subset $U\subseteq X$. Hence, the assignment $U\mapsto \epsilon(M(U))$ defines a presheaf of persistence modules $\underline{\epsilon}(M)$. Now, for every $n\in\mathbb{N}_0$, we can construct a presheaf $\mathcal{E}(M)_n\in\mathbf{pShv}(X,\mathbf{Mod}_R)$ and a presheaf morphism $\mathcal{E}(M)_n^{n+1}\in\text{Hom}(\mathcal{E}(M)_n,\mathcal{E}(M)_{n+1})$ by defining 
\begin{equation} \label{330}
\begin{aligned}
& \mathcal{E}(M)_n(U)\coloneqq M_n(U) \text{ } , \quad \mathcal{E}(M)_n(U\xhookrightarrow{}V)\coloneqq M_n(U\xhookrightarrow{}V)  \quad \forall U\subseteq V \subseteq X \text{ open}  \\
& (\mathcal{E}(M)_n^{n+1})_U\coloneqq t \cdot\colon M_n(U)\rightarrow M_{n+1}(U) \quad \qquad \qquad \qquad \qquad \forall U\subseteq X \text{ open} \quad .
\end{aligned}
\end{equation}
In this way we obtain a functor $\mathcal{E}(M)\in\mathbf{Fun}\big(\mathbb{N}_0,\mathbf{pShv}(X,\mathbf{Mod}_R)\big)$.

\begin{proposition} \label{930}
There are natural isomorphisms $\mathcal{C}\circ \mathcal{E}\cong\text{id}$ and $\mathcal{E}\circ \mathcal{C}\cong \text{id}$, i.e.\ there is an equivalence of categories
\begin{equation*} 
\mathbf{Fun}\big(\mathbb{N}_0,\mathbf{pShv}(X,\mathbf{Mod}_R)\big)\cong\mathbf{pShv}(X,\mathbf{grMod}_{R[t]}) \quad .
\end{equation*}
\end{proposition}

\begin{proof}
Since $cr$ and $cr^{-1}$ are obtained by currying \cite{nlab:currying}, we have $cr\circ cr^{-1}\cong\text{id}$ and $cr^{-1}\circ cr\cong\text{id}$. By Theorem \ref{960}, we have $\underline{\eta}\circ\underline{\epsilon}\cong\text{id}$ and $\underline{\epsilon}\circ\underline{\eta}\cong\text{id}.$
Therefore, we obtain
\begin{align*} \label{952}
& \mathcal{C}\circ \mathcal{E}=\underline{\eta}\circ cr \circ cr^{-1} \circ \underline{\epsilon}\cong\underline{\eta}\circ\underline{\epsilon}\cong\text{id} \\
& \mathcal{E}\circ \mathcal{C} = cr^{-1}\circ \underline{\epsilon}\circ\underline{\eta}\circ cr\cong cr^{-1}\circ cr\cong \text{id} \quad .  \qedhere
\end{align*} 
\end{proof}

\noindent
If $R$ is a notherian ring, we aim for a result analogous to the correspondence of persistence modules of finite type and finitely generated graded modules. To establish such a result we adapt the definition of functors $\mathbb{N}_0\rightarrow \mathbf{Mod}_R$ of finite type to functors $\mathbb{N}_0\rightarrow\mathbf{pShv}(X,\mathbf{Mod}_R)$.

\begin{definition}[Finite type] \label{473}
A functor $\vec{\mathcal{F}}\in\mathbf{Fun}\big(\mathbb{N}_0,\mathbf{pShv}(X,\mathbf{Mod}_R)\big)$ is of \emph{weak finite type} if $\mathcal{F}_i$ is a presheaf of finitely generated $R$-modules for every $i\in\mathbb{N}_0$ and if for every open $U\subseteq X$ there exists an $m_U\in\mathbb{N}_0$ such that $(\mathcal{F}_i^{i+1})_U$ is an isomorphism for all $i\geq m_U$. 

The functor $\vec{\mathcal{F}}$ is of \emph{finite type} if $\text{sup}\{m_U\text{ }|\text{ }U\subseteq X \text{ open}\}<\infty$, i.e.\ there exists an $m\in\mathbb{N}_0$ such that $\mathcal{F}_i^{i+1}$ is an isomorphism for all $i\geq m$. 
\end{definition}

\noindent
Note that a functor $\vec{\mathcal{F}}\in\mathbf{Fun}\big(\mathbb{N}_0,\mathbf{pShv}(X,\mathbf{Mod}_R)\big)$ of weak finite type corresponds to a presheaf with values in the category of persistence modules of finite type. 

\begin{proposition} \label{454}
Let $X$ be a topological space and $R$ a noetherian ring. The subcategory of $\mathbf{Fun}\big(\mathbb{N}_0,\mathbf{pShv}(X,\mathbf{Mod}_R)\big)$ of functors of weak finite type is equivalent to the category \\ $\mathbf{pShv}(X,\mathbf{grmod}_{R[t]})$ of presheaves of finitely generated graded $R[t]$-modules on $X$.
\end{proposition}

\begin{proof}
This follows from Definition \ref{473}, Theorem \ref{930} and Theorem \ref{960}. 
\end{proof}

\noindent
Since every functor $\vec{\mathcal{F}}\in\mathbf{Fun}\big(\mathbb{N}_0,\mathbf{pShv}(X,\mathbf{Mod}_R)\big)$ of finite type is also of weak finite type, every functor of finite type corresponds to a presheaf of finitely generated graded modules but the converse is not true in general. Given $M\in\mathbf{pShv}(X,\mathbf{grmod}_{R[t]})$, for every open $U\subseteq X$ we obtain an $m_U\in\mathbb{N}_0$ such that $(\mathcal{E}(M)_i^{i+1})_U$ is an isomorphism for all $i\geq m_U$. However, the set $\{m_U\text{ }|\text{ }U\subseteq X\text{ open}\}$ might be unbounded, i.e.\ there is no $m\in\mathbb{N}_0$ such that $(\mathcal{E}(M)_i^{i+1})_U$ is an isomorphism for all $i\geq m$ and all $U\subseteq X$ open. In the special case of a finite topological space, for example a finite simplicial complex with the Alexandrov topology, the number of open sets is finite, hence setting $m:=\text{max}\{m_U\text{ }|\text{ }U\subseteq X\text{ open}\}$ we obtain that $\mathcal{E}(M)$ is of finite type. Thus, in this case, the subcategory of $\mathbf{Fun}\big(\mathbb{N}_0,\mathbf{pShv}(X,\mathbf{Mod}_R)\big)$ of functors of finite type is equivalent to the category $\mathbf{pShv}(X,\mathbf{grmod}_{R[t]})$. 

Next we show that the correspondence of Proposition \ref{930} carries over to sheaves.

\begin{proposition} \label{282}
The functors $\mathcal{C}$ and $\mathcal{E}$ restrict to equivalences
\begin{equation*} \label{147}
\begin{aligned}
& \mathcal{C}\colon\mathbf{Fun}\big(\mathbb{N}_0,\mathbf{Shv}(X,\mathbf{Mod}_R)\big)\rightarrow \mathbf{Shv}(X,\mathbf{grMod}_{R[t]}) \\
& \mathcal{E}\colon\mathbf{Shv}(X,\mathbf{grMod}_{R[t]})\rightarrow \mathbf{Fun}\big(\mathbb{N}_0,\mathbf{Shv}(X,\mathbf{Mod}_R)\big)
\end{aligned}
\end{equation*} 
on the respective subcategories.
\end{proposition}

\begin{proof}
\textbf{1)} Let $\vec{\mathcal{F}}\in\mathbf{Fun}\big(\mathbb{N}_0,\mathbf{Shv}(X,\mathbf{Mod}_R)\big)$, we have to show that $\mathcal{C}(\vec{\mathcal{F}})\in\mathbf{Shv}(X,\mathbf{grMod}_{R[t]})$, i.e.\ given an open set $U\subseteq X$ and an open cover $\mathcal{U}=(U_i)_{i\in I}$ of $U$ we have to show that 
\begin{equation*} \label{425}
\begin{tikzcd}
0 \arrow[r] & \mathcal{C}(\vec{\mathcal{F}})(U) \arrow[r,"\gamma"] & \underset{i\in I}{\prod} \mathcal{C}(\vec{\mathcal{F}})(U_i)  \arrow[r,"\alpha-\beta"] &[40pt] \underset{i,j\in I}{\prod}  \mathcal{C}(\vec{\mathcal{F}})(U_i\cap U_j) 
\end{tikzcd}
\end{equation*} 
is exact. By construction, $\mathcal{C}(\vec{\mathcal{F}})(U)=\underset{n\in\mathbb{N}_0}{\bigoplus}\mathcal{F}_n(U)$ and, by the definition of the product in the category of graded modules, $\underset{i\in I}{\prod}\mathcal{C}(\vec{\mathcal{F}})(U_i)\cong\underset{n\in\mathbb{N}_0}{\bigoplus}\underset{i\in I}{\prod}\mathcal{F}_n(U_i)$ \cite[page 20]{gradedrings}. Since $\vec{\mathcal{F}}$ is a sheaf-valued functor, for every $n\in \mathbb{N}_0$, we obtain the following commutative diagram of $R$-modules
\begin{equation} \label{277}
\begin{tikzcd}[column sep=large]
0 \arrow[r] & \mathcal{F}_n(U) \arrow[r,"\gamma_n"] \arrow[d] & \underset{i\in I}{\prod} \mathcal{F}_n(U_i)  \arrow[r,"\alpha_n-\beta_n"] \arrow[d] &[40pt] \underset{i,j\in I}{\prod}  \mathcal{F}_n(U_i\cap U_j) \arrow[d] \\
0 \arrow[r] & \mathcal{F}_{n+1}(U) \arrow[r,"\gamma_{n+1}"] & \underset{i\in I}{\prod} \mathcal{F}_{n+1}(U_i) \arrow[r,"\alpha_{n+1}-\beta_{n+1}"] & \underset{i,j\in I}{\prod}  \mathcal{F}_{n+1}(U_i\cap U_j)
\end{tikzcd}
\end{equation} 
where the vertical arrows are the morphisms induced by $\mathcal{F}_n^{n+1}$ and the rows are exact. By taking the direct sum over all modules and horizontal arrows and by the commutativity of the squares in (\ref{277}), we obtain the following exact sequence of graded $R[t]$-modules
\begin{equation*} \label{373}
\begin{tikzcd}[column sep=large]
0 \arrow[r] & \underset{n\in\mathbb{N}_0}{\bigoplus}\mathcal{F}_n(U) \arrow[r,"\underset{n\in\mathbb{N}_0}{\bigoplus}\gamma_n\circ p_n"] &[10pt] \underset{n\in\mathbb{N}_0}{\bigoplus}\underset{i\in I}{\prod} \mathcal{F}_n(U_i) \arrow[r,"\underset{n\in\mathbb{N}_0}{\bigoplus}(\alpha_n-\beta_n)\circ p_n"] &[40pt] \underset{n\in\mathbb{N}_0}{\bigoplus} \underset{i,j\in I}{\prod}  \mathcal{F}_n(U_i\cap U_j)
\end{tikzcd}
\end{equation*}  
where $\gamma=\underset{n\in\mathbb{N}_0}{\bigoplus}\gamma_n\circ p_n$ and $\alpha-\beta=\underset{n\in\mathbb{N}_0}{\bigoplus}(\alpha_n-\beta_n)\circ p_n$. Hence, $\mathcal{C}(\vec{\mathcal{F}})$ is a sheaf of graded modules.

\textbf{2)} Let $M\in\mathbf{Shv}(X,\mathbf{grMod}_{R[t]})$, we have to show that $\mathcal{E}(M)\in \mathbf{Fun}\big(\mathbb{N}_0,\mathbf{Shv}(X,\mathbf{Mod}_R)\big)$, i.e.\ given an open set $U\subseteq X$ and an open cover $\mathcal{U}=(U_i)_{i\in I}$ of $U$ we have to show that
\begin{equation*} 
\begin{tikzcd}[column sep=large]
0 \arrow[r] & \mathcal{E}(M)_m(U) \arrow[r,"\gamma_m"] & \underset{i\in I}{\prod} \mathcal{E}(M)_m(U_i)  \arrow[r,"\alpha_m-\beta_m"] &[40pt] \underset{i,j\in I}{\prod}  \mathcal{E}(M)_m(U_i\cap U_j) 
\end{tikzcd}
\end{equation*} 
is exact, for every $m\in\mathbb{N}_0$. By construction, if $M(U)=\underset{n\in\mathbb{N}_0}{\bigoplus}M_n(U)$, then $\mathcal{E}(M)_m(U)=M_m(U)$. By the definition of the product of graded modules, $\underset{i\in I}{\prod}M(U_i)\cong\underset{n\in\mathbb{N}_0}{\bigoplus}\underset{i\in I}{\prod}M_n(U_i)$. Since $M$ is a sheaf, we obtain the following exact sequence 
\begin{equation*} \label{837}
\begin{tikzcd}[column sep=large]
0 \arrow[r] & \underset{n\in\mathbb{N}_0}{\bigoplus}M_n(U) \arrow[r,"\underset{n\in\mathbb{N}_0}{\bigoplus}\gamma_n\circ p_n"] &[10pt] \underset{n\in\mathbb{N}_0}{\bigoplus}\underset{i\in I}{\prod} M_n(U_i) \arrow[r,"\underset{n\in\mathbb{N}_0}{\bigoplus}(\alpha_n-\beta_n)\circ p_n"] &[40pt] \underset{n\in\mathbb{N}_0}{\bigoplus} \underset{i,j\in I}{\prod}  M_n(U_i\cap U_j)
\end{tikzcd}
\end{equation*}
of graded $R[t]$-modules and graded morphisms. Hence, for every $m\in\mathbb{N}_0$, we obtain the following exact sequence of $R$-modules
\begin{equation*} \label{356}
\begin{tikzcd}[column sep=large]
0 \arrow[r] & M_m(U) \arrow[r,"\gamma_m"] & \underset{i\in I}{\prod} M_m(U_i) \arrow[r,"\alpha_m-\beta_m"] &[40pt] \underset{i,j\in I}{\prod}  M_m(U_i\cap U_j)
\end{tikzcd}
\end{equation*}
and this implies that $\mathcal{E}(M)_m$ is a sheaf. 
\end{proof}

\noindent
By combining all the previous results, we obtain the following theorem characterizing functors $\vec{\mathcal{F}}\colon\mathbb{N}_0\rightarrow \mathbf{Shv}(X,\mathbf{Mod}_R)$.

\begin{theorem} \label{697}
Let $X$ be a topological space. Then there is an equivalence of categories 
\begin{equation*} 
\mathbf{Fun}\big(\mathbb{N}_0,\mathbf{Shv}(X,\mathbf{Mod}_R)\big)\cong \mathbf{Shv}(X,\mathbf{grMod}_{R[t]}) \quad .
\end{equation*}
Moreover, if $R$ is a noetherian ring and $X$ is finite, then the subcategory of $\mathbf{Fun}\big(\mathbb{N}_0,\mathbf{Shv}(X,\mathbf{Mod}_R)\big)$ of functors of finite type is equivalent to $\mathbf{Shv}(X,\mathbf{grmod}_{R[t]})$.
\end{theorem} 

\begin{proof}
This follows from Proposition \ref{930}, Proposition \ref{454} and Proposition \ref{282}. 
\end{proof}

\noindent
Theorem \ref{697} generalizes Theorem \ref{960}: If $X=\{p\}$ is a one-point space, then a sheaf with values in $\mathbf{M}$ on $X$ can be identified with an object in $\mathbf{M}$, hence $\mathbf{Shv}(X,\mathbf{M})\cong\mathbf{M}$. \newline

\subsection{Correspondence of sheaf persistence modules of algebraic type with the cohomology of sheaves of graded modules} \label{819}

The significance of Theorem \ref{697} is that it allows us to compute persistent sheaf cohomology in two different ways.

Let $\vec{\mathcal{F}}\in\mathbf{Fun}\big(\mathbb{N}_0,\mathbf{Shv}(X,\mathbf{vec}_\mathbb{F})\big)$ be a functor of finite type, where $X$ is a finite topological space. One way to compute persistent sheaf cohomology is to first apply the (pointwise) sheaf cohomology functor $\underline{H}^k(X,-)$ to $\vec{\mathcal{F}}$ :
\begin{equation*} 
\begin{tikzcd}[column sep=huge]
H^k(X,\mathcal{F}_0) \arrow[r,"H^k(X\text{,}\mathcal{F}_0^1)"] &[5pt] H^k(X,\mathcal{F}_1) \arrow[r,"H^k(X\text{,}\mathcal{F}_1^2)"] &[5pt] H^k(X,\mathcal{F}_2) \arrow[r,"H^k(X\text{,}\mathcal{F}_2^3)"] &[5pt] \cdots 
\end{tikzcd}
\end{equation*} 
then compute the graded $\mathbb{F}[t]$-module $\eta(\underline{H}^k(X,\vec{\mathcal{F}}))=\bigoplus_{n\in\mathbb{N}_0} H^k(X,\mathcal{F}_n)$,
with $t$-multiplication defined by $t \cdot\coloneqq H^k(X,\mathcal{F}_n^{n+1})\colon H^k(X,\mathcal{F}_n)\rightarrow H^k(X,\mathcal{F}_{n+1})$ for all $n\in\mathbb{N}_0$, corresponding to the obtained persistence module, and finally compute the interval decomposition of $\underline{H}^k(X,\vec{\mathcal{F}})$ from the direct sum decomposition of $\eta(\underline{H}^k(X,\vec{\mathcal{F}}))$ into indecomposable modules. 

The results of Section \ref{608} allow for the following alternative approach. Let $\mathcal{C}(\vec{\mathcal{F}})\in \mathbf{Shv}(X,\mathbf{grmod}_{\mathbb{F}[t]})$ be the sheaf of finitely generated graded modules corresponding to $\vec{\mathcal{F}}$. Then we can compute $H^k(X,\mathcal{C}(\vec{\mathcal{F}}))$, the ordinary sheaf cohomology  of $\mathcal{C}(\vec{\mathcal{F}})$. 

Consider again the example of $\vec{F}$ and $\underline{H}^0(X,\vec{F})$ in Figure \ref{257}. The graded $\mathbb{F}[t]$-module corresponding to $\underline{H}^0(X,\vec{F})$ is isomorphic to $\mathbb{F}[t]_1\oplus\mathbb{F}[t]_2\oplus \mathbb{F}[t]_4$. From this graded module, which is already written as a direct sum of indecomposables, we can read of the barcode depicted in Figure \ref{257}. On the other hand, the zero-dimensional cohomology of the sheaf of graded modules $\mathcal{C}(\vec{F})$ in (\ref{907}) is given by
\begin{equation*} 
H^0(X,\mathcal{C}(\vec{F}))\cong\text{ker}\begin{blockarray}{cccccc}
 & 0 & 1 & 0 & 2 & 4  \\
\begin{block}{c(ccccc)}
0 & -1 & 0 & 0 & 0 & t^4  \\
0 & 0 & -t & 1 & 0 & 0  \\
3 & 0 & 0 & 0 & 0 & 0 \\
\end{block}
\end{blockarray}\cong\text{ker}\begin{blockarray}{cccccc}
 & 0 & 1 & 0 & 2 & 4  \\
\begin{block}{c(ccccc)}
0 & -1 & 0 & 0 & 0 & 0  \\
0 & 0 & 0 & 1 & 0 & 0  \\
3 & 0 & 0 & 0 & 0 & 0 \\
\end{block}
\end{blockarray}\cong \mathbb{F}[t]_1\oplus\mathbb{F}[t]_2\oplus \mathbb{F}[t]_4
\end{equation*}
where the numbers before the rows and over the columns are the degrees of the corresponding basis elements in $C^0(X,\mathcal{C}(\vec{F}))$ and $C^1(X,\mathcal{C}(\vec{F}))$, respectively. In this example, both methods of computation discussed above yield the same result, i.e.\ $\bigoplus_{n\in\mathbb{N}_0} H^0(X,F_n)\cong H^0(X,\mathcal{C}(\vec{F}))$. Hence, the ordinary sheaf cohomology of the sheaf of graded modules $\mathcal{C}(\vec{F})$ computes the persistent sheaf cohomology of $\vec{F}$. The following results show that this holds in general. 

\begin{proposition} \label{238}
Let $\mathbf{A}$ be an abelian category. If $\vec{I}\in\mathbf{Fun}(\mathbb{N}_0,\mathbf{A})$ is injective, then $I_n$ is injective in $\mathbf{A}$ for all $n\in\mathbb{N}_0$.  
\end{proposition}

\begin{proof}
Note that $\mathbf{Fun}(\mathbb{N}_0,\mathbf{A})$ is abelian as a category of functors to an abelian category \cite[Page 25]{weibel}. Consider the following diagram in $\mathbf{A}$ 
\begin{equation} \label{179}
\begin{tikzcd}
0 \arrow[r] & A \arrow[r,"f"] \arrow[d,swap,"g"] & B \\
&  I_n 
\end{tikzcd}
\end{equation}
where the top row is exact. Inductively extend (\ref{179}) to the diagram of solid arrows in $\mathbf{Fun}(\mathbb{N}_0,\mathbf{A})$ as depicted in (\ref{136})
\begin{equation} \label{136}
\begin{tikzcd}[row sep=small,column sep=large]
& & & 0 \arrow[r] & I_{n+1} \arrow[r,"\alpha_{n+1}"] \arrow[d] & P_{n+1} \arrow[dl,dashed] \\[3pt]
& & & & I_{n+2} \arrow[ddl,<-] \\[-10pt]
& & 0 \arrow[r] & I_n \arrow[r,"\alpha_n",crossing over] \arrow[d] \arrow[uur,"I_n^{n+1}"] & P_n \arrow[dl,dashed] \arrow[uur,swap,"\beta_{n+1}"] \\[3pt]
& & & I_{n+1} \arrow[ddl,<-] \\[-10pt]
& 0 \arrow[r] & A \arrow[r,crossing over,"f"{xshift=-2pt}] \arrow[d] \arrow[uur,"g"] & B \arrow[uur,swap,"\beta_n"] \arrow[dl,dashed]  \\[3pt] 
& & I_n \arrow[ddl,<-] \\[-10pt]
0 \arrow[r] & 0 \arrow[r,crossing over] \arrow[d] \arrow[uur] & 0 \arrow[uur] \arrow[dl,dashed]  \\[3pt]
& I_{n-1} 
\end{tikzcd}
\end{equation}
where $(P_n,\alpha_n,\beta_n)$ is the pushout of $(A,f,g)$, $(P_{n+1},\alpha_{n+1},\beta_{n+1})$ is the pushout of $(I_n,\alpha_n,I_n^{n+1})$ and so forth. For a pushout diagram in an abelian category we have the following result: If $f$ is a monomorphism, then $\alpha_n$ is a monomorphism \cite[Lemma 12.5.13]{pushout}. Therefore, the horizontal rows are exact. Since $\vec{I}$ is injective in $\mathbf{Fun}(\mathbb{N}_0,\mathbf{A})$ there exists a morphism, consisting of the dashed arrows in (\ref{136}), making everything commute. Hence, we can extend $f$ in (\ref{179}) and $I_n$ is injective.
\end{proof}

\begin{theorem} \label{133}
The following diagram commutes
\begin{equation*} \label{614}
\begin{tikzcd}
\mathbf{Shv}(X,\mathbf{grMod}_{R[t]}) \arrow[r,shift left=1,"\mathcal{E}"] \arrow[d,swap,"H^k(X\text{,}-)"] &[10pt] \arrow[l,shift left=1,"\mathcal{C}"] \mathbf{Fun}\big(\mathbb{N}_0,\mathbf{Shv}(X,\mathbf{Mod}_R)\big) \arrow[d,"\underline{H}^k(X\text{,}-)"] \\
\mathbf{grMod}_{R[t]} \arrow[r,shift left=1,"\epsilon"] & \arrow[l,shift left=1,"\eta"] \mathbf{Fun}(\mathbb{N}_0,\mathbf{Mod}_R)
\end{tikzcd} 
\end{equation*}
i.e.\ there are natural isomorphisms $\underline{H}^k(X,-)\circ \mathcal{E}\cong \epsilon\circ H^k(X,-)$ and $\eta\circ\underline{H}^k(X,-)\cong H^k(X,-)\circ \mathcal{C}$.
\end{theorem}

\begin{proof}
By construction of the functors $\mathcal{C}$ and $\mathcal{E}$, the following diagram commutes
\begin{equation} \label{739}
\begin{tikzcd}
\mathbf{Shv}(X,\mathbf{grMod}_{R[t]}) \arrow[r,shift left=1,"\mathcal{E}"] \arrow[d,swap,"\Gamma(X\text{,}-)"] &[10pt] \arrow[l,shift left=1,"\mathcal{C}"] \mathbf{Fun}\big(\mathbb{N}_0,\mathbf{Shv}(X,\mathbf{Mod}_R)\big) \arrow[d,"\underline{\Gamma}(X\text{,}-)"] \\
\mathbf{grMod}_{R[t]} \arrow[r,shift left=1,"\epsilon"] & \arrow[l,shift left=1,"\eta"] \mathbf{Fun}(\mathbb{N}_0,\mathbf{Mod}_R)
\end{tikzcd} \quad .
\end{equation}
By currying, we obtain
\begin{equation*} \label{748}
\begin{tikzcd}\mathbf{coCh}\Big(\mathbf{Fun}\big(\mathbb{N}_0,\mathbf{Shv}(X,\mathbf{Mod}_R)\big)\Big)\arrow[r,shift left=1,"cr^{-1}"] & \arrow[l,shift left=1,"cr"] \mathbf{Fun}\Big(\mathbb{N}_0,\mathbf{coCh}\big(\mathbf{Shv}(X,\mathbf{Mod}_R)\big)\Big)
\end{tikzcd} \quad .
\end{equation*}
By the commutativity of (\ref{739}) and by Proposition \ref{635}, we obtain the following commutative diagram where we overload the notation of $cr$ and $\underline{\Gamma}$ 
\begin{equation} \label{570}
\begin{tikzcd}
\mathbf{coCh}\big(\mathbf{Shv}(X,\mathbf{grMod}_{R[t]})\big) \arrow[r,shift left=1,"cr^{-1}\circ \underline{\mathcal{E}}"] \arrow[d,swap,"\underline{\Gamma}(X\text{,}-)"] &[20pt] \arrow[l,shift left=1,"\underline{\mathcal{C}}\circ cr"] \mathbf{Fun}\Big(\mathbb{N}_0,\mathbf{coCh}\big(\mathbf{Shv}(X,\mathbf{Mod}_R)\big)\Big) \arrow[d,"\underline{\Gamma}(X\text{,}-)"] \\
\mathbf{coCh}\big(\mathbf{grMod}_{R[t]}\big) \arrow[r,shift left=1,"cr^{-1}\circ \underline{\epsilon}"] \arrow[d,swap,"H^k"] & \arrow[l,shift left=1,"\underline{\eta}\circ cr"] \mathbf{Fun}(\mathbb{N}_0,\mathbf{coCh}\big(\mathbf{Mod}_R)\big) \arrow[d,"\underline{H}^k"] \\
\mathbf{grMod}_{R[t]} \arrow[r,shift left=1,"\epsilon"] & \arrow[l,shift left=1,"\eta"] \mathbf{Fun}(\mathbb{N}_0,\mathbf{Mod}_R) 
\end{tikzcd} \quad .
\end{equation}
Let $\phi\colon M\rightarrow N$ be a morphism in $\mathbf{Shv}(X,\mathbf{grMod}_{R[t]})$ and $I^\bullet$ and $J^\bullet$ injective resolutions of $M$ and $N$, respectively. By Theorem \ref{171}, we obtain the following morphisms
\begin{equation*} \label{596}
\begin{tikzcd}
0 \arrow[r] & M \arrow[r] \arrow[d,"\phi"] & I^\bullet \arrow[d,"g^\bullet"] & & 0 \arrow[r] & \mathcal{E}(M) \arrow[r] \arrow[d,"\mathcal{E}(\phi)"] & \underline{\mathcal{E}}(I^\bullet) \arrow[d,"\underline{\mathcal{E}}(g^\bullet)"] \\
0 \arrow[r] & N \arrow[r] & J^\bullet & , & 0 \arrow[r] & \mathcal{E}(N) \arrow[r] & \underline{\mathcal{E}}(J^\bullet)
\end{tikzcd}
\end{equation*}
in $\mathbf{coCh}\big(\mathbf{Shv}(X,\mathbf{grMod}_{R[t]})\big)$ and $\mathbf{coCh}\Big(\mathbf{Fun}\big(\mathbb{N}_0,\mathbf{Shv}(X,\mathbf{Mod}_{R})\big)\Big)$, respectively. Since $\mathcal{E}$ is an equivalence of categories, $\underline{\mathcal{E}}(I^\bullet)$ and $\underline{\mathcal{E}}(J^\bullet)$ are injective resolutions of $\mathcal{E}(M)$ and $\mathcal{E}(N)$, respectively. Expanding the diagram on the right yields
\begin{equation} \label{259}
\begin{tikzcd}[column sep=tiny,row sep=small]
& &[-2pt] &[-2pt] &[-2pt] & \iddots &&&  \iddots &&&  \iddots &[-3pt] &[-3pt] \\
& 0 \arrow[rrr] &&& \mathcal{E}(M)_1 \arrow[rrr] \arrow[ddd] \arrow[ur] &&& \mathcal{E}(I^0)_1 \arrow[rrr] \arrow[ddd] \arrow[ur] &&& \mathcal{E}(I^1)_1 \arrow[rrr] \arrow[ddd] \arrow[ur] &&& \cdots \\
0 \arrow[rrr] &&& \mathcal{E}(M)_0 \arrow[rrr,crossing over] \arrow[ur] &&& \mathcal{E}(I^0)_0 \arrow[rrr,crossing over] \arrow[ur] &&& \mathcal{E}(I^1)_0 \arrow[rrr,crossing over] \arrow[ur] &&& \cdots \\[-25pt]
&&&&& \iddots &&&  \iddots &&&  \iddots \\
& 0 \arrow[rrr] &&& \mathcal{E}(N)_1 \arrow[rrr] \arrow[ur] &&& \mathcal{E}(J^0)_1 \arrow[rrr] \arrow[ur] &&& \mathcal{E}(J^1)_1 \arrow[rrr] \arrow[ur] &&& \cdots \\
0 \arrow[rrr] &&& \mathcal{E}(N)_0 \arrow[rrr] \arrow[ur] \arrow[uuu,<-,crossing over] &&& \mathcal{E}(J^0)_0 \arrow[rrr] \arrow[ur] \arrow[uuu,<-,crossing over] &&& \mathcal{E}(J^1)_0 \arrow[rrr] \arrow[ur] \arrow[uuu,<-,crossing over] &&& \cdots
\end{tikzcd} \quad .
\end{equation}
Since $\underline{\mathcal{E}}(I^\bullet)^l,\underline{\mathcal{E}}(J^\bullet)^l\in\mathbf{Fun}\big(\mathbb{N}_0,\mathbf{Shv}(X,\mathbf{Mod}_R)\big)$, represented by the linear diagrams going backwards in (\ref{259}), are injective, Proposition \ref{238} implies that $\mathcal{E}(I^l)_m$ and $\mathcal{E}(J^l)_m$ are injective for all $l,m\in\mathbb{N}_0$. Hence, for all $m\in\mathbb{N}_0$, $(cr^{-1}\circ\underline{\mathcal{E}})(I^\bullet)_m$ and $(cr^{-1}\circ\underline{\mathcal{E}})(J^\bullet)_m$, represented by the rows of (\ref{259}), are injective resolutions of $\mathcal{E}(M)_m$ and $\mathcal{E}(N)_m$, respectively. By the definition of $H^k(X,-)$ up to natural isomorphism (left- and rightmost square) and by the commutativity of (\ref{570}) (middle square), we obtain the following commutative diagram
\begin{equation*} \label{838}
\begin{tikzcd}[column sep=small]
\epsilon\big(H^k(X,M)\big) \arrow[r,"\cong"] \arrow[d,swap,"\epsilon(H^k(X\text{,}\phi))"] & \big(\epsilon\circ H^k\circ\underline{\Gamma}(X,-)\big)(I^\bullet) \arrow[d,"(\epsilon\circ H^k\circ\underline{\Gamma}(X\text{,}-))(g^\bullet)"] \arrow[r,"\cong"] & \big(\underline{H}^k\circ \underline{\Gamma}(X,-)\circ cr^{-1} \circ \underline{\mathcal{E}}\big)(I^\bullet) \arrow[d,"(\underline{H}^k\circ \underline{\Gamma}(X\text{,}-)\circ cr^{-1} \circ \underline{\mathcal{E}})(g^\bullet)"] \arrow[r,"\cong"] & \underline{H}^k\big(X,\mathcal{E}(M)\big) \arrow[d,"\underline{H}^k(X\text{,}\mathcal{E}(\phi))"] \\
\epsilon\big(H^k(X,N)\big) \arrow[r,"\cong"] & \big(\epsilon\circ H^k\circ\underline{\Gamma}(X,-)\big)(J^\bullet)  \arrow[r,"\cong"] & \big(\underline{H}^k\circ \underline{\Gamma}(X,-)\circ cr^{-1} \circ \underline{\mathcal{E}}\big)(J^\bullet) \arrow[r,"\cong"] & \underline{H}^k\big(X,\mathcal{E}(N)\big)
\end{tikzcd}
\end{equation*} 
and therefore $\epsilon\circ H^k(X,-)\cong \underline{H}^k(X,-)\circ \mathcal{E}$. Moreover, since $\mathcal{E}\circ \mathcal{C}\cong \text{id}$ and $\eta\circ \epsilon\cong \text{id}$, we also obtain $\eta\circ\underline{H}^k(X,-)\cong H^k(X,-)\circ \mathcal{C}$.
\end{proof}

\noindent
Theorem \ref{133} implies that for all $\vec{\mathcal{F}}\in\mathbf{Fun}\big(\mathbb{N}_0,\mathbf{Shv}(X,\mathbf{Mod}_R)\big)$ we have $\underline{H}^k(X,\vec{\mathcal{F}})\cong \epsilon\big(H^k\big(X,\mathcal{C}(\vec{\mathcal{F}})\big)\big)$, i.e.\ sheaf persistence modules of algebraic type correspond to the cohomology of sheaves of graded modules. In many cases sheaf cohomology agrees with \v{C}ech cohomology, which is often more convenient to use for practical computations. The following theorem states that the cohomological correspondence also holds in the setting of presheaves and \v{C}ech cohomology.

\begin{theorem} \label{285}
The following diagram commutes
\begin{equation*} \label{799}
\begin{tikzcd}
\mathbf{pShv}(X,\mathbf{grMod}_{R[t]}) \arrow[r,"\mathcal{E}",shift left=3pt] \arrow[d,swap,"\check{H}^k(X\text{,}-)"] &[15pt] \arrow[l,"\mathcal{C}",shift left=3pt] \mathbf{Fun}\big(\mathbb{N}_0,\mathbf{pShv}(X,\mathbf{Mod}_R)\big) \arrow[d,"\underline{\check{H}}^k(X\text{,}-)"] \\
\mathbf{grMod}_{R[t]} \arrow[r,"\epsilon",shift left=3pt] & \arrow[l,"\eta",shift left=3pt] \mathbf{Fun}(\mathbb{N}_0,\mathbf{Mod}_R)
\end{tikzcd} 
\end{equation*} 
i.e.\ there are natural isomorphisms $\epsilon\circ \check{H}^k(X,-)\cong \underline{\check{H}}^k(X,-)\circ \mathcal{E}$ and $\check{H}^k(X,-)\circ \mathcal{C}\cong \eta\circ \underline{\check{H}}^k(X,-)$.
\end{theorem}

\begin{proof}
Let $\mathcal{U}$ be an open cover of $X$ and $\phi\colon \vec{\mathcal{F}}\rightarrow \vec{\mathcal{G}}$ be the following morphism in $\mathbf{Fun}\big(\mathbb{N}_0,\mathbf{pShv}(X,\mathbf{Mod}_R)\big)$ 
\begin{equation*} \label{470}
\begin{tikzcd}
\mathcal{F}_0 \arrow[r] \arrow[d,"\phi_0"] & \mathcal{F}_1 \arrow[r] \arrow[d,"\phi_1"] & \mathcal{F}_2 \arrow[r] \arrow[d,"\phi_2"] & \cdots \\
\mathcal{G}_0 \arrow[r] & \mathcal{G}_1 \arrow[r] & \mathcal{G}_2 \arrow[r] & \cdots 
\end{tikzcd} \quad .
\end{equation*} 
Applying $\underline{\check{H}}^k(\mathcal{U},-)$ yields the morphism
\begin{equation*} \label{513}
\begin{tikzcd}
\check{H}^k(\mathcal{U},\mathcal{F}_0) \arrow[r] \arrow[d,"\check{H}^k(\mathcal{U}\text{,}\phi_0)"] & \check{H}^k(\mathcal{U},\mathcal{F}_1) \arrow[r] \arrow[d,"\check{H}^k(\mathcal{U}\text{,}\phi_1)"] & \check{H}^k(\mathcal{U},\mathcal{F}_2) \arrow[r] \arrow[d,"\check{H}^k(\mathcal{U}\text{,}\phi_2)"] & \cdots \\
\check{H}^k(\mathcal{U},\mathcal{G}_0) \arrow[r] & \check{H}^k(\mathcal{U},\mathcal{G}_1) \arrow[r] & \check{H}^k(\mathcal{U},\mathcal{G}_2) \arrow[r] & \cdots 
\end{tikzcd} 
\end{equation*} 
in $\mathbf{Fun}(\mathbb{N}_0,\mathbf{Mod}_R)$  and applying $\eta$ yields the morphism
\begin{equation*} \label{810}
\begin{tikzcd}
\eta\big(\underline{\check{H}}^k(\mathcal{U},\phi)\big)\colon &[-35pt] \eta\big(\underline{\check{H}}^k(\mathcal{U},\vec{\mathcal{F}})\big) \arrow[r] & \eta\big(\underline{\check{H}}^k(\mathcal{U},\vec{\mathcal{G}})\big) \\[-25pt]
\verteq & \verteq & \verteq \\[-25pt]
\underset{n\in\mathbb{N}_0}{\bigoplus}\check{H}^k(\mathcal{U},\phi_n)\circ p_n\colon &[-35pt] \underset{n\in\mathbb{N}_0}{\bigoplus}\check{H}^k(\mathcal{U},\mathcal{F}_n) \arrow[r] & \underset{n\in\mathbb{N}_0}{\bigoplus}\check{H}^k(\mathcal{U},\mathcal{G}_n)
\end{tikzcd}
\end{equation*} 
in $\mathbf{grMod}_{R[t]}$ where the $t$-multiplication is given by $t \cdot\coloneqq\check{H}^k(\mathcal{U},\mathcal{F}_n^{n+1})\colon\check{H}^k(\mathcal{U},\mathcal{F}_n)\rightarrow \check{H}^k(\mathcal{U},\mathcal{F}_{n+1})$. On the other hand, we obtain the following graded $R[t]$-modules of \v{C}ech cochains 
\begin{equation*} \label{494}
C^k\big(\mathcal{U},\mathcal{C}(\vec{\mathcal{F}})\big)=\underset{\sigma\in N_\mathcal{U}^k}{\prod}\mathcal{C}(\vec{\mathcal{F}})(U_{\sigma})\cong\underset{n\in\mathbb{N}_0}{\bigoplus}\underset{\sigma\in N_\mathcal{U}^k}{\prod}\mathcal{F}_n(U_{\sigma})=\underset{n\in\mathbb{N}_0}{\bigoplus}C^k(\mathcal{U},\mathcal{F}_n)
\end{equation*}
where we use the definition of the product in the category of graded modules \cite[page 20]{gradedrings} and where $t \cdot\coloneqq\underset{\sigma\in N_\mathcal{U}^k}{\prod}(\mathcal{F}_n^{n+1})_{U_{\sigma}}\circ p_\sigma \colon C^k(\mathcal{U},\mathcal{F}_n)\rightarrow C^k(\mathcal{U},\mathcal{F}_{n+1})$. Moreover, we obtain the graded boundary morphisms
\begin{align*} \label{318}
\delta^k&=\underset{\sigma\in N_\mathcal{U}^{k+1}}{\prod}\Big(\sum_{j=0}^{k+1}(-1)^j\mathcal{C}(\vec{\mathcal{F}})(U_{\sigma}\xhookrightarrow{} U_{\partial_j\sigma})\circ p_{\partial_j\sigma}\Big) \\
&=\underset{\sigma\in N_\mathcal{U}^{k+1}}{\prod}\Big(\sum_{j=0}^{k+1}(-1)^j\underset{n\in\mathbb{N}_0}{\bigoplus}\big(\mathcal{F}_n(U_{\sigma}\xhookrightarrow{} U_{\partial_j\sigma})\circ p^n_{\partial_j\sigma}\circ p_n\big)\Big) \\
&=\underset{\sigma\in N_\mathcal{U}^{k+1}}{\prod} \underset{n\in\mathbb{N}_0}{\bigoplus} \Big(\sum_{j=0}^{k+1}(-1)^j\mathcal{F}_n(U_{\sigma}\xhookrightarrow{} U_{\partial_j\sigma})\circ p^n_{\partial_j\sigma}\circ p_n\Big) \\
&= \underset{n\in\mathbb{N}_0}{\bigoplus}\underset{\sigma\in N_\mathcal{U}^{k+1}}{\prod} \Big(\sum_{j=0}^{k+1}(-1)^j\mathcal{F}_n(U_{\sigma}\xhookrightarrow{} U_{\partial_j\sigma})\circ p^n_{\partial_j\sigma}\Big)\circ p_n = \underset{n\in\mathbb{N}_0}{\bigoplus} \delta_n^k\circ p_n \quad.
\end{align*}
Hence, the \v{C}ech cohomology of $\mathcal{C}(\vec{\mathcal{F}})$ with respect to $\mathcal{U}$ is given by
\begin{equation*} \label{317}
\check{H}^k\big(\mathcal{U},\mathcal{C}(\vec{\mathcal{F}})\big)=H^k\big(C^\bullet(\mathcal{U},\mathcal{C}(\vec{\mathcal{F}}))\big)\cong \underset{n\in\mathbb{N}_0}{\bigoplus}H^k\big(C^\bullet(\mathcal{U},\mathcal{F}_n)\big)=\underset{n\in\mathbb{N}_0}{\bigoplus}\check{H}^k(\mathcal{U},\mathcal{F}_n)
\end{equation*}
where $t \cdot\coloneqq\check{H}^k\big(\mathcal{U},\mathcal{F}_n^{n+1}\big) \colon \check{H}^k(\mathcal{U},\mathcal{F}_n)\rightarrow \check{H}^k(\mathcal{U},\mathcal{F}_{n+1})$. Analogously one can show that 
\begin{equation*} \label{478}
\check{H}^k\big(\mathcal{U},\mathcal{C}(\phi)\big)=\underset{n\in\mathbb{N}_0}{\bigoplus}\check{H}^k(\mathcal{U},\phi_n)\circ p_n \quad .
\end{equation*}
This implies the following commutative diagram
\begin{equation*} \label{605}
\begin{tikzcd}
\eta\big(\underline{\check{H}}^k(\mathcal{U},\vec{\mathcal{F}})\big) \arrow[r,"\cong"] \arrow[d,swap,"\eta\big(\underline{\check{H}}^k(\mathcal{U}\text{,}\phi)\big)"] & \check{H}^k\big(\mathcal{U},\mathcal{C}(\vec{\mathcal{F}})\big) \arrow[d,"\check{H}^k\big(\mathcal{U}\text{,}\mathcal{C}(\phi)\big)"] \\
\eta\big(\underline{\check{H}}^k(\mathcal{U},\vec{\mathcal{G}})\big) \arrow[r,"\cong"] & \check{H}^k\big(\mathcal{U},\mathcal{C}(\vec{\mathcal{G}})\big)
\end{tikzcd} 
\end{equation*} 
and therefore $\eta\circ\underline{\check{H}}^k(\mathcal{U},-)\cong\check{H}^k(\mathcal{U},-)\circ \mathcal{C}$. By using $\check{H}^k(X,-)\cong \underset{\mathcal{U}}{\text{colim }}\check{H}^k(\mathcal{U},-)$ and the universal property of the colimit, one can see that $\underline{\check{H}}^k(X,\vec{\mathcal{F}})\cong \underset{\mathcal{U}}{\text{colim }}\underline{\check{H}}^k(\mathcal{U},\vec{\mathcal{F}})$, $\underline{\check{H}}^k(X,\phi)=\big(\underset{\mathcal{U}}{\text{colim }}\underline{\check{H}}^k(\mathcal{U},\vec{\mathcal{F}})\rightarrow\underset{\mathcal{U}}{\text{colim }}\underline{\check{H}}^k(\mathcal{U},\vec{\mathcal{G}})\big)$ and moreover $\underline{\check{H}}^k(X,-)\cong \underset{\mathcal{U}}{\text{colim }}\underline{\check{H}}^k(\mathcal{U},-)$. Since $\mathcal{C}$ and $\eta$ are equivalences, we obtain
\begin{align*} \label{412}
\eta\circ\underline{\check{H}}^k(X,-)&\cong \eta\circ \big(\underset{\mathcal{U}}{\text{colim }}\underline{\check{H}}^k(\mathcal{U},-)\big)\cong\underset{\mathcal{U}}{\text{colim }}\big(\eta\circ\underline{\check{H}}^k(\mathcal{U},-)\big) \\
& \cong \underset{\mathcal{U}}{\text{colim }}\big(\check{H}^k(\mathcal{U},-)\circ \mathcal{C}\big) \cong \big(\underset{\mathcal{U}}{\text{colim }}\check{H}^k(\mathcal{U},-)\big)\circ \mathcal{C} \cong \check{H}^k(X,-)\circ \mathcal{C} \quad .
\end{align*}
Since $\epsilon\circ\eta\cong\text{id}$ and $\mathcal{C}\circ \mathcal{E}\cong\text{id}$, we also obtain $\underline{\check{H}}^k(X,-)\circ \mathcal{E}\cong \epsilon\circ \check{H}^k(X,-)$.
\end{proof}

\noindent
A large class of spaces where sheaf and \v{C}ech cohomology agree are the paracompact Hausdorff spaces \cite[Theorem 13.17]{cohomology}. As we show in Section \ref{233}, another class of such spaces are abstract simplicial complexes with the Alexandrov topology. In the case of a simplicial complexes $X$ equipped with the Alexandrov topology, there exists a (not necessarily unique) finest open cover $\mathcal{U}$ of $X$. This implies that $\check{H}^k(X,-)\cong \check{H}^k(\mathcal{U},-)$. For all spaces $X$ that satisfy $H^k(X,-)\cong\check{H}^k(X,-)\cong \check{H}^k(\mathcal{U},-)$ for some finite open cover $\mathcal{U}=(U_i)_{i\in I}$, by Theorem \ref{285}, we have $\eta\big(\underline{H}^k(X,\vec{\mathcal{F}})\big)\cong \eta\big(\underline{\check{H}}^k(\mathcal{U},\vec{\mathcal{F}})\big)\cong \check{H}^k\big(\mathcal{U},\mathcal{C}(\vec{\mathcal{F}})\big)$, for all $\vec{\mathcal{F}}\in\mathbf{Fun}\big(\mathbb{N}_0,\mathbf{Shv}(X,\mathbf{vec}_\mathbb{F})\big)$. Moreover, the interval decomposition of $\underline{H}^k(X,\vec{\mathcal{F}})$ can be computed from the decomposition of $\eta\big(\underline{H}^k(X,\vec{\mathcal{F}})\big)$. Therefore, Theorem \ref{285} suggests the following algorithm to compute the persistent sheaf cohomology of a filtration of sheaves. \newline

\noindent
\textbf{Algorithm:}
Given a filtration $\vec{\mathcal{F}}\colon\mathbf{[n]}\rightarrow \mathbf{Shv}(X,\mathbf{vec}_\mathbb{F})$ of sheaves on $X$, i.e.\ $\mathcal{F}_m^{m+1}$ is an inclusion or, more generally, a monomorphism for all $0\leq m< n-1$, and a finite open cover $\mathcal{U}=(U_i)_{i\in I}$ of $X$ such that $H^k(X,-)\cong\check{H}^k(\mathcal{U},-)$. The information necessary to compute $\check{H}^k\big(\mathcal{U},\mathcal{C}(\vec{\mathcal{F}})\big)$ can be represented by a finite collection of matrices $\mathcal{F}_m(U_{\sigma}\xhookrightarrow{}U_{\partial_j\sigma})$ and $(\mathcal{F}_m^{m+1})_{U_{\sigma}}$ for all $\sigma\in N_\mathcal{U}^k$, $0\leq k<|I|$ and $0\leq m\leq n-1$. Given this input we proceed in the following way:

\begin{enumerate}
\item Construct the sheaf of graded modules $\mathcal{C}(\vec{\mathcal{F}})$.
\item Construct the \v{C}ech cochain complex $C^k\big(\mathcal{U},\mathcal{C}(\vec{\mathcal{F}})\big)$. 
\item Compute the \v{C}ech cohomology of $\mathcal{C}(\vec{\mathcal{F}})$ by (graded) matrix reduction.
\item Compute the persistence barcode from $\check{H}^k\big(X,\mathcal{C}(\vec{\mathcal{F}})\big)$.
\end{enumerate}

\noindent
Note that, since the sheaf morphisms $\mathcal{F}_m^{m+1}$ are monomorphisms, the linear maps $(\mathcal{F}_m^{m+1})_{U_\sigma}$ are injective. This implies that the graded modules $\mathcal{C}(\vec{\mathcal{F}})(U_{\sigma})$ are free. Hence, $C^\bullet\big(\mathcal{U},\mathcal{C}(\vec{\mathcal{F}})\big)$ is a cochain complex of free graded modules and the boundary morphisms can be represented by matrices with respect to homogenous bases. 

By Proposition \ref{538} and Theorem \ref{742}, this method of computation is applicable, for example, to filtrations of cellular sheaves on finite simplicial complexes and the \v{C}ech coboundary morphisms are given by (\ref{604}). 

Under the conditions discussed above, we have an explicit method to compute persistent sheaf cohomology by (graded) matrix reduction. This method of computation resembles the method used in persistent homology as described in \cite{carlsson}. The advantage of this method is that, in the graded module approach, we only consider each generator of the graded modules once whereas, if we would compute the cohomology of each sheaf $\mathcal{F}_m$ individually, we would reconsider a generator at the $0$-th level on every consecutive level.

\begin{remark} \label{363}
Instead of linear diagrams of sheaves we could also consider linear diagrams of cosheaves $\vec{\mathcal{L}}\in\mathbf{Fun}\big(\mathbb{N}_0,\mathbf{coShv}(X,\mathbf{Mod}_R)\big)$. By the same (or dual) arguments, one can show that Theorem \ref{697} also holds for cosheaves, i.e.\ $\mathbf{Fun}\big(\mathbb{N}_0,\mathbf{coShv}(X,\mathbf{Mod}_R)\big)\cong \mathbf{coShv}(X,\mathbf{grMod}_{R[t]})$ with an analogous statement for the subcatgories of functors of finite type and cosheaves of finitely generated graded modules if $X$ is a finite space. If $R=\mathbb{F}$ and $X$ is a poset equipped with the Alexandrov topology, then the category of cosheaves on $X$ has enough projectives \cite[page 117]{curry}. Hence, in (dual) analogy to the sheaf case, we can define the cosheaf homology functor $H_k(X,-)\colon\mathbf{coShv}(X,\mathbf{Vec}_\mathbb{F})\rightarrow \mathbf{Vec}_\mathbb{F}$ as the $k$-th left derived functor of the gloabl section functor. By applying this functor to $\vec{\mathcal{L}}$, we obtain a \emph{cosheaf persistence module of algebraic type} $H_k(X,-)\circ \vec{\mathcal{L}}$. In this case, by the same (dual) arguments, we can also show a cosheaf version of Theorem \ref{133} and \ref{285}, i.e.\ the persistent cosheaf homology of $\vec{\mathcal{L}}$ can be computed from the homology of the corresponding cosheaf of graded modules. Therefore, the whole theory of Section \ref{129} also applies to cosheaves of vector spaces on posets.
\end{remark}

\section{Persistent sheaf cohomology type T} \label{779}

\subsection{Sheaf copersistence modules of topological type} \label{769}

There is a connection between sheaf cohomology and singular or simplicial cohomology. If $X$ is a semi-locally contractible topological space, then, for all $k\geq 0$, we have $H^k_{\text{sing}}(X,R)\cong H^k(X,R_X)$ \cite{sella}. In other words, if $X$ is a sufficiently nice topological space, then the singular cohomology of $X$ is the sheaf cohomology of $R_X$ (the constant $R$-valued sheaf on $X$). If $X$ is a simplicial complex, then, as described in Example \ref{374}, the cohomology of the constant sheaf on $X$ agrees with the simplicial cohomology of $X$. Therefore, it is natural to ask the question: Can we obtain persistent singular or simplicial cohomology from persistent sheaf cohomology ? The problem is that the construction of Section \ref{129} applies to objects in $\mathbf{Fun}\big(\mathbb{N}_0,\mathbf{Shv}(X,\mathbf{Mod}_R)\big)$, whereas the usual construction of cohomology copersistence modules applies to objects in $\mathbf{Fun}(\mathbb{N}_0,\mathbf{Top})$. Therefore, we now introduce an alternative construction of sheaf (co)persistence modules, generalizing the construction of ordinary cohomology copersistence modules in a natural way.  

Let $\vec{X}\in\mathbf{Fun}(\overline{\mathbb{N}}_0,\mathbf{Top})$ be a functor on the category corresponding to the poset $\overline{\mathbb{N}}_0\coloneqq\mathbb{N}_0\cup\{\infty\}$ and let $\mathcal{F}\in\mathbf{Shv}(X_\infty,\mathbf{Mod}_R)$. The functor $\vec{X}$ corresponds to the following commutative diagram 
\begin{equation} \label{269}
\begin{tikzcd}[column sep=large,row sep=large]
X_0 \arrow[r,"X_0^1"] \arrow[dr,swap,"X_0^\infty"] & X_1 \arrow[r,"X_1^2"] \arrow[d,swap,"X_1^\infty"] & X_2 \arrow[r,"X_{2}^3"] \arrow[dl,swap,"X_2^\infty"] & \cdots \arrow[dll,] \\
& X_\infty
\end{tikzcd} \, .
\end{equation}
Define, for every $i\in\mathbb{N}_0$, the sheaf $\mathcal{F}^i:=(X_i^\infty)^*\mathcal{F}$ on $X_i$, where $(X_i^\infty)^*$ is the inverse image functor with respect to $X_i^\infty$ (Definition \ref{894}). Note that $\mathcal{F}^i=(X_i^\infty)^*\mathcal{F}=(X_{i+1}^\infty\circ X_i^{i+1})^*\mathcal{F}=(X_i^{i+1})^*(X_{i+1}^\infty)^*\mathcal{F}=(X_i^{i+1})^*\mathcal{F}^{i+1}$. Hence, we can visualize the construction as iteratively pulling back the sheaves $\mathcal{F}^{i}$ along the maps $X_{i-1}^i$  
\begin{equation*} 
\begin{tikzcd}[column sep=large]
\mathcal{F}^0 & \arrow[l,swap,"(X_0^1)^*",maps to] \mathcal{F}^1 & \arrow[l,swap,"(X_1^2)^*",maps to] \mathcal{F}^2 & \arrow[l,swap,"(X_2^3)^*",maps to] \cdots \\[-20pt]
X_0 \arrow[r,"X_0^1"] & X_1 \arrow[r,"X_1^2"] & X_2 \arrow[r,"X_2^3"] & \cdots 
\end{tikzcd} \, .
\end{equation*}
Now, for every $i\in\mathbb{N}_0$, apply the sheaf cohomology functor $H^k(X_i,-)$ to the sheaf $\mathcal{F}^i$ and connect the sheaf cohomology modules by the induced morphisms $H^k(X_i^{i+1})\colon H^k(X_{i+1},\mathcal{F}^{i+1})\rightarrow H^k(X_i,(X_i^{i+1})^*\mathcal{F}^{i+1})=H^k(X_i,\mathcal{F}^i)$ \cite[page 100]{iversen}, to obtain the functor $\cev{D}^k(\vec{X},\mathcal{F})\colon\mathbb{N}_0^\text{op}\rightarrow \mathbf{Mod}_R$
\begin{equation*} 
\begin{tikzcd}[column sep=large]
H^k(X_{0},\mathcal{F}^{0}) & \arrow[l,swap,"H^k(X_{0}^1)"] H^k(X_{1},\mathcal{F}^{1}) & \arrow[l,swap,"H^k(X_{1}^2)"] H^k(X_{2},\mathcal{F}^{2}) & \arrow[l,swap,"H^k(X_{2}^3)"] \cdots \quad .
\end{tikzcd}
\end{equation*}
Hence, given $\vec{X}\in\mathbf{Fun}(\overline{\mathbb{N}}_0,\mathbf{Top})$ and $\mathcal{F}\in\mathbf{Shv}(X_{\infty},\mathbf{Mod}_R)$, for every $k\in\mathbb{N}_0$, we obtain a copersistence module $\cev{D}^k(\vec{X},\mathcal{F})\in\mathbf{Fun}(\mathbb{N}_0^\text{op},\mathbf{Mod}_R)$. We call $\cev{D}^k(\vec{X},\mathcal{F})$ a \emph{sheaf copersistence module of topological type} or for brevity a \emph{copersistence module of type T}.

We can also consider a functor $\vec{X}\in\mathbf{Fun}(\mathbf{[n]},\mathbf{Top})$ as depicted in the following diagram
\begin{equation} \label{832}
\begin{tikzcd}
X_0 \arrow[r,"X_0^1"] & X_1 \arrow[r,"X_1^2"] & \cdots \arrow[r,"X_{n-3}^{n-2}"] & X_{n-2} \arrow[r,"X_{n-2}^{n-1}"] & X_{n-1}
\end{tikzcd}
\end{equation}
and a sheaf $\mathcal{F}\in\mathbf{Shv}(X_{n-1},\mathbf{Mod}_R)$ and construct $\cev{D}^k(\vec{X},\mathcal{F})\in\mathbf{Fun}(\mathbf{[n]}^\text{op},\mathbf{Mod}_R)$ in a similar way. In this case, we pull back from the terminal object $n-1$ of $\mathbf{[n]}$. We can identify (\ref{832}) with a functor $\vec{X}'\in\mathbf{Fun}(\overline{\mathbb{N}}_0,\mathbf{Top})$ such that $\vec{X}'$ agrees with $\vec{X}$ on $\mathbf{[n]}$, $X_i'\coloneqq X_{n-1}$ and $(X')_i^{i+1}\coloneqq\text{id}$ for all $i\geq n-1$. 

If the maps $X_i^{i+1}$ in (\ref{832}) are inclusions, one can think about the pair $(\vec{X},\mathcal{F})$ as a sheaf on a filtered topological space. If $R=\mathbb{F}$ and $\cev{D}^k(\vec{X},\mathcal{F})\in\mathbf{Fun}(\mathbb{N}_0^\text{op},\mathbf{vec}_\mathbb{F})$, by Remark \ref{151}, there is an interval decomposition $\cev{D}^k(\vec{X},\mathcal{F})\cong\bigoplus_{i\in I}\cev{I}_{[a_i,b_i]}$. An interval module $\vec{I}_{[a_i,b_i]}$ in this decomposition corresponds to a sheaf cohomology class in $\cev{D}^k(\vec{X},\mathcal{F})$ that is born at index $a_i$ and dies at index $b_i$. The corresponding persistence barcode $\text{BC}\big(\cev{D}^k(\vec{X},\mathcal{F})\big)$ describes the evolution of all sheaf cohomology classes of degree $k$ along the sheaves over $\vec{X}$. We call the collection of interval modules corresponding to a sheaf copersistence module of topological type, represented by $\text{BC}\big(\cev{D}^k(\vec{X},\mathcal{F})\big)$, the \emph{(type T) persistent sheaf cohomology} of $(\vec{X},\mathcal{F})$.  

Suppose $\vec{X}\in\mathbf{Fun}(\overline{\mathbb{N}}_0,\mathbf{Top})$ such that $X_i$ is semi-locally contractible for all $i\in\mathbb{N}_0$ and $\mathcal{F}=R_{X_\infty}$ is the constant sheaf on $X_\infty$. It is easy to see that the inverse image of the constant sheaf on $X_\infty$ is the constant sheaf on $X_i$, i.e.\ $\mathcal{F}^i=(X_i^\infty)^*R_{X_\infty}\cong R_{X_i}$ for all $i\in\mathbb{N}_0$. This implies $H^k(X_i,\mathcal{F}^i)\cong H^k_{\text{sing}}(X_i,R)$. Moreover, the induced morphism $H^k(X_i^{i+1})$ agrees with the induced morphism $H^k_{\text{sing}}(X_i^{i+1})$ in singular cohomology. Therefore, we get $\cev{D}^k(\vec{X},\mathcal{F})\cong H^k_\text{sing}(-,R)\circ \vec{X}|_{\mathbb{N}_0}$ and, if we assume $R=\mathbb{F}$ and finite-dimensional cohomology, the type T persistent sheaf cohomology of $(\vec{X},\mathcal{F})$ agrees with the persistent singular cohomology of $\vec{X}|_{\mathbb{N}_0}$. If we consider simplicial complexes, by a similar argument we obtain that type T persistent sheaf cohomology also generalizes persistent simplicial cohomology. Hence, type T persistent sheaf cohomology is a natural generalization of ordinary persistent cohomology. 

The example in Figure \ref{624} shows how this construction plays out in the case of cellular sheaves on simplicial complexes if we consider a general sheaf $F$ instead of the constant sheaf. 

One could also consider index categories with arbitrary orientations or index categories of the form $\mathbb{N}_0^d$ together with some additional terminal object and construct zigzag or multi-dimensional copersistence modules of type T.
 
\begin{figure}[h] 
\centering
\includegraphics[scale=0.52]{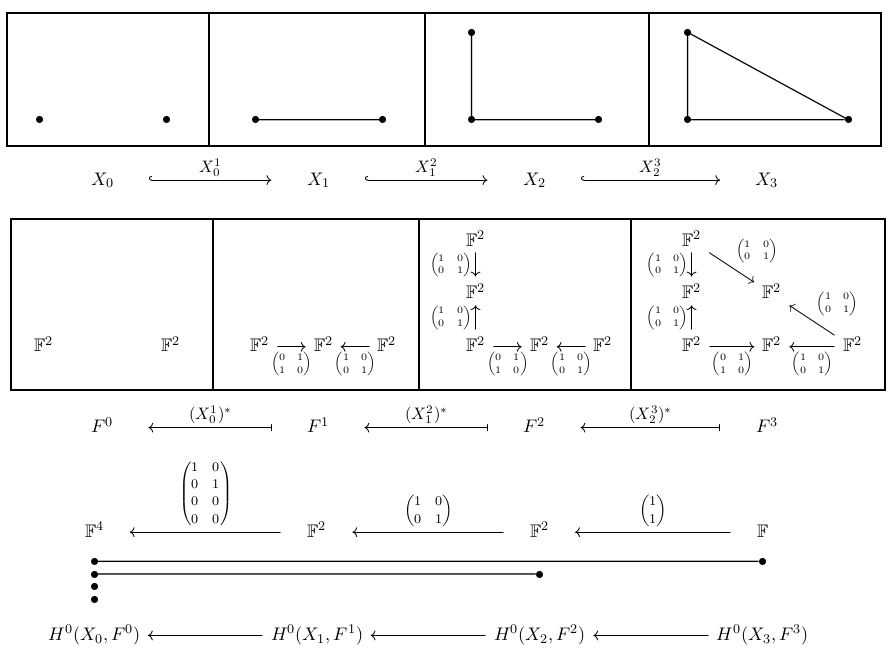}
\caption{The first row shows a filtered simplicial complex $X_3$. The second row shows a sheaf $F_3$ on $X_3$ and its pullbacks to the subcomplexes of the filtration. The third row shows the corresponding copersistence module of type T and persistence barcode.}
\label{624}
\end{figure}

\subsection{Relation between (co)persistence modules of type A and T} \label{959}

Now that we have two different versions of sheaf (co)persistence modules the question of how these two constructions relate to each other arises. Consider again the diagram of topological spaces and continuous maps $\vec{X}\in\mathbf{Fun}(\overline{\mathbb{N}}_0,\mathbf{Top})$, as depicted  in (\ref{269}), and a sheaf $\mathcal{F}$ on $X_{\infty}$. For every $i\in\mathbb{N}_0$, we can apply the direct image functor $(X_i^\infty)_*$ (Definition \ref{434}) to the sheaf $\mathcal{F}^i=(X_i^\infty)^*\mathcal{F}$ on $X_i$, to push it forward to a sheaf $\mathcal{G}_i\coloneqq(X_i^\infty)_*(X_i^\infty)^*\mathcal{F}$ on $X_{\infty}$. Since the inverse image functor $(X_{i-1}^i)^*$ is left adjoint to the direct image functor $(X_{i-1}^i)_*$ (Theorem \ref{573}), for every $i\geq 1$, there is a natural transformation $\overline{\eta}_i\colon\text{id}\rightarrow (X_{i-1}^i)_*(X_{i-1}^i)^*$ inducing a sheaf morphism $\eta_{i}\colon\mathcal{F}^i\rightarrow (X_{i-1}^i)_*(X_{i-1}^i)^*\mathcal{F}^i$ or equivalently $\eta_{i}\colon(X_i^\infty)^*\mathcal{F}\rightarrow (X_{i-1}^i)_*(X_{i-1}^i)^*(X_i^\infty)^*\mathcal{F}$. Applying the direct image functor $(X_i^\infty)_*$ yields a morphism $\mathcal{G}_i^{i-1}\coloneqq(X_i^\infty)_*\eta_{i}\colon(X_i^\infty)_*(X_i^\infty)^*\mathcal{F}\rightarrow (X_i^\infty)_*(X_{i-1}^i)_*(X_{i-1}^i)^*(X_i^\infty)^*\mathcal{F}=(X_{i-1}^\infty)_*(X_{i-1}^\infty)^*\mathcal{F}$ of sheaves on $X_{\infty}$. Hence, to the diagram of topological spaces $\vec{X}$ and the sheaf $\mathcal{F}$ on $X_{\infty}$ we can associate the following functor $\cev{\mathcal{G}}(\vec{X},\mathcal{F})\in\mathbf{Fun}\big(\mathbb{N}_0^\text{op},\mathbf{Shv}(X_{\infty},\mathbf{Mod}_R)\big)$ 
\begin{equation*} 
\begin{tikzcd}
\mathcal{G}_{0} & \arrow[l,swap,"\mathcal{G}_1^0"] \mathcal{G}_{1} & \arrow[l,swap,"\mathcal{G}_2^1"] \mathcal{G}_2 & \arrow[l,swap,"\mathcal{G}_3^2"] \cdots \quad .
\end{tikzcd}
\end{equation*}
Figure \ref{875} shows an example where $\vec{X}\in\mathbf{Fun}(\mathbf{[3]},\mathbf{Top})$ is a finite filtration of simplicial complexes and $F$ is the constant $\mathbb{F}$-valued cellular sheaf on $X_3$. As already discussed, in this case the persistent sheaf cohomology of type T of $(\vec{X},F)$ agrees with the usual simplicial persistent cohomology of the filtration $\vec{X}$. Figure \ref{875} also shows the corresponding functor $\cev{G}(\vec{X},F)$ of cellular sheaves on $X_3$, the sheaf copersistence module of algebraic type $H^k(X_3,-)\circ \cev{G}(\vec{X},F)$ and its persistence barcode. The persistence barcode of $H^k(X_3,-)\circ \cev{G}(\vec{X},F)$ is exactly the simplicial persistent cohomology barcode of the filtration $\vec{X}$. Therefore, in this example, $\cev{D}^k(\vec{X},F)\cong H^k(X_3,-)\circ \cev{G}(\vec{X},F)$.  
\begin{figure}[h!] 
\centering
\includegraphics[scale=0.55]{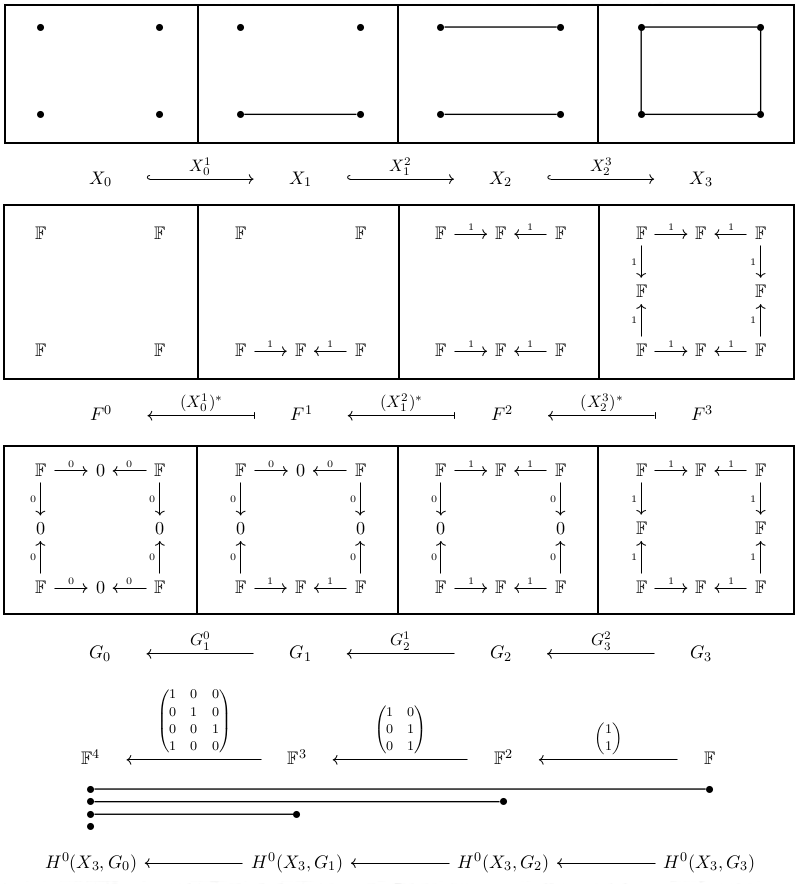}
\caption{The first row is a filtered simplicial complex $X_3$. The second row is the constant sheaf $F^3$ on $X_3$ and its pullbacks to the subcomplexes of the filtration. The third row is $\cev{G}(\vec{X},F)$. The fourth row shows $H^k(X_3,-)\circ \cev{G}(\vec{X},F)$ and the corresponding persistence barcode.}
\label{875}
\end{figure}
It is obvious that this can not hold in general. If $X_\infty$ is a one-point space, then $H^k(X_\infty,\mathcal{F}^i)=0$ for all $k>0$. On the other hand, if $\mathcal{F}$ is, for example, the constant sheaf, then $\cev{D}^k(\vec{X},\mathcal{F})$ is isomorphic to the singular cohomology persistence module of $\vec{X}$. Therefore, if we choose $\vec{X}$ in such a way that there is some non-trivial cohomology in degree $k>0$, then $\cev{D}^k(\vec{X},\mathcal{F})\ncong H^k(X_\infty,-)\circ \cev{G}(\vec{X},\mathcal{F})$. In the following we derive conditions under which a copersistence module of type T is isomorphic to a corresponding copersistence module of type~$\text{A}$.

\begin{proposition} \label{452}
Let $f\colon X\rightarrow Y$ be a continuous map, $\mathcal{F}\in\mathbf{Shv}(Y,\mathbf{M})$ and $\eta_\mathcal{F}\colon\mathcal{F}\rightarrow f_*f^*\mathcal{F}$ the sheaf morphism induced by the unit $\eta\colon\text{id}\rightarrow f_*f^*$ of the adjunction $f^*\dashv f_*$. If $f_*$ is exact, then, for every $k\in\mathbb{N}_0$, we have the following commutative diagram
\begin{equation} \label{884}
\begin{tikzcd}[column sep=large]
H^k(Y,\mathcal{F}) \arrow[r,"H^k(Y\text{,}\eta_\mathcal{F})"] \arrow[dr,swap,"H^k(f)"] &[10pt] H^k(Y,f_*f^*\mathcal{F}) \arrow[d,"\cong"] \\
& H^k(X,f^*\mathcal{F})
\end{tikzcd} \quad .
\end{equation}
\end{proposition}

\begin{proof}
Let $0\rightarrow \mathcal{F}\rightarrow \mathcal{I}^\bullet$ and $0\rightarrow f^*\mathcal{F}\rightarrow\mathcal{J}^\bullet$ be injective resolutions of $\mathcal{F}$ and $f^*\mathcal{F}$, respectively. Since $f_*$ has the exact left adjoint $f^*$ (Theorem \ref{573}) it preserves injectives \cite[Lemma 3.1]{nlab:injectives}. By assumption $f_*$ is exact. Therefore $0\rightarrow f_*f^*\mathcal{F}\rightarrow f_*\mathcal{J}^\bullet$ is an injective resolution of $f_*f^*\mathcal{F}$. The unit $\eta$ induces a cochain morphism $\eta^\bullet\colon \mathcal{I}^\bullet\rightarrow f_*f^*\mathcal{I}^\bullet$ defined by $\eta^\bullet\coloneqq (\eta_{\mathcal{I}^n})_{n\in\mathbb{N}_0}$. By Theorem \ref{171}, there exist cochain morphisms $g^\bullet\colon f^*\mathcal{I}^\bullet\rightarrow\mathcal{J}^\bullet$ and $h^\bullet\colon \mathcal{I}^\bullet\rightarrow f_*\mathcal{J}^\bullet$ lifting $\text{id}\colon f^*\mathcal{F}\rightarrow f^*\mathcal{F}$ and $\eta_\mathcal{F}\colon \mathcal{F}\rightarrow f_*f^*\mathcal{F}$, respectively. Then $f_*g^\bullet\colon f_*f^*\mathcal{I}^\bullet\rightarrow f_*\mathcal{J}^\bullet$ is a cochain morphism lifting $\text{id}\colon f_*f^*\mathcal{F}\rightarrow f_*f^*\mathcal{F}$. Moreover, $f_*g^\bullet\circ \eta^\bullet$ is a cochain morphism lifting $\eta_\mathcal{F}\colon \mathcal{F}\rightarrow f_*f^*\mathcal{F}$ and therefore, again by Theorem \ref{171}, $f_*g^\bullet\circ \eta^\bullet\simeq h^\bullet$. Since $\Gamma(Y,f_*\mathcal{J}^k)=f_*\mathcal{J}^k(Y)=\mathcal{J}^k(f^{-1}(Y))=\mathcal{J}^k(X)=\Gamma(X,\mathcal{J}^k)$ and, by the same argument, $\Gamma(Y,f_*g^\bullet)=\Gamma(X,g^\bullet)$,  we obtain the following diagram which is commutative up to homotopy 
\begin{equation} \label{863}
\begin{tikzcd}[row sep=large, column sep=large]
\Gamma(Y,\mathcal{I}^\bullet) \arrow[r,"\Gamma(Y\text{,}h^\bullet)"] \arrow[dr,swap,"\Gamma(Y\text{,}\eta^\bullet)"] & \Gamma(Y,f_*\mathcal{J}^\bullet) \arrow[r,"="] & \Gamma(X,\mathcal{J}^\bullet) \\
& \Gamma(Y,f_*f^*\mathcal{I}^\bullet) \arrow[u,swap,"\Gamma(Y\text{,}f_*g^\bullet)"{yshift=4pt}] \arrow[ur,swap,"\Gamma(X\text{,}g^\bullet)"]
\end{tikzcd} \quad .
\end{equation}
By the construction in Section \ref{710}, the morphism induced by $f$ is given by $H^k(f)\coloneqq H^k\big(\Gamma(X,g^\bullet)\big)\circ H^k\big(\Gamma(Y,\eta^\bullet)\big)$. The construction of sheaf cohomology functors up to natural isomorphism and (\ref{863}) then imply (\ref{884}).
\end{proof}

\begin{theorem} \label{112}
Let $\vec{X}\in\mathbf{Fun}(\overline{\mathbb{N}}_0,\mathbf{Top})$ be a functor such that $(X_i^j)_*$ is exact for all $i\leq j\in\overline{\mathbb{N}}_0$ and $\mathcal{F}\in\mathbf{Shv}(X_{\infty},\mathbf{Mod}_R)$. Then, for all $k\in\mathbb{N}_0$, we obtain isomorphisms of copersistence modules 
\begin{equation*} 
H^k(X_{\infty},-)\circ \cev{\mathcal{G}}(\vec{X},\mathcal{F})\cong \cev{D}^k(\vec{X},\mathcal{F}) \quad .
\end{equation*}
\end{theorem} 

\begin{proof}
For an $i\in\mathbb{N}_0$ we have $\mathcal{F}^{i+1}=(X_{i+1}^\infty)^*\mathcal{F}$ and $\eta_{i+1}\colon\mathcal{F}^{i+1}\rightarrow (X_i^{i+1})_*(X_i^{i+1})^*\mathcal{F}^{i+1}$. Since $(X_{i+1}^\infty)_*$ and $(X_i^{i+1})_*$ are exact, the statement and the proof of Proposition \ref{452} imply 
\begin{equation*} 
\begin{tikzcd}[column sep=huge,row sep=large]
H^k\big(X_{\infty},(X_{i+1}^\infty)_*\mathcal{F}^{i+1}\big) \arrow[d,"\cong"] \arrow[rr,"H^k(X_{\infty}\text{,}(X_{i+1}^\infty)_*\eta_{i+1})"] & & H^k\big(X_{\infty},(X_{i+1}^\infty)_*(X_i^{i+1})_*(X_i^{i+1})^*\mathcal{F}^{i+1}\big) \arrow[d,"\cong"] \\
H^k\big(X_{i+1},\mathcal{F}^{i+1}\big) \arrow[rr,"H^k(X_{i+1}\text{,}\eta_{i+1})"] \arrow[drr,swap,"H^k(X_i^{i+1})"] & & H^k\big(X_{i+1},(X_i^{i+1})_*(X_i^{i+1})^*\mathcal{F}^{i+1}\big) \arrow[d,"\cong"] \\
& & H^k\big(X_i,(X_i^{i+1})^*\mathcal{F}^{i+1}\big) \quad .
\end{tikzcd}
\end{equation*} 
With $\mathcal{G}_{i+1}=(X_{i+1}^\infty)_*\mathcal{F}^{i+1}$, $\mathcal{G}_i=(X_{i+1}^\infty)_*(X_i^{i+1})_*(X_i^{i+1})^*\mathcal{F}^{i+1}$ and $\mathcal{G}_{i+1}^i=(X_{i+1}^\infty)_*\eta_{i+1}$, the iterative application of this result finishes the proof.
\end{proof} 

\noindent
In other words, Theorem \ref{112} states, if $(X_i^j)_*$ is exact for all $i\leq j\in \overline{\mathbb{N}}_0$, then a copersistence module of type T is isomorphic to a corresponding copersistence module of type A. The problem is that $\cev{\mathcal{G}}(\vec{X},\mathcal{F})\in\mathbf{Fun}\big(\mathbb{N}_0^\text{op},\mathbf{Shv}(X_{\infty},\mathbf{Mod}_R)\big)$, whereas in Section \ref{129} we considered the category $\mathbf{Fun}\big(\mathbb{N}_0,\mathbf{Shv}(X_\infty,\mathbf{Mod}_R)\big)$. If $\vec{X}\in\mathbf{Fun}(\mathbf{[m]},\mathbf{Top})$ we obtain $\cev{\mathcal{G}}(\vec{X},\mathcal{F})\in\mathbf{Fun}\big(\mathbf{[m]}^\text{op},\mathbf{Shv}(X_{m-1},\mathbf{Mod}_R)\big)\cong\mathbf{Fun}\big(\mathbf{[m]},\mathbf{Shv}(X_{m-1},\mathbf{Mod}_R)\big)$ where the equivalence assigns to $\cev{\mathcal{G}}(\vec{X},\mathcal{F})$ its mirror image $\vec{\mathcal{G}}(\vec{X},\mathcal{F})\in\mathbf{Fun}\big(\mathbf{[m]},\mathbf{Shv}(X_{m-1},\mathbf{Mod}_R)\big)$ :
\begin{equation*} 
\begin{tikzcd}
\mathcal{G}_{m-1} \arrow[r,"\mathcal{G}_{m-1}^{m-2}"] & \mathcal{G}_{m-2} \arrow[r,"\mathcal{G}_{m-2}^{m-3}"] & \cdots \arrow[r,"\mathcal{G}_{2}^{1}"] & \mathcal{G}_1 \arrow[r,"\mathcal{G}_{1}^{0}"] & \mathcal{G}_0 \quad .
\end{tikzcd}
\end{equation*} 
If $R=\mathbb{F}$ and if we define $\{\{[a_i,b_i]\text{ }|\text{ } i\in I\}\}^{-1}\coloneqq \{\{[m-1-b_i,m-1-a_i]\text{ }|\text{ }i\in I\}\}$, we obtain the following corollary.

\begin{corollary} \label{139}
Let $\vec{X}\in\mathbf{Fun}(\mathbf{[m]},\mathbf{Top})$ be a functor such that $(X_i^j)_*$ is exact for every $0\leq i\leq j\leq m-1$ and $\mathcal{F}\in\mathbf{Shv}(X_{m-1},\mathbf{vec}_\mathbb{F})$. For $k\in\mathbb{N}_0$, if $\cev{D}^k(\vec{X},\mathcal{F})\in\mathbf{Fun}(\mathbb{N}_0^\text{op},\mathbf{vec}_\mathbb{F})$, then
\begin{equation*} 
\text{BC}\big(\cev{D}^k(\vec{X},\mathcal{F})\big)=\text{BC}\big(H^k(X_{m-1},-)\circ \vec{\mathcal{G}}(\vec{X},\mathcal{F})\big)^{-1} \quad .
\end{equation*}
\end{corollary}

\begin{proof}
By Theorem \ref{112}, $H^k(X_{m-1},-)\circ\cev{\mathcal{G}}(\vec{X},\mathcal{F})\cong\cev{D}^k(\vec{X},\mathcal{F})\in\mathbf{Fun}(\mathbb{N}_0^\text{op},\mathbf{vec}_\mathbb{F})$. Therefore, $H^k(X_{m-1},-)\circ \vec{\mathcal{G}}(\vec{X},\mathcal{F})\cong \bigoplus_{i\in I}\vec{I}_{[a_i,b_i]}$ implies $H^k(X_{m-1},-)\circ \cev{\mathcal{G}}(\vec{X},\mathcal{F})\cong \bigoplus_{i\in I}\cev{I}_{[m-1-b_i,m-1-a_i]}$.
\end{proof}

\noindent
Hence, in the finite case, the persistence barcode of a copersistence module of type T is the reflection of the persistence barcode of a corresponding persistence module of type A. 

Since $\mathbf{Fun}\big(\mathbb{N}_0^\text{op},\mathbf{Shv}(X,\mathbf{Mod}_R)\big) \cong \mathbf{Fun}\big(\mathbb{N}_0,\mathbf{Shv}(X,\mathbf{Mod}_R)^\text{op}\big)^\text{op}$ and $\mathbf{Shv}(X,\mathbf{Mod}_R)^\text{op}\cong \mathbf{coShv}(X,\mathbf{Mod}_R^\text{op})$ \cite{cosheaf}, the general case naturally leads us to cosheaves. The problem here is that the opposite category of modules is not necessarily a category of modules. If we consider a category of finite-dimensional vector spaces, we have the equivalence of categories $\text{dual}\colon \mathbf{vec}_\mathbb{F}^\text{op}\xrightarrow{\cong} \mathbf{vec}_\mathbb{F}$. This extends to equivalences of categories $V\colon\mathbf{Shv}(X,\mathbf{vec}_\mathbb{F})^\text{op}\xrightarrow{\cong}\mathbf{coShv}(X,\mathbf{vec}_\mathbb{F})$, defined by $V(-)\coloneqq\text{dual}\circ -$ and $\underline{V}\colon\mathbf{Fun}\big(\mathbb{N}_0^\text{op},\mathbf{Shv}(X,\mathbf{vec}_\mathbb{F})\big)^\text{op} \xrightarrow{\cong} \mathbf{Fun}\big(\mathbb{N}_0,\mathbf{coShv}(X,\mathbf{vec}_\mathbb{F})\big)$. Figure \ref{758} shows the functor $\underline{V}\big(\cev{G}(\vec{X},F)\big)$ of cellular cosheaves corresponding to the functor of cellular sheaves $\cev{G}(\vec{X},F)$ in Figure \ref{875} and the corresponding cosheaf persistence module of algebraic type $H_k(X_3,-)\circ \underline{V}\big(\cev{G}(\vec{X},F)\big)$.
\begin{figure}[h]
\centering
\includegraphics[scale=0.55]{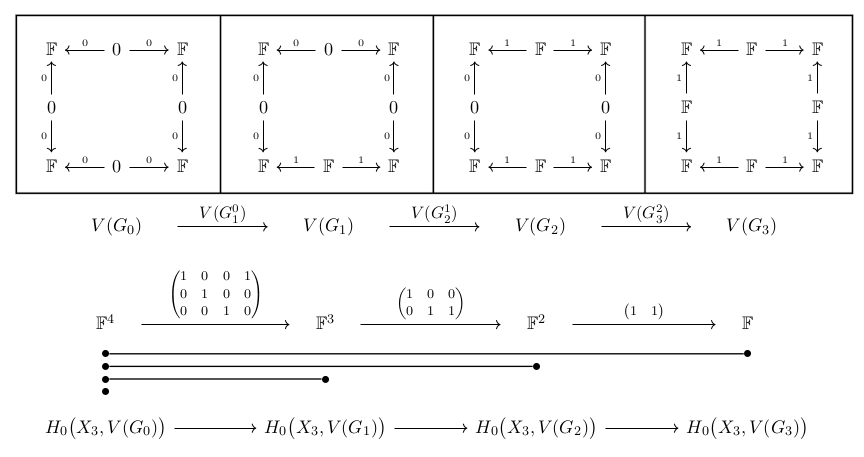}
\caption{The functor $\underline{V}\big(\cev{G}(\vec{X},F)\big)$ of cellular cosheaves corresponding to the functor of cellular sheaves $\cev{G}(\vec{X},F)$ in Figure \ref{875} and the corresponding cosheaf persistence module of algebraic type $H_k(X_3,-)\circ \underline{V}\big(\cev{G}(\vec{X},F)\big)$}
\label{758}
\end{figure} 
Note that the persistence barcodes of $H_k(X_3,-)\circ \underline{V}\big(\cev{G}(\vec{X},F)\big)$ and $H^k(X_3,-)\circ \cev{G}(\vec{X},F)$ agree. The following corollary states that this holds for any poset $X$ equipped with the Alexandrov topology. 
\begin{corollary} \label{866}
Let $\vec{X}\in\mathbf{Fun}(\overline{\mathbb{N}}_0,\mathbf{Top})$ be a functor such that $X_i$ is a poset equipped with the Alexandrov topology, $(X_i^j)_*$ is exact for all $i\leq j\in\overline{\mathbb{N}}_0$ and $\mathcal{F}\in\mathbf{Shv}(X_{\infty},\mathbf{vec}_\mathbb{F})$. Then, for all $k\in\mathbb{N}_0$,
\begin{equation*}
\cev{D}^k(\vec{X},\mathcal{F})\cong\text{dual}\Big(H_k(X_\infty,-)\circ \underline{V}\big(\cev{\mathcal{G}}(\vec{X},\mathcal{F})\big)\Big)
\end{equation*}
and, moreover,
\begin{equation*} 
\text{BC}\big(\cev{D}^k(\vec{X},\mathcal{F})\big)=\text{BC}\Big(H_k(X_\infty,-)\circ \underline{V}\big(\cev{\mathcal{G}}(\vec{X},\mathcal{F})\big)\Big) \quad .
\end{equation*}
\end{corollary}

\begin{proof}
We have the following commutative diagram (cf. Remark \ref{363}) \cite[Lemma 6.2.9]{curry}
\begin{equation*}
\begin{tikzcd}
 \mathbf{Fun}\big(\mathbb{N}_0^\text{op},\mathbf{Shv}(X,\mathbf{vec}_\mathbb{F})\big)^\text{op} \arrow[r,"\underline{V}"] \arrow[d,swap,"\underline{H}^k(X\text{,}-)"] & \mathbf{Fun}\big(\mathbb{N}_0,\mathbf{coShv}(X,\mathbf{vec}_\mathbb{F})\big) \arrow[d,"\underline{H_k}(X\text{,}-)"] \\
\mathbf{Fun}(\mathbb{N}_0^\text{op},\mathbf{vec}_\mathbb{F})^\text{op} \arrow[r,"\underline{\text{dual}}"] & \mathbf{Fun}(\mathbb{N}_0,\mathbf{vec}_\mathbb{F})
\end{tikzcd}
\end{equation*}
i.e.\ $\underline{\text{dual}}\circ \underline{H}^k(X,-)\cong \underline{H_k}(X,-)\circ \underline{V}$.  Therefore, $\cev{D}^k(\vec{X},\mathcal{F})\cong H^k(X_\infty,-)\circ \cev{\mathcal{G}}(\vec{X},\mathcal{F})$ \\ $\cong\text{dual}\Big(H_k(X_\infty,-)\circ \underline{V}\big(\cev{\mathcal{G}}(\vec{X},\mathcal{F})\big)\Big)$.
\end{proof}

\noindent
Hence, a sheaf copersistence module of topological type is dual to a corresponding cosheaf persistence module of algebraic type. This allows us to apply the cosheaf version of the results in Section \ref{129}. For example, we can compute the persistence barcode of $\cev{D}^k(\vec{X},\mathcal{F})$ by computing the homology of the cosheaf of graded modules corresponding to $V\big(\cev{\mathcal{G}}(\vec{X},\mathcal{F})\big)$.

A common scenario where the functor $f_*$ is exact is when $f$ is the inclusion of a closed subspace \cite[page 102]{iversen}. This is applicable, for example, to filtrations of simplicial complexes. If we view a simplicial complex as a topological space with the Alexandrov topology, then a subcomplex is a closed subspace. Hence, if $\vec{X}\in\mathbf{Fun}(\mathbf{[m]},\mathbf{Top})$ is a functor such that, for all $i$, $X_i$ is a simplicial complex (with the Alexandrov topology), $X_i^{i+1}$ is a simplicial inclusion and $\mathcal{F}\in\mathbf{Shv}(X_{m-1},\mathbf{Mod}_R)$, then $(X_i^{i+1})_* $ is exact and $H^k(X_{m-1},-)\circ \cev{\mathcal{G}}(\vec{X},\mathcal{F})\cong \cev{D}^k(\vec{X},\mathcal{F})$, for all $k\in\mathbb{N}_0$. We already saw an instance where this result applies in the example in Figure \ref{875}.

Given a finite filtration of simplicial complexes $\vec{X}\in\mathbf{Fun}(\mathbf{[m]},\mathbf{Top})$ and a sheaf $\mathcal{F}\in\mathbf{Shv}(X_{m-1},\mathbf{vec}_\mathbb{F})$, corollary \ref{139} and \ref{866} allow us to reduce the computation of type T persistent sheaf cohomology to the computation of type A persistent sheaf cohomology or type A persistent cosheaf homology, respectively. We want to apply the method of computation discussed in the end of Section \ref{819}. Unfortunately, by using the formula of $\eta_\mathcal{F}$ derived in the proof of Theorem \ref{546}, we obtain that, for a filtration of simplicial complexes, the linear maps $(\mathcal{G}_i^{i-1})_\sigma$, with $\mathcal{G}_i^{i-1}$ as defined above, are surjective but in general not injective. Hence, $\vec{\mathcal{G}}(\vec{X},\mathcal{F})$ not necessarily corresponds to a sheaf of free graded modules. Fortunately, the linear maps $V\big(\mathcal{G}_i^{i-1}\big)_\sigma$ are injective and $\underline{V}\big(\cev{\mathcal{G}}(\vec{X},\mathcal{F})\big)$ is a filtration of cosheaves. Therefore, there is a corresponding cosheaf of free graded modules and we can compute the persistence barcode of $\cev{D}^k(\vec{X},\mathcal{F})$ by (graded) matrix reduction. Figure \ref{214} shows the cosheaf of graded modules on $X_3$ corresponding to $V\big(\cev{\mathcal{G}}(\vec{X},\mathcal{F})\big)$, with $\cev{\mathcal{G}}(\vec{X},\mathcal{F})$ as in Figure \ref{875}. 
\begin{figure}[h]
\centering
\begin{tikzcd}[row sep=small,column sep=small]
\bullet \arrow[rr,dash,shorten <= -.5em,shorten >= -.5em,thick] &[20pt] &[20pt] \bullet &[7pt] &[7pt] \mathbb{F}[t]_0 \arrow[r,<-,"t^2"] \arrow[d,<-,swap,"t^3"] &[7pt] \mathbb{F}[t]_2 &[7pt] \arrow[l,<-,swap,"t^2"] \mathbb{F}[t]_0 \arrow[d,<-,"t^3"] \\[7pt]
 & & & & \mathbb{F}[t]_3 &  & \mathbb{F}[t]_3 \\[7pt]
\bullet \arrow[rr,dash,shorten <= -.5em,shorten >= -.5em,thick] \arrow[uu,dash,shorten <= -.5em,shorten >= -.5em,thick] & & \bullet \arrow[uu,dash,shorten <= -.5em,shorten >= -.5em,thick] & & \mathbb{F}[t]_0 \arrow[r,<-,"t"] \arrow[u,<-,"t^3"] & \mathbb{F}[t]_1 & \arrow[l,<-,swap,"t"] \mathbb{F}[t]_0 \arrow[u,<-,swap,"t^3"] 
\end{tikzcd}
\caption{The cosheaf of graded $\mathbb{F}[t]$-modules on $X_3$ corresponding to the functor $\underline{V}\big(\cev{\mathcal{G}}(\vec{X},\mathcal{F})\big)$, with $\cev{\mathcal{G}}(\vec{X},\mathcal{F})$ as in Figure \ref{875}.}
\label{214}
\end{figure} 
The homology of the cosheaf in Figure \ref{214} corresponds to the persistent cohomology of the simplicial filtration in Figure \ref{875}. Hence, these results open up the following new perspective on persistent cohomology of simplicial filtrations:
\begin{itemize}
\item A filtered simplicial complex $X$ corresponds to a (co)sheaf of graded modules on $X$.
\item The persistent simplicial cohomology of a filtered simplicial complex $X$ is given by the (co)homology of the corresponding (co)sheaf of graded modules. 
\end{itemize}

\subsection{Combination of (co)persistence modules of type A and T} \label{911}

In the previous sections we constructed two versions of sheaf (co)persistence modules. In the first version the sheaves are variable and the topological space is fixed. In the second version we have in some sense a fixed sheaf and the topological spaces are variable. In this section we show that it is possible to combine both approaches. The combined approach is based on the following theorem. This result is probably well-known but for completeness we provide a proof in the appendix.

\begin{theorem} \label{511}
Let $f\colon X\rightarrow Y$ be a continuous map and $\phi\colon\mathcal{F}\rightarrow \mathcal{G}$ a morphism of sheaves on $Y$. Then, for all $k\in\mathbb{N}_0$, the following diagram commutes
\begin{equation*} 
\begin{tikzcd}[column sep=huge]
H^k(Y,\mathcal{F}) \arrow[r,"H^k(Y\text{,}\phi)"] \arrow[d,swap,"H^k(f\text{,}\mathcal{F})"] & H^k(Y,\mathcal{G}) \arrow[d,"H^k(f\text{,}\mathcal{G})"] \\
H^k(X,f^*\mathcal{F}) \arrow[r,"H^k(X\text{,}f^*\phi)"] & H^k(X,f^*\mathcal{G}) 
\end{tikzcd}
\end{equation*}
i.e.\ the morphisms induced by continuous maps and sheaf morphisms commute.
\end{theorem} 

\begin{proof}
See Section \ref{295}.
\end{proof}

\noindent
Consider functors $\vec{X}\in\mathbf{Fun}(\overline{\mathbb{N}}_0,\mathbf{Top})$ and $\vec{\mathcal{F}}\in\mathbf{Fun}\big(\mathbb{N}_0,\mathbf{Shv}(X_\infty,\mathbf{Mod}_R)\big)$ as depicted in the following diagram
\begin{equation*} 
\begin{tikzcd}[column sep=large]
\mathcal{F}_0 \arrow[r,"\mathcal{F}_0^1"] & \mathcal{F}_1 \arrow[rr,"\mathcal{F}_1^2"] &[-32pt] &[-32pt] \mathcal{F}_2 \arrow[r,"\mathcal{F}_2^3"] & \cdots  \\[-20pt]
& & X_\infty \\[5pt]
X_0 \arrow[r,swap,"X_0^1"] \arrow[urr,"X_0^\infty"] & X_1 \arrow[rr,swap,"X_1^2"] \arrow[ur,"X_1^\infty"{yshift=-7pt}] & & X_2 \arrow[r,swap,"X_{2}^3"] \arrow[ul,"X_2^\infty"] & \cdots \arrow[ull]
\end{tikzcd} \quad .
\end{equation*}
In Section \ref{769}, we defined sheaves $\mathcal{F}^i$ on $X_i$ by pulling back a fixed sheaf $\mathcal{F}$ on $X_\infty$. In the same way we can pull back a whole functor $\vec{\mathcal{F}}$. For all $i\in\mathbb{N}_0$, we define functors $\vec{\mathcal{F}}^i\in\mathbf{Fun}\big(\mathbb{N}_0,\mathbf{Shv}(X_i,\mathbf{Mod}_R)\big)$ by $\vec{\mathcal{F}}^i\coloneqq(X_i^\infty)^*\vec{\mathcal{F}}$. This means that $(\mathcal{F}^i)_j\coloneqq (X_i^\infty)^*\mathcal{F}_j$ and $(\mathcal{F}^i)_j^{j+1}\coloneqq (X_i^\infty)^*\mathcal{F}_j^{j+1}\colon (X_i^\infty)^*\mathcal{F}_j\rightarrow (X_i^\infty)^*\mathcal{F}_{j+1}$ for all $j \in \mathbb{N}_0$. Therefore, for every $i\in\mathbb{N}_0$, we obtain the following diagram of sheaves on $X_i$ 
\begin{equation*} 
\begin{tikzcd}[column sep=large]
(\mathcal{F}^i)_0 \arrow[r,"(\mathcal{F}^i)_0^1"] & (\mathcal{F}^i)_1 \arrow[r,"(\mathcal{F}^i)_1^2"] & (\mathcal{F}^i)_2 \arrow[r,"(\mathcal{F}^i)_2^3"] & \cdots
\end{tikzcd}
\end{equation*}
and the sheaf persistence module of algebraic type $H^k(X_i,-)\circ\vec{\mathcal{F}}^i$ :
\begin{equation*} \label{400}
\begin{tikzcd}[column sep=huge]
H^k\big(X_i,(\mathcal{F}^i)_0\big) \arrow[r,"H^k(X_i\text{,}(\mathcal{F}^i)_0^1)"] &[10pt]  H^k\big(X_i,(\mathcal{F}^i)_1\big) \arrow[r,"H^k(X_i\text{,}(\mathcal{F}^i)_1^2)"] &[10pt]  H^k\big(X_i,(\mathcal{F}^i)_2\big) \arrow[r,"H^k(X_i\text{,}(\mathcal{F}^i)_2^3)"] &[10pt]  \cdots
\end{tikzcd} \, .
\end{equation*}

\noindent
Moreover, for every $j\in\mathbb{N}_0$, we obtain the formal diagram of topological spaces equipped with sheaves
\begin{equation*} 
\begin{tikzcd}[column sep=large]
(\mathcal{F}^0)_j & \arrow[l,swap,"(X_0^1)^*",maps to] (\mathcal{F}^1)_j & \arrow[l,swap,"(X_1^2)^*",maps to] (\mathcal{F}^2)_j & \arrow[l,swap,"(X_2^3)^*",maps to] \cdots \\[-20pt]
X_0 \arrow[r,"X_0^1"] & X_1 \arrow[r,"X_1^2"] & X_2 \arrow[r,"X_2^3"] & \cdots 
\end{tikzcd}
\end{equation*}
and the sheaf copersistence module of topological type  $\cev{D}^k(\vec{X},\mathcal{F}_j)$
\begin{equation*} \label{668} 
\begin{tikzcd}[column sep=huge]
H^k\big(X_{0},(\mathcal{F}^0)_j\big) &[10pt] \arrow[l,swap,"H^k(X_0^1\text{,}(\mathcal{F}^1)_j)"] H^k\big(X_{1},(\mathcal{F}^1)_j\big) &[10pt]  \arrow[l,swap,"H^k(X_1^2\text{,}(\mathcal{F}^2)_j)"] H^k\big(X_{2},(\mathcal{F}^2)_j\big) &[10pt]  \arrow[l,swap,"H^k(X_2^3\text{,}(\mathcal{F}^3)_j)"] \cdots \quad .
\end{tikzcd}
\end{equation*}
Now we combine $H^k(X_i,-)\circ\vec{\mathcal{F}}^i$ and $\cev{D}^k(\vec{X},\mathcal{F}_j)$, for all $i,j\in\mathbb{N}_0$, to obtain the following two-dimensional diagram,
\begin{equation} \label{644}
\begin{tikzcd}
\vdots \arrow[d] & \vdots \arrow[d] & \vdots \arrow[d] \\
H^k\big(X_2,(\mathcal{F}^2)_0\big) \arrow[r] \arrow[d] & H^k\big(X_2,(\mathcal{F}^2)_1\big) \arrow[r] \arrow[d] & H^k\big(X_2,(\mathcal{F}^2)_2\big) \arrow[r] \arrow[d] & \cdots \\
H^k\big(X_1,(\mathcal{F}^1)_0\big) \arrow[r] \arrow[d] & H^k\big(X_1,(\mathcal{F}^1)_1\big) \arrow[r] \arrow[d] & H^k\big(X_1,(\mathcal{F}^1)_2\big) \arrow[r] \arrow[d] & \cdots \\
H^k\big(X_0,(\mathcal{F}^0)_0\big) \arrow[r] & H^k\big(X_0,(\mathcal{F}^0)_1\big) \arrow[r] & H^k\big(X_0,(\mathcal{F}^0)_2\big) \arrow[r] &  \cdots
\end{tikzcd}
\end{equation}
which is commutative by Theorem \ref{511}. The diagram in (\ref{644}) represents a functor in $\mathbf{Fun}(\mathbb{N}_0\times\mathbb{N}_0^\text{op},\mathbf{Mod}_R)$, where $\mathbb{N}_0\times\mathbb{N}_0^\text{op}$ denotes the category corresponding to the underlying poset with the product order. 

If $\vec{X}\in\mathbf{Fun}(\mathbf{[m]},\mathbf{Top})$ and $\vec{\mathcal{F}}\in\mathbf{Fun}(\mathbf{[n]},\mathbf{Shv}(X_{m-1},\mathbf{vec}_\mathbb{F}))$, we obtain a diagram similar to (\ref{644}) of finite size. If we assume finite-dimensional cohomology, such a diagram can be viewed as a functor in $\mathbf{Fun}(\mathbf{[m]}\times\mathbf{[n]},\mathbf{vec}_\mathbb{F})$, i.e.\ a finite two-dimensional persistence module. The category $\mathbf{Fun}(\mathbf{[m]}\times\mathbf{[n]},\mathbf{vec}_\mathbb{F})$ is studied extensively in the field of multiparameter persistent homology (see for example \cite{buchet,botnan,carlsson2}) and the general methods to investigate two-dimensional persistence modules can be applied unaltered in our situation. 

If the maps $X_i^{i+1}$ satisfy the conditions of Theorem \ref{112}, the two-dimensional persistence module obtained from the combined construction can be viewed as the two-dimensional sheaf persistence module of algebraic type corresponding to a functor $\mathbf{[m]}\times \mathbf{[n]}\rightarrow \mathbf{Shv}(X_{m-1},\mathbf{vec}_\mathbb{F})$. In general, this construction is different from the straightforward generalization of sheaf (co)persistence modules to the two-dimensional case because the two directions of variation are of different nature. In one direction the topological spaces on which the sheaves are defined vary and in the other direction the sheaves on the respective topological spaces vary. Therefore, the combination of our two constructions defines a new kind of two-dimensional persistence module with a topological and an algebraic dimension. 

\section{Examples} \label{563}

\subsection{Type A} \label{334}

We discuss a construction of functors $\vec{F}\colon\mathbf{[n]}\rightarrow \mathbf{Shv}(X,\mathbf{vec}_\mathbb{F})$, compute the type A persistent sheaf cohomology of these functors and interpret the obtained persistence barcodes. 

Recall the usual pipeline of persistent homology applications: Given a data set $P$, construct a simplicial filtration over $P$, compute the persistent homology of this filtration and extract homological information from the corresponding barcode. There is a natural way to extend this method to a more general setting using sheaf theory. 

We start with the non-persistent case. Instead of just considering a simplicial complex $K$, we consider a triple $(K,L,f)$ consisting of two simplicial complexes $K$ and $L$ and a simplicial map $f\colon K\rightarrow L$. For every $\tau\in L$, we define a subcomplex of $K$ by $\tau_f\coloneqq\{\sigma\in K\text{ }|\text{ }f(\sigma)\leq \tau\}$. Note that $\tau\leq \kappa\in L$ implies $\tau_f\subseteq \kappa_f\subseteq K$. Hence, for every $n\in\mathbb{N}_0$, we can construct a cellular sheaf of $\mathbb{F}$-vector spaces $F^n_K$ on the simplicial complex $L$ that assigns each simplex $\tau\in L$ the simplicial homology of $\tau_f$:
\begin{equation*} 
\begin{aligned} 
& F^n_K(\tau)\coloneqq H_n(\tau_f,\mathbb{F}) \quad \quad \quad \quad \quad \quad \forall \tau\in L \\
& F^n_K(\tau\rightarrow \kappa)\coloneqq H_n(\tau_f\xhookrightarrow{} \kappa_f,\mathbb{F}) \quad \text{ } \forall \tau\leq\kappa\in L \, .
\end{aligned}
\end{equation*}
Since a simplicial map is determined by an assignment of vertices, one can view a map to a complete $(r-1)$-simplex as an assignment of every vertex in $K$ to one of $r$ classes. The sheaf $F_K^n$ stores the information of the $n$-dimensional homology of these classes and their unions and how the homology of different classes relate to each other. Figure \ref{964} shows the sheaf $F^n_K$ on a $0-$, $1-$ and $2-$simplex where $K_x$ denotes the set of all simplices of $K$ labeled by a letter in $x$.
\begin{figure}[h] 
\centering
\includegraphics[scale=0.4]{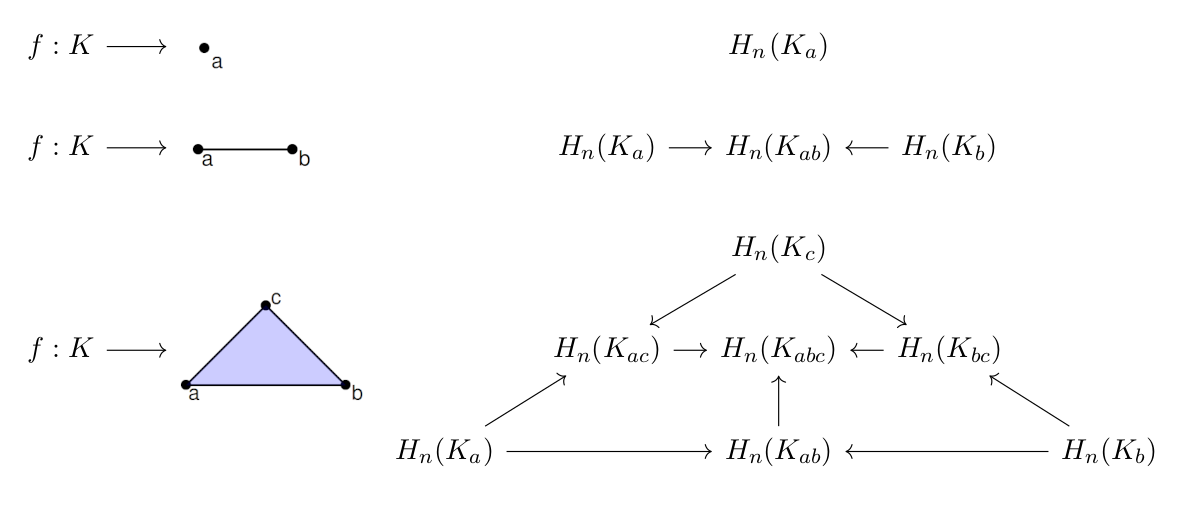}
\caption{Examples of $F^n_K$ for $L$ being a 0-, 1- or 2-simplex. Here $K_x$ denotes the set of all simplices of $K$ labeled by a letter in $x$.}
\label{964}
\end{figure} 

\noindent
Now we want to investigate the cohomology of the sheaves $F^n_K$. As an example, we consider simplicial maps of the form
\begin{figure}[t] 
\centering
\includegraphics[scale=0.37]{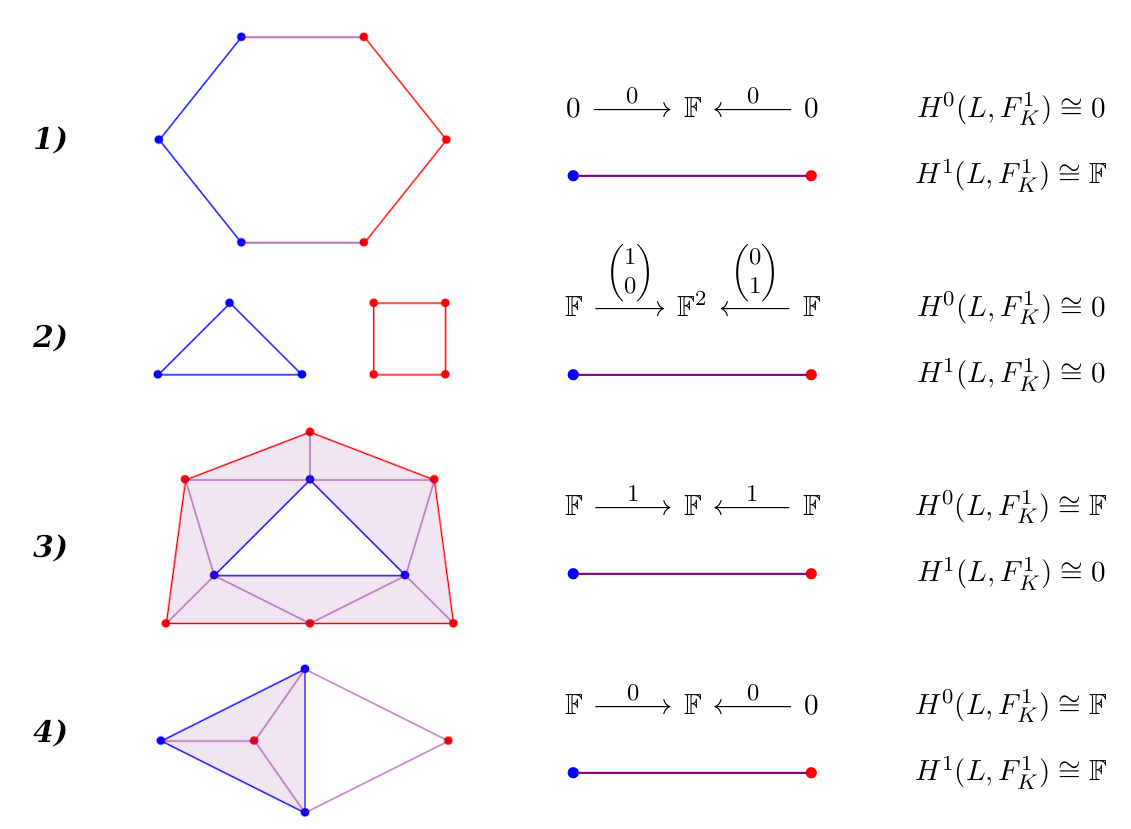}
\caption{The left column shows four instances of $(K,L,f)$, where the colors of the simplices encode the map $f$. The middle column shows the sheaves $F^1_K$ on $L$ and the right column shows their cohomology.}
\label{988}
\end{figure}
\begin{equation*} 
\begin{tikzcd} 
f\colon K \arrow[r] & L\quad = \quad \color{blue} \bullet \arrow[rr,dash,shorten <=-5.5pt,shorten >=-5.5pt,thick,violet] &[-5pt] &[-5pt] \color{red} \bullet \quad \color{black}.
\end{tikzcd}
\end{equation*}
These maps correspond to labelings of the vertices of $K$ by two labels $\{b,r\}$. Figure \ref{988} shows four examples of $(K,L,f)$. How can we interpret the cohomology of the sheaves $F^1_K$? To answer this question, we discuss the four examples in Figure \ref{988}. The sheaf $F^1_K$ collects the one-dimensional homological features of the blue and red simplices on the vertices of $L$ and the one-dimensional homological features of the whole complex $K$ on the edge of $L$. 
\begin{enumerate}
\item There are no cycles in $K_b$ or $K_r$, i.e.\ there are no cycles on the level of vertices of $L$. But if we go up to the level of edges in $L$, the red and the blue piece connect to form a cycle in $K$, i.e.\ we gain a one-dimensional feature.
\item There is a cycle in $K_b$ and a cycle in $K_r$, i.e.\ there are two cycles on the level of vertices. If we go up to the level of edges, there are still the same two cycles. So we do not loose any features on the level of vertices and do not gain any features on the level of edges.
\item There is again a cycle in $K_b$ and another cycle in $K_r$ but this time if we go from the level of vertices to the level of edges, the cycles in $K_b$ and $K_r$ merge to a single cycle in $K$. Therefore, we loose one cycle on the level of edges. 
\item There is a cycle in $K_b$ and no cycle in $K_r$ on the level of vertices of $L$. If we go to the level of edges, the blue cycle is destroyed but a new cycle consisting of red and blue points emerges, i.e.\ we loose a cycle on the level of vertices and gain a cycle on the level of edges. 
\end{enumerate}
From these observations, we obtain the following interpretation. The dimension of the zero-dimensional sheaf cohomology of $F^1_K$ is the number of cycles on the level of vertices that die if we go to the level of edges. In other words, it describes the number of cycles that exist in the individual classes $K_b$ and $K_r$ but not in $K$. The dimension of the one-dimensional sheaf cohomology describes the number of cycles we gain by going from the level of vertices to the level of edges, i.e.\ the number of cycles in $K$ that need red and blue points to get closed. 

We can now lift this construction to filtrations of simplicial complexes or, more generally, commutative diagrams of the form: 
\begin{equation} \label{793}
\begin{tikzcd} [column sep=large, row sep=large]
K_0 \arrow[r,"K_0^1"] \arrow[drr,swap,"f_0"] & K_1 \arrow[r,"K_1^2"] \arrow[dr,"f_1"] & \cdots \arrow[r,"K_{m-3}^{m-2}"] & K_{m-2} \arrow[r,"K_{m-2}^{m-1}"] \arrow[dl,swap,"f_{m-2}"] & K_{m-1} \arrow[dll,"f_{m-1}"] \\
& & L
\end{tikzcd}
\end{equation}
We construct a sheaf morphism $\phi^n_i\colon F^n_{K_{i}}\rightarrow F^n_{K_{i+1}}$ by defining $(\phi^n_i)_\tau\colon F^n_{K_{i}}(\tau)\rightarrow F^n_{K_{i+1}}(\tau)$ as  
\begin{equation*} 
(\phi^n_i)_\tau\coloneqq H_n(K_i^{i+1}|_{\tau_{f_i}},\mathbb{F})\colon H_n(\tau_{f_i},\mathbb{F})\rightarrow H_n(\tau_{f_{i+1}},\mathbb{F})
\end{equation*}
for all $\tau\in L $ and $0\leq i<m-1$. The sheaves $F_{K_i}^n$ together with the sheaf morphisms $\phi^n_i$ define the following functor in $\mathbf{Fun}\big(\mathbf{[m]},\mathbf{Shv}(L,\mathbf{vec}_\mathbb{F})\big)$ : 
\begin{equation} \label{740}
\begin{tikzcd} [column sep=large, row sep=large]
F^n_{K_0} \arrow[r,"\phi^n_0"] & F^n_{K_1} \arrow[r,"\phi^n_1"] & \cdots \arrow[r,"\phi^n_{m-3}"] & F^n_{K_{m-2}} \arrow[r,"\phi^n_{m-2}"] & F^n_{K_{m-1}} \quad .
\end{tikzcd}
\end{equation}

\noindent
A diagram as in (\ref{793}) can be obtained in the following way. Suppose we are given a labeled data set, i.e.\ a data set $P$ together with a function $f\colon P\rightarrow \{0,\ldots,r-1\}$ assigning each data point one of $r$ classes (this is very common in practice). To apply persistence theory, we construct a filtered simplicial complex on $P$ by using, for example, the $\check{\text{C}}$ech or Vietoris-Rips construction. The map $f$ then induces a simplicial map from every simplicial complex with vertex set $P$ to the complete $(r-1)$-simplex that commutes with inclusions, i.e.\ we obtain a diagram as in (\ref{793}) for the simplicial filtration on $P$. Note that, in the case $r=1$, we obtain just a single point for $L$ and the construction (\ref{740}) yields the usual persistence module of $n$-dimensional homology. In this case, the zero-dimensional persistent sheaf cohomology yields usual persistent homology. Therefore, this construction can be viewed as a generalisation of the construction of usual persistent homology. Consider the example of a simplicial filtration and a simplicial map to a $1$-simplex, encoded by the colors of the vertices, shown in Figure \ref{331}.
\begin{figure}[h] 
\centering
\includegraphics[scale=0.252]{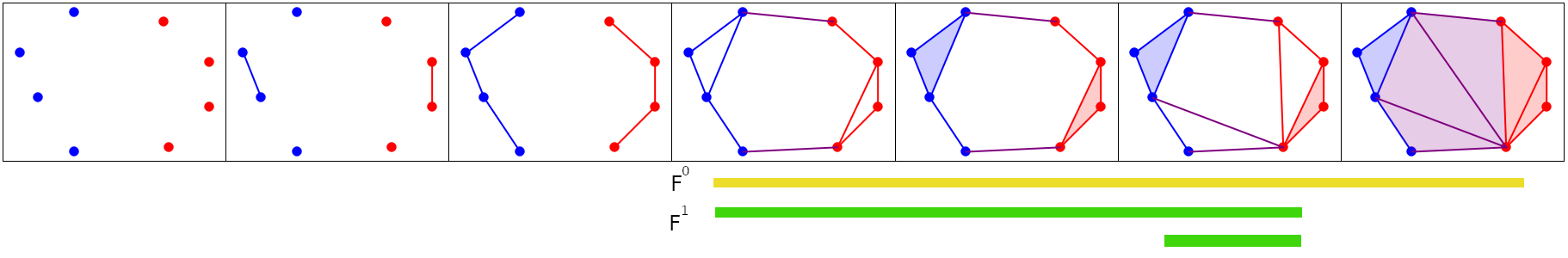}
\caption{A filtration of a labeled simplicial complex and the corresponding persistent sheaf cohomology barcodes of $F^0$ and $F^1$.}
\label{331}
\end{figure}
The explicit representations of (\ref{740}) in the case $n=0$ and $n=1$ for the filtration depicted in Figure \ref{331} are shown in Figure \ref{590} and \ref{628}, respectively, where 
\begin{equation*} \tiny 
A=\begin{pmatrix} 1 \quad 0 \quad 0 \quad 0 \quad 0 \quad 0 \quad 0 \quad 0 \\ 0 \quad 1 \quad 1 \quad 0 \quad 0 \quad 0 \quad 0 \quad 0 \\ 0 \quad 0 \quad 0 \quad 1 \quad 0 \quad 0 \quad 0 \quad 0 \\ 0 \quad 0 \quad 0 \quad 0 \quad 1 \quad 0 \quad 0 \quad 0 \\ 0 \quad 0 \quad 0 \quad 0 \quad 0 \quad 1 \quad 1 \quad 0 \\ 0 \quad 0 \quad 0 \quad 0 \quad 0 \quad 0 \quad 0 \quad 1 \end{pmatrix} \normalsize \quad\text{and}\quad \tiny \quad  B=\begin{pmatrix} 1 \quad 1 \quad 1 \quad 0 \quad 0 \quad 0 \\ 0 \quad 0 \quad 0 \quad 1 \quad 1 \quad 1 \end{pmatrix} \quad . \normalsize
\end{equation*} 
\begin{figure}[h] 
\centering
\begin{tikzcd}[row sep=large, column sep=large, every label/.append style = {font = \tiny}]
&[-35pt] \mathbb{F}^4 \arrow[r,"\begin{pmatrix} 1 \quad 0 \quad 0 \quad 0 \\ 0 \quad 1 \quad 1 \quad 0 \\ 0 \quad 0 \quad 0 \quad 1 \end{pmatrix}"] \arrow[d,"\begin{pmatrix} \text{id} \\ 0 \end{pmatrix}"] & \mathbb{F}^3 \arrow[r,"\begin{pmatrix} 1 \quad 1 \quad 1 \end{pmatrix}"] \arrow[d,"\begin{pmatrix} \text{id} \\ 0 \end{pmatrix}"] & \mathbb{F} \arrow[r,"1"] \arrow[d,"\begin{pmatrix} 1 \\ 0 \end{pmatrix}"] & \mathbb{F} \arrow[r,"1"] \arrow[d,"1"] & \mathbb{F} \arrow[r,"1"] \arrow[d,"1"] & \mathbb{F} \arrow[r,"1"] \arrow[d,"1"] & \mathbb{F} \arrow[d,"1"] \\
F^0_{K_{-}}: & \mathbb{F}^8 \arrow[r,"A"] & \mathbb{F}^6 \arrow[r,"B"] & \mathbb{F}^2 \arrow[r,"\begin{pmatrix} 1 \quad 1 \end{pmatrix}"] & \mathbb{F} \arrow[r,"1"] & \mathbb{F} \arrow[r,"1"] & \mathbb{F} \arrow[r,"1"] & \mathbb{F} \\
& \mathbb{F}^4 \arrow[r,"\begin{pmatrix} 1 \quad 0 \quad 0 \quad 0 \\ 0 \quad 1 \quad 1 \quad 0 \\ 0 \quad 0 \quad 0 \quad 1 \end{pmatrix}"] \arrow[u,swap,"\begin{pmatrix} 0 \\ \text{id} \end{pmatrix}"{yshift=7.5pt}] & \mathbb{F}^3 \arrow[r,"\begin{pmatrix} 1 \quad 1 \quad 1 \end{pmatrix}"] \arrow[u,swap,"\begin{pmatrix} 0 \\ \text{id} \end{pmatrix}"] & \mathbb{F} \arrow[r,"1"] \arrow[u,swap,"\begin{pmatrix} 0 \\ 1 \end{pmatrix}"] & \mathbb{F} \arrow[r,"1"] \arrow[u,swap,"1"] & \mathbb{F} \arrow[r,"1"] \arrow[u,swap,"1"] & \mathbb{F} \arrow[r,"1"] \arrow[u,swap,"1"] & \mathbb{F} \arrow[u,swap,"1"] \\[-10pt]
H^0(L,-): & 0 \arrow[r,"0"] & 0 \arrow[r,"0"] & 0 \arrow[r,"0"] & \mathbb{F} \arrow[r,"1"] & \mathbb{F} \arrow[r,"1"] & \mathbb{F} \arrow[r,"1"] & \mathbb{F} \\[-30pt]
& & & & \bullet \arrow[rrr,dash,thick,shorten <= -.5em,shorten >= -.5em] & & & \bullet \\[-20pt] 
\end{tikzcd}
\caption{The sheaves $F^0_{K_{-}}$ on the simplicial filtration depiced in Figure \ref{331} and the corresponding persistence barcode.}
\label{590}
\end{figure}
\begin{figure}[h] 
\centering
\begin{tikzcd}[row sep=large, column sep=large, every label/.append style = {font = \tiny}]
&[-35pt] 0 \arrow[r,"0"] \arrow[d,"0"] & 0 \arrow[r,"0"] \arrow[d,"0"] & 0 \arrow[r,"0"] \arrow[d,"0"] & \mathbb{F} \arrow[r,"0"] \arrow[d,"\begin{pmatrix} 1 \\ 0 \\ 0 \end{pmatrix}"] & 0 \arrow[r,"0"] \arrow[d,"0"] & 0\arrow[r,"0"] \arrow[d,"0"] & 0 \arrow[d,"0"] \\
F^1_{K_{-}}: & 0 \arrow[r,"0"] & 0 \arrow[r,"0"] & 0 \arrow[r,"0"] & \mathbb{F}^3 \arrow[r,"\begin{pmatrix} 0 \quad 1 \quad 0 \end{pmatrix}"] & \mathbb{F} \arrow[r,"\begin{pmatrix} 1\\ 1\\1\end{pmatrix}"] & \mathbb{F}^3 \arrow[r,"0"] & 0 \\
& 0 \arrow[r,"0"] \arrow[u,swap,"0"] & 0 \arrow[r,"0"] \arrow[u,swap,"0"] & 0 \arrow[r,"0"] \arrow[u,swap,"0"] & \mathbb{F} \arrow[r,"0"] \arrow[u,swap,"\begin{pmatrix} 0 \\ 0 \\ 1 \end{pmatrix}"] & 0 \arrow[r,"0"] \arrow[u,swap,"0"] & \mathbb{F} \arrow[r,"0"] \arrow[u,swap,"\begin{pmatrix} 0 \\ 0 \\ 1 \end{pmatrix}"] & 0 \arrow[u,swap,"0"] \\
H^1(L,-): & 0 \arrow[r,"0"] & 0 \arrow[r,"0"] & 0 \arrow[r,"0"] & \mathbb{F} \arrow[r,"1"] & \mathbb{F} \arrow[r,"\begin{pmatrix} 1\\1\end{pmatrix}"] & \mathbb{F}^2 \arrow[r,"0"] & 0 \\[-30pt]
& & & & \bullet \arrow[rr,dash,thick,shorten <= -.5em,shorten >= -.5em] & & \bullet \\[-35pt]
& & & & & & \bullet
\end{tikzcd}
\caption{The sheaves $F^1_{K_{-}}$ on the simplicial filtration depiced in Figure \ref{331} and the corresponding persistence barcode.}
\label{628}
\end{figure}
\noindent
The barcode corresponding to the zero(one)-dimensional persistent sheaf cohomology of the diagram of sheaves in Figure \ref{590}(\ref{628}) is depicted in yellow(green) in Figure \ref{331}. We can interpret the barcode in the following way. There is a persistent component that would decompose into two components if we would not consider edges between points of different colors. There is a persistent cycle that would not exist if we would not consider edges between red and blue points. In other words, the long bars describe the persistent mixed features in the filtration. 

\subsection{Type T} \label{741}  

Consider a simplicial complex $K$ with the following filtration 
\begin{equation} \label{753}
\begin{tikzcd}
K_0 \arrow[r,hook,"\iota_0"] & K_1 \arrow[r,hook,"\iota_1"] & \cdots \arrow[r,hook,"\iota_{m-3}"] & K_{m-2} \arrow[r,hook,"\iota_{m-2}"] & K_{m-1}=K \quad .
\end{tikzcd}
\end{equation}
In this section, we want to discuss a construction of a non-constant cellular sheaf $F$ on $K$, compute the type T persistent sheaf cohomology of $F$ on (\ref{753}) and give an interpretation of the corresponding persistence barcode. We again assume that we have some additional information in the form of a simplicial map $f\colon K\rightarrow L=\begin{tikzcd} \color{blue} \bullet \arrow[rr,dash,shorten <=-5.7pt,shorten >=-5.5pt,violet,thick] &[-5pt] &[-5pt] \color{red} \bullet \color{black}\end{tikzcd}$. As in the previous section, imagine that the filtration in (\ref{753}) comes from a Vietoris-Rips construction on a data set $P$ and the simplicial map $f$ comes from a labeling of the data set by two labels $P\rightarrow \{b,r\}$.  Now consider the following sheaf $F$ on $L$ 
\begin{equation} \label{349}
\begin{tikzcd}[row sep=tiny, every label/.append style = {font = \tiny}]
\mathbb{F} \arrow[r,"\begin{pmatrix} 1 \\ 0 \end{pmatrix}"] & \mathbb{F}^2 & \arrow[l,swap,"\begin{pmatrix} 0 \\ 1 \end{pmatrix}"] \mathbb{F} \\
\color{blue} \bullet \arrow[rr,dash,swap,thick,shorten <= -.6em,shorten >= -.6em,violet,thick] & & \color{red} \bullet
\end{tikzcd}
\end{equation} 
A property of this sheaf is that it has no non-zero global sections.
We can extend (\ref{753}) by $L$ and apply the construction of copersistence modules of type T
\begin{equation} \label{312}
\begin{tikzcd}
F^0 & \arrow[l,swap,"\iota_0^*",mapsto] F^1 & \arrow[l,swap,"\iota_1^*",mapsto] \cdots & \arrow[l,swap,"\iota_{m-3}^*",mapsto] F^{m-2} & \arrow[l,swap,"\iota_{m-2}^*",mapsto] F^{m-1} & \arrow[l,swap,"f^*",mapsto] F \\[-20pt]
K_0 \arrow[r,hook,"\iota_0"] & K_1 \arrow[r,hook,"\iota_1"] & \cdots \arrow[r,hook,"\iota_{m-3}"] & K_{m-2} \arrow[r,hook,"\iota_{m-2}"] & K_{m-1} \arrow[r,"f"] & L
\end{tikzcd}
\end{equation}
where we only consider the sheaves $F^i$ on $K_i$. On all the edges in $K_i$ of the form $\begin{tikzcd} \color{blue} \bullet \arrow[rr,dash,shorten <=-6pt,shorten >=-6pt,blue,thick] &[-5pt] &[-5pt] \color{blue} \bullet \end{tikzcd}$ or $\begin{tikzcd} \color{red} \bullet \arrow[rr,dash,shorten <=-6pt,shorten >=-6pt,red,thick] &[-5pt] &[-5pt] \color{red} \bullet \end{tikzcd}$ the sheaf $F^i$ is just the constant sheaf. On a purple edge $F^i$ is equal to (\ref{349}). This means that the sheaf $F^i$ restricted to a component of $K_i$ that contains points with label $b$ and points with label $r$ has no global sections. On the other hand, if a component has only points labeld $b$ or $r$ the sheaf $F^i$ restricted to $K_i$ is the constant sheaf and the space of global sections of the constant sheaf restricted to a component is one-dimensional. Hence, for a component $C\subseteq K_i$, we obtain
\begin{equation*} 
\text{dim } H^0(C,F^i|_{C})=\begin{cases} 0 \text{ if } C \text{ contains b and r points} \\ 1 \text{ else} \end{cases} \quad .
\end{equation*}
Therefore, the zero-dimensional persistent sheaf cohomology describes the persistent unicolored components in the labeled filtration (\ref{753}). The example in Figure \ref{983} shows an instance of this construction. 
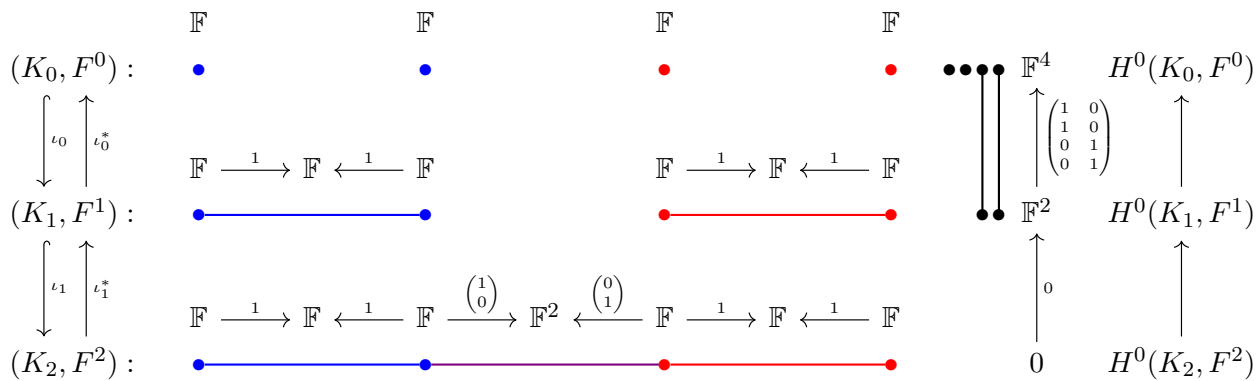
\begin{figure}[h]
\centering 
\begin{tikzcd}[every label/.append style = {font = \tiny}]
&[-15pt] \mathbb{F} & & \mathbb{F} & & \mathbb{F} & & \mathbb{F} \\[-20pt]
(K_0,F^0): \arrow[dd,hook,"\iota_0",shift left=-4] &[-15pt] \color{blue} \bullet & & \color{blue} \bullet &[10pt] &[10pt] \color{red} \bullet & & \color{red} \bullet &[-20pt] \bullet &[-35pt] \bullet &[-35pt] \bullet &[-35pt] \bullet &[-30pt] \mathbb{F}^4 &[-15pt] H^0(K_0,F^0) \\
& \mathbb{F} \arrow[r,"1"] & \mathbb{F} & \arrow[l,swap,"1"] \mathbb{F} & & \mathbb{F} \arrow[r,"1"] & \mathbb{F} & \arrow[l,swap,"1"] \mathbb{F} \\[-20pt]
(K_1,F^1): \arrow[dd,hook,"\iota_1",shift left=-4] \arrow[uu,shift left=-2,swap,"\iota_0^*"] & \color{blue} \bullet \arrow[rr,dash,thick,shorten <= -.5em,shorten >= -.5em,blue] & & \color{blue} \bullet & & \color{red} \bullet \arrow[rr,dash,thick,shorten <= -.5em,shorten >= -.5em,red] & & \color{red} \bullet &  & & \bullet \arrow[uu,dash,thick,shorten <= -.5em,shorten >= -.5em] & \bullet \arrow[uu,dash,thick,shorten <= -.5em,shorten >= -.5em] & \mathbb{F}^2 \arrow[uu,swap,"\begin{pmatrix}1\quad 0\\1\quad 0\\0\quad 1\\0\quad 1\end{pmatrix}"] & H^0(K_1,F^1) \arrow[uu] \\
& \mathbb{F} \arrow[r,"1"] & \mathbb{F} & \arrow[l,swap,"1"] \mathbb{F} \arrow[r,"\begin{pmatrix}1\\0\end{pmatrix}"] & \mathbb{F}^2 & \arrow[l,swap,"\begin{pmatrix}0\\1\end{pmatrix}"] \mathbb{F} \arrow[r,"1"] & \mathbb{F} & \arrow[l,swap,"1"] \mathbb{F} \\[-20pt]
(K_2,F^2): \arrow[uu,shift left=-2,swap,"\iota_1^*"] & \color{blue} \bullet \arrow[rr,dash,thick,shorten <= -.5em,shorten >= -.5em,blue] & & \color{blue} \bullet \arrow[rr,dash,thick,shorten <= -.5em,shorten >= -.5em,violet] & & \color{red} \bullet \arrow[rr,dash,thick,shorten <= -.5em,shorten >= -.5em,red] & & \color{red} \bullet &  &  &  &  & 0 \arrow[uu,swap,"0"] & H^0(K_2,F^2) \arrow[uu] 
\end{tikzcd}
\caption{The left and middle columns show the sheaves of construction (\ref{312}) on a colored simplicial filtration. The right column shows the corresponding copersistence module and barcode.}
\label{983}
\end{figure}
In $K_0$ there are $4$ unicolored components. After some of them merge in $K_1$, there are $2$ unicolored components left. In $K_2$ all components merge such that there is just a single component left but this component is no longer unicolored. Therefore, there is no zero-dimensional sheaf cohomology class in $F^2$.

\section{Conclusion} \label{244}

The main contribution of this work is to establish a rigorous basis of the theory of persistent sheaf cohomology. We distinguish between two different versions of persistent sheaf cohomology. The first version measures the persistence of sheaf cohomology under the variation of sheaves on a fixed topological space. Most notably, we generalized the ZC-representation theorem to sheaves and outlined a method for computing persistent sheaf cohomology in an efficient way by computing the cohomology of a sheaf of graded modules. The second version measures the persistence of sheaf cohomology of a (in some sense) fixed sheaf under the variation of the topological spaces on which this sheaf is defined. We showed that in certain scenarios one can reduce the second construction to the first one. This result, applied to constant sheaves on filtered simplicial complexes, suggests viewing filtered simplicial complexes as (co)sheaves of graded modules and persistent simplicial cohomology as the (co)homology of this (co)sheaf. Moreover, we showed that we can combine both constructions to obtain a new kind of two-dimensional persistence module with an algebraic and a topological dimension.

An important question that is not addressed here is the question of stability. That means the question if one can define a notion of stability for persistent sheaf cohomology. We plan to address this question in further work. We also do not address the question of applications. In Section \ref{563}, we describe how to obtain instances of both constructions of sheaf (co)persistence modules from labeled data sets. Since labeled data sets are very common in practice a possible direction for further research would be to investigate if persistent sheaf cohomology is a useful invariant of labeled data sets.

From our point of view, the greatest difficulty of applied sheaf theory in general is to construct interesting sheaves. Many of the established practical applications of sheaf theory focus on sheaves on graphs. If we want to construct sheaves on higher-dimensional simplicial complexes, we have to satisfy a lot of commutativity constraints. The problem is that ``the real world" might not provide this commutativity. Therefore, it is natural to look for constructions, similar in spirit to the approximation of topological spaces or data sets by simplicial complexes, for approximating objects of interest by sheaves. Given such a hypothetical parameter-dependent construction, persistent sheaf cohomology could be an interesting tool to investigate the structure of the obtained sheaf models. In this way, we could extend the cohomological investigation of structures far beyond the study of holes in simplicial complexes.   \newline

\noindent
\textbf{Acknowledgements} \newline
\noindent
The author is supported by the Austrian Science Fund (FWF): W1230. The author thanks Michael Kerber, Barbara Giunti and Jan Jendrysiak for useful discussions.

\bibliographystyle{plain}

\appendix

\section{Appendix} \label{759}

\subsection{Cellular sheaves on simplicial complexes} \label{684}

In this section we discuss how to interpret abstract simplicial complexes as topological spaces with the Alexandrov topology. For a general reference on cellular sheaf theory see \cite{curry}.

Let $X$ be an abstract simplicial complex. Define the following order relation 
\begin{equation*} 
\sigma\leq \tau \mathrel{\vcentcolon\Longleftrightarrow} \sigma\subseteq \tau \quad \forall \sigma,\tau\in X
\end{equation*}
i.e.\ a simplex $\sigma$ in $X$ is smaller or equal than a simplex $\tau$ in $X$ if and only if $\sigma$ is a face of $\tau$. The set $X$ together with this order relation forms a poset $(X,\leq)$ called the face relation poset of the simplicial complex. Figure \ref{849} shows the face relation poset of a $2$-simplex.
\begin{figure}[h] 
\centering
\includegraphics[scale=0.25]{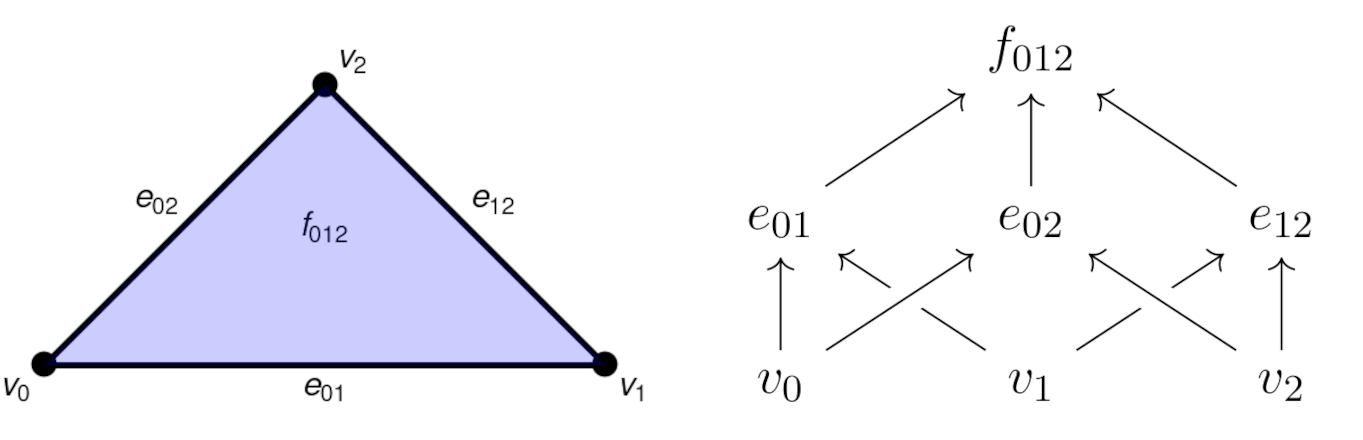}
\caption{Face relation poset of $2$-simplex.}
\label{849}
\end{figure} 

\begin{definition}[Alexandrov topology] \label{913} \cite[Definition 4.2.2]{curry}
Let $(X,\leq)$ be a partially ordered set. Define the \emph{Alexandrov topology} to be the topology on $X$ whose open sets are the sets $U\subseteq X$ that satisfy the following property: 
\begin{equation*} 
x\in U \text{ and } x\leq y \implies y\in U \, .
\end{equation*}
The open sets $U_x\coloneqq \{y\in X \text{ }|\text{ } x\leq y\}$ form a basis of the Alexandrov topology. 
\end{definition} 

\noindent
Now we can view a simplicial complex $X$ as a topological space with the Alexandrov topology of the face relation poset. Figure \ref{314} shows the Alexandrov basic open sets of a $2$-simplex as defined in \ref{913}.
\begin{figure}[h] 
\centering
\includegraphics[scale=0.5]{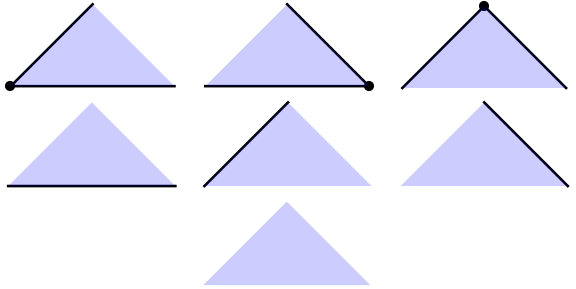}
\caption{Alexandrov basic open sets of the $2$-simplex.}
\label{314}
\end{figure} 
The basic open set corresponding to a simplex is the set of all its cofaces. Viewing a simplicial complex as a topological space allows us to interpret properties of abstract simplicial complexes in terms of the topology on $X$. For example, one can easily show that:
\begin{itemize}
\item $X$ is connected as a simplicial complex if and only if it is connected as a topological space.
\item A subcomplex $Y\subseteq X$ is a closed subspace.
\item A simplicial map $f\colon X\rightarrow Y$ is a continuous map with respect to the Alexandrov topology on $X$ and $Y$.
\end{itemize}  

\noindent
We can now consider sheaves $\mathcal{F}\colon\mathbf{Open}(X)^{op}\rightarrow\mathbf{M}$  on the abstract simplicial complex $X$ viewed as a topological space. We can restrict these functors to the full subcategory $\mathbf{bOpen}(X)^{op}$ of basic open subsets of $X$ as defined in \ref{913}. We denote by $\mathbf{X}$ the category corresponding to the face relation poset of $X$. It is easy to see that $\mathbf{bOpen}(X)^{op}\cong \mathbf{X}$. Hence, a sheaf $\mathcal{F}$ on $X$ defines a functor $F\colon\mathbf{X}\rightarrow \mathbf{M}$ such that $F(x)=\mathcal{F}(U_x)$ for all $x\in X$. One can show that conversely every functor $F\colon\mathbf{X}\rightarrow \mathbf{M}$, i.e.\ every functor defined on the basis of the Alexandrov topology, can be uniquely extended to a sheaf on $X$ and $\mathbf{Fun}(\mathbf{X},\mathbf{M})\cong\mathbf{Shv}(X,\mathbf{M})$ \cite[Theorem 4.2.10]{curry}. This implies that we can define a cellular sheaf on a simplicial complex in the following way.
\begin{definition}[Cellular sheaf on simplicial complex] \label{362}
A \emph{cellular sheaf on a simplicial complex} $X$ is a functor $F\colon\mathbf{X}\rightarrow \mathbf{M}$, i.e.\ an assignment of an object $F(\sigma)\in\mathbf{M}$ to each simplex $\sigma\in X$ and a morphism $F(\sigma\rightarrow\tau)\in\text{Hom}(F(\sigma),F(\tau))$ to each incident pair of simplices $\sigma\leq\tau\in X$ subject to commutativity requirements. 
\end{definition}
\noindent
Figure \ref{670} shows an example of a cellular sheaf of $\mathbb{F}$-vector spaces on a $2$-simplex.
\begin{figure}[h] 
\centering
\begin{tikzcd}[column sep=large, row sep=large,every label/.append style = {font = \scriptsize}]
& & \mathbb{F}^2 \arrow[dl,swap,"\begin{pmatrix}  1 \quad 0 \\  0 \quad 1 \end{pmatrix}"] \arrow[dr,"\begin{pmatrix}  1 \quad 0 \\  0 \quad 1 \end{pmatrix}"] &  \\
& \mathbb{F}^2 \arrow[r,"\begin{pmatrix}  1 \quad 0 \end{pmatrix}"] & \mathbb{F} & \mathbb{F}^2 \arrow[l,swap,"\begin{pmatrix}  1 \quad 0 \end{pmatrix}"] \\
\mathbb{F}^2 \arrow[ur,"\begin{pmatrix}  0 \quad 1 \\  1 \quad 0 \end{pmatrix}"] \arrow[rr,"\begin{pmatrix}  1 \quad 0 \\  0 \quad 1 \end{pmatrix}"] & & \mathbb{F}^2 \arrow[u,swap,"\begin{pmatrix} 0 \quad 1 \end{pmatrix}"] & & \mathbb{F}^2 \arrow[ul,swap,"\begin{pmatrix}  0 \quad 1 \\  1 \quad 0 \end{pmatrix}"] \arrow[ll,swap,"\begin{pmatrix}  1 \quad 0 \\  0 \quad 1 \end{pmatrix}"]
\end{tikzcd}
\caption{Sheaf of $\mathbb{F}$-vector spaces on $2$-simplex}
\label{670}
\end{figure}
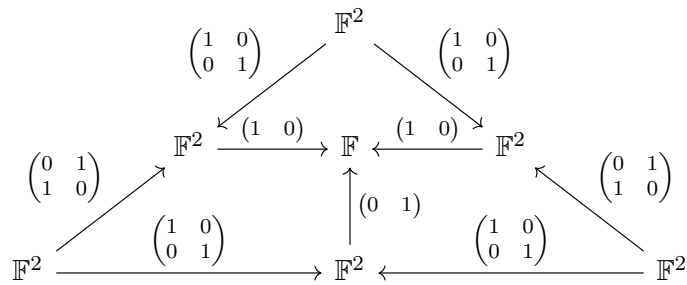
Note that, since different paths in the poset in Figure \ref{849} commute, the linear maps in Figure \ref{670} also have to commute.

\subsection{Proof of Proposition \ref{635}} \label{702}

\begin{proposition} 
The following diagram commutes 
\begin{equation*} 
\begin{tikzcd}
\mathbf{coCh}(\mathbf{grMod}_{R[t]}) \arrow[r,shift left=1,"\underline{\epsilon}"] \arrow[d,swap,"H^k"] &[10pt] \arrow[l,shift left=1,"\underline{\eta}"] \mathbf{coCh}\big(\mathbf{Fun}(\mathbb{N}_0,\mathbf{Mod}_R)\big) \arrow[r,shift left=1,"cr^{-1}"] & \arrow[l,shift left=1,"cr"] \mathbf{Fun}\big(\mathbb{N}_0,\mathbf{coCh}(\mathbf{Mod}_R)\big) \arrow[d,"\underline{H}^k"]  \\
\mathbf{grMod}_{R[t]} \arrow[rr,shift left=1,"\epsilon"] && \arrow[ll,shift left=1,"\eta"] \mathbf{Fun}(\mathbb{N}_0,\mathbf{Mod}_R)
\end{tikzcd} 
\end{equation*}
i.e.\ there are natural isomorphisms $\eta\circ \underline{H}^k\cong H^k\circ \underline{\eta}\circ cr$ and $\epsilon\circ H^k\cong \underline{H}^k\circ cr^{-1}\circ\underline{\epsilon}$ where $cr$ and $cr^{-1}$ are obtained by currying \cite{nlab:currying}. 
\end{proposition}

\begin{proof}
Let $g^\bullet\colon C^\bullet\rightarrow D^\bullet$ be the following morphism in $\mathbf{coCh}(\mathbf{grMod}_{R[t]})$ 
\begin{equation*}
\begin{tikzcd}[row sep=large,column sep=large]
0 \arrow[r] & \underset{n\in\mathbb{N}_0}{\bigoplus}C_n^0 \arrow[r,"\underset{n\in\mathbb{N}_0}{\bigoplus}\delta^0_n\circ p_n"] \arrow[d,"\underset{n\in\mathbb{N}_0}{\bigoplus}g_n^0\circ p_n"{yshift=5pt}] & \underset{n\in\mathbb{N}_0}{\bigoplus}C_n^1 \arrow[r,"\underset{n\in\mathbb{N}_0}{\bigoplus}\delta^1_n\circ p_n"] \arrow[d,"\underset{n\in\mathbb{N}_0}{\bigoplus}g_n^1\circ p_n"{yshift=5pt}] & \underset{n\in\mathbb{N}_0}{\bigoplus}C_n^2 \arrow[r,"\underset{n\in\mathbb{N}_0}{\bigoplus}\delta^2_n\circ p_n"] \arrow[d,"\underset{n\in\mathbb{N}_0}{\bigoplus}g_n^2\circ p_n"{yshift=5pt}] & \cdots \\
0 \arrow[r] & \underset{n\in\mathbb{N}_0}{\bigoplus}D_n^0 \arrow[r,"\underset{n\in\mathbb{N}_0}{\bigoplus}\partial^0_n\circ p_n"] & \underset{n\in\mathbb{N}_0}{\bigoplus}D_n^1 \arrow[r,"\underset{n\in\mathbb{N}_0}{\bigoplus}\partial^1_n\circ p_n"] & \underset{n\in\mathbb{N}_0}{\bigoplus}D_n^2 \arrow[r,"\underset{n\in\mathbb{N}_0}{\bigoplus}\partial^2_n\circ p_n"] & \cdots
\end{tikzcd} \quad .
\end{equation*}
By applying $H^k$, we obtain a morphisms of graded $R[t]$-modules 
\begin{equation*}
\begin{tikzcd}
H^k(g^\bullet)\colon &[-30pt] H^k(C^\bullet) \arrow[r] & H^k(D^\bullet) \\[-23pt]
\verteq & \verteq & \verteq \\[-23pt]
\underset{n\in\mathbb{N}_0}{\bigoplus}H^k(g^\bullet_n)\circ p_n\colon &[-20pt] \underset{n\in\mathbb{N}_0}{\bigoplus}H^k(C^\bullet_n) \arrow[r] & \underset{n\in\mathbb{N}_0}{\bigoplus}H^k(D^\bullet_n) 
\end{tikzcd} 
\end{equation*}
where $t\cdot\coloneqq H^k(t\cdot)\colon H^k(C_n)\rightarrow H^k(C_{n+1})$. On the other hand, applying $cr^{-1}\circ\underline{\epsilon}$ yields the following morphism $cr^{-1}\circ\underline{\epsilon}(g^\bullet)\colon cr^{-1}\circ\underline{\epsilon}(C^\bullet)\rightarrow cr^{-1}\circ \underline{\epsilon}(D^\bullet)$
\begin{equation*}
\begin{tikzcd}[row sep=small,column sep=small]
&& \text{ } &[-15pt] && \iddots &&& \iddots &&& \iddots \\[5pt]
& 0 \arrow[rrr] &&& C_1^0 \arrow[rrr,"\delta_1^0"] \arrow[ddd,swap,"g_1^0"{yshift=-5pt}] \arrow[ur,"t\cdot"] &&& C_1^1 \arrow[rrr,"\delta_1^1"] \arrow[ddd,swap,"g_1^1"{yshift=-5pt}] \arrow[ur,"t\cdot"] &&& C_1^2 \arrow[rrr,"\delta_1^2"] \arrow[ddd,swap,"g_1^2"{yshift=-5pt}] \arrow[ur,"t\cdot"] &&& \cdots  \\[5pt]
0 \arrow[rrr,crossing over] &&& C_0^0 \arrow[rrr,crossing over,"\delta_0^0"] \arrow[ur,"t\cdot"] &&& C_0^1 \arrow[rrr,crossing over,"\delta_0^1"] \arrow[ur,"t\cdot"] &&& C_0^2 \arrow[rrr,crossing over,"\delta_0^2"] \arrow[ur,"t\cdot"] &&& \cdots \\[-20pt] 
&& \text{ } &&& \iddots &&& \iddots &&& \iddots \\[5pt]
& 0 \arrow[rrr] &&& D_1^0 \arrow[rrr,"\partial_1^0"] \arrow[ur,"t\cdot"] &&& D_1^1 \arrow[rrr,"\partial_1^1"] \arrow[ur,"t\cdot"] &&& D_1^2 \arrow[rrr,"\partial_1^2"] \arrow[ur,"t\cdot"] &&& \cdots \\[5pt]
0 \arrow[rrr] &&& D_0^0 \arrow[rrr,"\partial_0^0"] \arrow[uuu,<-,crossing over,swap,"g_0^0"{yshift=3pt}] \arrow[ur,"t\cdot"] &&& D_0^1 \arrow[rrr,"\partial_0^1"] \arrow[uuu,<-,crossing over,swap,"g_0^1"{yshift=3pt}] \arrow[ur,"t\cdot"] &&& D_0^2 \arrow[rrr,"\partial_0^2"] \arrow[uuu,<-,crossing over,swap,"g_0^2"{yshift=3pt}] \arrow[ur,"t\cdot"] &&& \cdots 
\end{tikzcd} \quad .
\end{equation*}
Hence, applying $\underline{H}^k$ to $cr^{-1}\circ\underline{\epsilon}(g^\bullet)\colon cr^{-1}\circ\underline{\epsilon}(C^\bullet)\rightarrow cr^{-1}\circ \underline{\epsilon}(D^\bullet)$ as well as applying $\epsilon$ to $H^k(g^\bullet)\colon H^k(C^\bullet)\rightarrow H^k(D^\bullet)$ yields
\begin{equation*}
\begin{tikzcd}[column sep=large]
H^k(C_0^\bullet) \arrow[r,"H^k(t \cdot)"] \arrow[d,"H^k(g_0^\bullet)"] & H^k(C_1^\bullet) \arrow[r,"H^k(t\cdot )"] \arrow[d,"H^k(g_1^\bullet)"] & H^k(C_2^\bullet) \arrow[r,"H^k(t\cdot)"] \arrow[d,"H^k(g_2^\bullet)"] & \cdots \\
H^k(D_0^\bullet) \arrow[r,"H^k(t\cdot)"] & H^k(D_1^\bullet) \arrow[r,"H^k(t\cdot)"] & H^k(D_2^\bullet) \arrow[r,"H^k(t\cdot)"] & \cdots
\end{tikzcd} \quad .
\end{equation*}
Therefore, we obtain  
\begin{equation*}
\begin{tikzcd}
\epsilon\circ H^k(C^\bullet) \arrow[r,"\cong"] \arrow[d,swap,"\epsilon\circ H^k(g^\bullet)"] & \underline{H}^k\circ cr^{-1}\circ\underline{\epsilon}(C^\bullet) \arrow[d,"\underline{H}^k\circ cr^{-1}\circ\underline{\epsilon}(g^\bullet)"] \\
\epsilon\circ H^k(D^\bullet) \arrow[r,"\cong"] & \underline{H}^k\circ cr^{-1}\circ\underline{\epsilon}(D^\bullet)
\end{tikzcd} \quad .
\end{equation*}
This implies $\epsilon\circ H^k\cong \underline{H}^k\circ cr^{-1}\circ \underline{\epsilon}$ and, since $\eta\circ\epsilon\cong\text{id}$ and $cr^{-1}\circ cr\cong\text{id}$, it also implies $\eta\circ \underline{H}^k\cong H^k\circ \underline{\eta}\circ cr$. 
\end{proof}

\subsection{The construction of sheaf cohomology} \label{658}

For a general reference on sheaf cohomology see \cite{bredon,cohomology,iversen}. As discussed in Section \ref{260}, sheaf cohomology is defined as the right derived functors of the global section functor (Definition \ref{661}). In this section we discuss how these functors are constructed.  We consider the category $\mathbf{Shv}(X,\mathbf{M})$, where $\mathbf{M}$ is a category of modules or graded modules. To define right derived functors, we have to introduce injective objects and injective resolutions in the category of sheaves. An injective resolution is, in some sense, an alternative representation of an object by a complex of special objects called injectives. 
\begin{definition}[Injective sheaf] \label{427} 
A sheaf $\mathcal{I}\in\mathbf{Shv}(X,\mathbf{M})$ is an \emph{injective sheaf} if for any sheaf monomorphism $\phi\colon\mathcal{F}\rightarrow \mathcal{G}$ and any sheaf morphism $\psi\colon\mathcal{F}\rightarrow \mathcal{I}$ there exists a sheaf morphism $\hat{\psi}\colon\mathcal{G}\rightarrow \mathcal{I}$ such that the following diagram commutes
\begin{equation*} 
\begin{tikzcd}
0 \arrow[r] & \mathcal{F} \arrow[r,"\phi"] \arrow[d,swap,"\psi"] & \mathcal{G} \arrow[dl,dashed,"\hat{\psi}"] \\
& \mathcal{I} 
\end{tikzcd} \quad .
\end{equation*}
\end{definition}

\noindent 
To construct injective resolutions of objects in a certain category it is necessary that the category has enough injective objects. In our setting, this means that for every $\mathcal{F}\in\mathbf{Shv}(X,\mathbf{M})$ there exists an injective sheaf $\mathcal{I}\in\mathbf{Shv}(X,\mathbf{M})$ and a sheaf monomorphism $\phi\colon\mathcal{F}\rightarrow \mathcal{I}$. The category $\mathbf{Shv}(X,\mathbf{M})$ has enough injectives but this might not be a well-known fact in the case $\mathbf{M}=\mathbf{grMod}_R$. We can use the following results. The categories $\mathbf{Mod}_R$ and $\mathbf{grMod}_R$ are Grothendieck categories \cite{gradedrings}. A category of sheaves $\mathbf{Shv}(X,\mathbf{M})$ with values in a Grothendieck category $\mathbf{M}$ is itself a Grothendieck category \cite{gray}. A Grothendieck category has enough injectives \cite{tohoku}.  This property of the category of sheaves allows the construction of injective resultions.
\begin{definition}[Injective resolution] \label{939} 
Let $\mathcal{F}\in\mathbf{Shv}(X,\mathbf{M})$. An \emph{injective resolution} of $\mathcal{F}$ is any exact sequence 
\begin{equation*} 
\begin{tikzcd}
0 \arrow[r] & \mathcal{F} \arrow[r] & \mathcal{I}^0 \arrow[r] & \mathcal{I}^1 \arrow[r] & \mathcal{I}^2 \arrow[r] & \cdots
\end{tikzcd}
\end{equation*}
such that $\mathcal{I}^n$ is an injective sheaf for all $n\in\mathbb{N}_0$. We use the short notation $\mathcal{I}^\bullet$ for the exact sequence starting at $\mathcal{I}^0$. In this notation an injective resolution of $\mathcal{F}$ can be written as 
\begin{equation*} 
\begin{tikzcd}
0 \arrow[r] & \mathcal{F} \arrow[r] & \mathcal{I}^\bullet \, .
\end{tikzcd}
\end{equation*}
\end{definition}

\noindent 
One can interpret an injective resolution of an object as measuring the deviation of the object from being injective. If the corresponding category has enough injectives, one can inductively construct an injective resolution for every object. Hence, every sheaf has an injective resolution. The following is a key result about injective resolutions.
\begin{theorem} \label{171} \cite[Theorem 11.21]{cohomology}
If $\phi\colon\mathcal{F}\rightarrow\mathcal{G}$ is a sheaf morphism, $0\rightarrow \mathcal{F}\rightarrow\mathcal{L}^\bullet$ an exact sequence (a right resolution of $\mathcal{F}$) and $0\rightarrow \mathcal{G}\rightarrow\mathcal{I}^\bullet$ a cochain complex such that every $\mathcal{I}^n$ is injective, then there exists a cochain morphism $\psi^\bullet\colon\mathcal{L}^\bullet\rightarrow\mathcal{I}^\bullet$ lifting $\phi$. In other words, there exists a family of sheaf morphisms $\psi^\bullet=(\psi^k)_{k\in\mathbb{N}_0}$ such that the following diagram commutes
\begin{equation*} 
\begin{tikzcd}
0 \arrow[r] & \mathcal{F} \arrow[r] \arrow[d,"\phi"] & \mathcal{L}^0 \arrow[r] \arrow[d,"\psi^0"] & \mathcal{L}^1 \arrow[r] \arrow[r] \arrow[d,"\psi^1"] & \mathcal{L}^2 \arrow[r] \arrow[d,"\psi^2"] \arrow[r] & \cdots \\
0 \arrow[r] & \mathcal{G} \arrow[r] & \mathcal{I}^0 \arrow[r] & \mathcal{I}^1 \arrow[r] & \mathcal{I}^2 \arrow[r] & \cdots
\end{tikzcd}
\end{equation*}
Moreover, this cochain morphism $\psi^\bullet$ is unique up to homotopy.
\end{theorem}

\noindent
Note that Theorem \ref{171}, for $\phi=\text{id}\colon\mathcal{F}\rightarrow \mathcal{F}$, implies that any two injective resolutions of $\mathcal{F}$ are homotopy equivalent. This property is important for derived functors to be well-defined. We are now able to construct right derived functors.

\begin{definition}[Right derived functors of the global section functor] \label{656} 
Let $\phi\colon\mathcal{F}\rightarrow\mathcal{G}$ be a morphism in $\mathbf{Shv}(X,\mathbf{M})$. Then there exist injective resolutions $0\rightarrow \mathcal{F}\rightarrow\mathcal{I}^\bullet$ and $0\rightarrow \mathcal{G}\rightarrow\mathcal{J}^\bullet$ and, by Theorem \ref{171}, a cochain morphism $\psi^\bullet\colon\mathcal{I}^\bullet\rightarrow\mathcal{J}^\bullet$ lifting $\phi$, i.e.\ we have the following commutative diagram
\begin{equation} \label{482}
\begin{tikzcd}
0 \arrow[r]  & \mathcal{I}^0 \arrow[r] \arrow[d,"\psi^0"] & \mathcal{I}^1 \arrow[r] \arrow[r] \arrow[d,"\psi^1"] & \mathcal{I}^2 \arrow[r] \arrow[d,"\psi^2"] \arrow[r] & \cdots \\
0 \arrow[r] & \mathcal{J}^0 \arrow[r] & \mathcal{J}^1 \arrow[r] & \mathcal{J}^2 \arrow[r] & \cdots
\end{tikzcd} \, .
\end{equation}
If we apply the global section functor $\Gamma(X,-)$ to this diagram, we obtain the following commutative diagram in $\mathbf{M}$
\begin{equation} \label{873}
\begin{tikzcd}
0 \arrow[r]  & \Gamma(X,\mathcal{I}^0) \arrow[r] \arrow[d,"\Gamma (X\text{,}\psi^0)"] & \Gamma(X,\mathcal{I}^1) \arrow[r] \arrow[r] \arrow[d,"\Gamma (X\text{,}\psi^1)"] & \Gamma(X,\mathcal{I}^2) \arrow[r] \arrow[d,"\Gamma (X\text{,}\psi^2)"] \arrow[r] & \cdots \\
0 \arrow[r] & \Gamma(X,\mathcal{J}^0) \arrow[r] & \Gamma(X,\mathcal{J}^1)  \arrow[r] & \Gamma(X,\mathcal{J}^2)  \arrow[r] & \cdots 
\end{tikzcd}\, .
\end{equation}
Let $\Gamma(X,\mathcal{I}^\bullet)$ denote the upper row cochain complex and $\Gamma(X,\mathcal{J}^\bullet)$ denote the lower row cochain complex of (\ref{873}). Let $\Gamma(X,\psi^\bullet)\colon\Gamma(X,\mathcal{I}^\bullet)\rightarrow\Gamma(X,\mathcal{J}^\bullet)$ denote the induced cochain morphism, i.e.\ the family of vertical arrows of (\ref{873}). For every $k\in\mathbb{N}_0$, define the functor $R^k\Gamma(X,-)\colon\mathbf{Shv}(X,\mathbf{M})\rightarrow \mathbf{M}$ by:
\begin{equation*} 
\begin{aligned} 
& \mathcal{F}\mapsto H^k\big(\Gamma(X,\mathcal{I}^\bullet)\big)  &\forall \mathcal{F}\in \mathbf{Shv}(X,\mathbf{M}) \text{ } \\
& (\phi\colon\mathcal{F}\rightarrow \mathcal{G}) \mapsto \Big(H^k\big(\Gamma(X,\psi^\bullet)\big)\colon H^k\big(\Gamma(X,\mathcal{I}^\bullet)\big)\rightarrow H^k\big(\Gamma(X,\mathcal{J}^\bullet)\big)\Big)  &\forall \phi \in \text{Hom}(\mathcal{F},\mathcal{G}) .
\end{aligned}
\end{equation*}
The functor $R^k\Gamma(X,-)$ is called the \emph{$k$-th right derived functor} of the global section functor.
\end{definition}

\noindent
With the help of Theorem \ref{171}, one can show that the right derived functors are well-defined, i.e.\ they are independent (up to natural isomorphism) of the choices of injective resolutions and lifting cochain morphisms \cite[Chapter 11.4]{cohomology}. Note that the injective resolutions in (\ref{482}) are exact sequences of sheaves but, since the global section functor $\Gamma(X,-)$ is not exact in general (it is only left exact), the rows of (\ref{873}) are not exact anymore. The cohomology of the cochain complexes in (\ref{873}) measures the deviation of the complexes from being exact. Hence, the right derived functors measure the deviation of the global section functor from being exact on injective resolutions of $\mathcal{F}$.

\subsection{The construction of morphisms induced by continuous maps in sheaf cohomology} \label{710}

In this section we discuss how to construct the morphisms $H^k(f)\colon H^k(Y,\mathcal{F})\rightarrow H^k(X,f^*\mathcal{F})$ induced by a continuous map $f\colon X\rightarrow Y$ in sheaf cohomology for a sheaf $\mathcal{F}\in \mathbf{Shv}(Y,\mathbf{M})$.

Let $0\rightarrow \mathcal{F}\rightarrow \mathcal{I}^\bullet$ be an injective resolution of $\mathcal{F}$. By Theorem \ref{573}, the inverse image functor $f^*$ is left adjoint to the direct image functor $f_*$. Hence, there is a natural transformation $\eta\colon\text{id}\rightarrow f_*f^*$ called the unit of the adjunction. This natural transformation induces a sheaf morphism $\eta_{\mathcal{G}}\colon\mathcal{G}\rightarrow f_*f^*\mathcal{G}$ for every $\mathcal{G}\in\mathbf{Shv}(Y,\mathbf{M})$. By applying the functor $f_*f^*$ on the diagram $0\rightarrow \mathcal{F}\rightarrow \mathcal{I}^\bullet$ and by the naturality of $\eta$, we obtain the following commutative diagram  
\begin{equation} \label{202}
\begin{tikzcd}
0 \arrow[r] & \mathcal{F} \arrow[r] \arrow[d,"\eta_{\mathcal{F}}"] & \mathcal{I}^0 \arrow[r] \arrow[d,"\eta_{\mathcal{I}^0}"] & \mathcal{I}^1 \arrow[r] \arrow[d,"\eta_{\mathcal{I}^1}"] & \cdots \\
0 \arrow[r] & f_*f^*\mathcal{F} \arrow[r] & f_*f^*\mathcal{I}^0 \arrow[r] & f_*f^*\mathcal{I}^1 \arrow[r] & \cdots
\end{tikzcd} \, .
\end{equation}
Now let $0\rightarrow f^*\mathcal{F}\rightarrow \mathcal{J}^\bullet$ be an injective resolution of $f^*\mathcal{F}$. By Theorem \ref{573}, the functor $f^*$ is exact, hence applying $f^*$ on the diagram $0\rightarrow \mathcal{F}\rightarrow \mathcal{I}^\bullet$ yields the exact sequence $0\rightarrow f^*\mathcal{F}\rightarrow f^*\mathcal{I}^\bullet$. By Theorem \ref{171}, there is a cochain morphism $g^\bullet=(g^n)_{n\in\mathbb{N}_0}\colon f^*\mathcal{I}^\bullet\rightarrow \mathcal{J}^\bullet$ lifting $\text{id}\colon f^*\mathcal{F}\rightarrow f^*\mathcal{F}$, i.e.\ a commutative diagram 
\begin{equation} \label{719}
\begin{tikzcd}
0 \arrow[r] & f^*\mathcal{F} \arrow[r] \arrow[d,"\text{id}"] & f^*\mathcal{I}^0 \arrow[r] \arrow[d,"g^0"] &  f^*\mathcal{I}^1 \arrow[r] \arrow[d,"g^1"] & \cdots \\
0 \arrow[r] & f^*\mathcal{F} \arrow[r] & \mathcal{J}^0 \arrow[r] &  \mathcal{J}^1 \arrow[r] & \cdots
\end{tikzcd} \, .
\end{equation}
We can now apply the global section functor $\Gamma(Y,-)$ on (\ref{202}) to obtain the upper part of (\ref{228}) and the global section functor $\Gamma(X,-)$ on (\ref{719}) to obtain the lower part of (\ref{228}). Moreover, we observe that for every sheaf $\mathcal{G}$ on $Y$ we have $\Gamma(Y,f_*f^*\mathcal{G})=(f_*f^*\mathcal{G})(Y)=(f^*\mathcal{G})(f^{-1}(Y))=(f^*\mathcal{G})(X)=\Gamma(X,f^*\mathcal{G})$. Hence, we can connect the image of (\ref{202}) and (\ref{719}) under the global section functors to obtain   
\begin{equation} \label{228}
\begin{tikzcd}
0 \arrow[r] & \Gamma(Y,\mathcal{F}) \arrow[r] \arrow[d,"\Gamma (Y\text{,}\eta_{\mathcal{F}})"] & \Gamma(Y,\mathcal{I}^0) \arrow[r] \arrow[d,"\Gamma (Y\text{,}\eta_{\mathcal{I}^0})"] & \Gamma(Y,\mathcal{I}^1) \arrow[r] \arrow[d,"\Gamma (Y\text{,}\eta_{\mathcal{I}^1})"] & \cdots  \\
0 \arrow[r] & \Gamma(Y,f_*f^*\mathcal{F}) \arrow[r] \arrow[d,"="] & \Gamma(Y,f_*f^*\mathcal{I}^0) \arrow[r] \arrow[d,"="] & \Gamma(Y,f_*f^*\mathcal{I}^1) \arrow[r] \arrow[d,"="] & \cdots \\
0 \arrow[r] & \Gamma(X,f^*\mathcal{F}) \arrow[r] \arrow[d,"\text{id}"] & \Gamma(X,f^*\mathcal{I}^0) \arrow[r] \arrow[d,"\Gamma(X\text{,}g^0)"] &  \Gamma(X,f^*\mathcal{I}^1) \arrow[r] \arrow[d,"\Gamma(X\text{,}g^1)"] & \cdots \\
0 \arrow[r] & \Gamma(X,f^*\mathcal{F}) \arrow[r] & \Gamma(X,\mathcal{J}^0) \arrow[r] &  \Gamma(X,\mathcal{J}^1) \arrow[r] & \cdots
\end{tikzcd} \, .
\end{equation}
From (\ref{228}) we get a composition of cochain morphisms
\begin{equation*} 
\begin{tikzcd}[column sep=large]
\Gamma(Y,\mathcal{I}^\bullet) \arrow[r,"\Gamma(Y\text{,}\eta^\bullet)"] & \Gamma(X,f^*\mathcal{I}^\bullet) \arrow[r,"\Gamma(X\text{,}g^\bullet)"] & \Gamma(X,\mathcal{J}^\bullet)
\end{tikzcd}
\end{equation*}
where $\eta^\bullet=(\eta_{\mathcal{I}^n})_{n\in\mathbb{N}_0}$. By applying the cohomology functor on this diagram of cochain complexes, for every $k\in\mathbb{N}_0$, we obtain
\begin{equation*} 
\begin{tikzcd}[column sep=large]
H^k(Y,\mathcal{F})=H^k\big(\Gamma(Y,\mathcal{I}^\bullet)\big) \arrow[r,"H^k(\Gamma(Y\text{,}\eta^\bullet))"] &[10pt] H^k\big(\Gamma(X,f^*\mathcal{I}^\bullet)\big) \arrow[r,"H^k(\Gamma(X\text{,}g^\bullet))"] &[10pt] H^k\big(\Gamma(X,\mathcal{J}^\bullet)\big)=H^k(X,f^*\mathcal{F})
\end{tikzcd} \, .
\end{equation*}
Therefore, for every $k\in\mathbb{N}_0$, the composition defines a natural morphism 
\begin{equation*} 
H^k(f)\coloneqq H^k\big(\Gamma(X,g^\bullet)\big)\circ H^k\big(\Gamma(Y,\eta^\bullet)\big)\colon H^k(Y,\mathcal{F})\rightarrow H^k(X,f^*\mathcal{F})
\end{equation*}
on sheaf cohomology modules. Of course one needs to show that these induced morphisms are well-defined, i.e.\ the construction does not depend on the choice of injective resolutions and lifting morphisms. This can be done by using the uniqueness up to homotopy in Theorem \ref{171}.

\subsection{Proof of Theorem \ref{511}} \label{295}

\begin{theorem} 
Let $f\colon X\rightarrow Y$ be a continuous map and $\phi\colon\mathcal{F}\rightarrow \mathcal{G}$ a morphism of sheaves on $Y$. Then, for all $k\in\mathbb{N}_0$, the following diagram commutes
\begin{equation*} 
\begin{tikzcd}[column sep=huge]
H^k(Y,\mathcal{F}) \arrow[r,"H^k(Y\text{,}\phi)"] \arrow[d,swap,"H^k(f\text{,}\mathcal{F})"] & H^k(Y,\mathcal{G}) \arrow[d,"H^k(f\text{,}\mathcal{G})"] \\
H^k(X,f^*\mathcal{F}) \arrow[r,"H^k(X\text{,}f^*\phi)"] & H^k(X,f^*\mathcal{G}) 
\end{tikzcd}
\end{equation*}
i.e.\ the morphisms induced by continuous maps and sheaf morphisms commute.
\end{theorem} 

\begin{proof}
Let $0\rightarrow\mathcal{F}\rightarrow \mathcal{I}^\bullet$, $0\rightarrow\mathcal{G}\rightarrow \mathcal{J}^\bullet$, $0\rightarrow f^*\mathcal{F}\rightarrow \overline{\mathcal{I}}^\bullet$ and $0\rightarrow f^*\mathcal{G}\rightarrow \overline{\mathcal{J}}^\bullet$ be injective resolutions of $\mathcal{F}$, $\mathcal{G}$, $f^*\mathcal{F}$ and $f^*\mathcal{G}$, respectively. The following construction is similar to the one in Section \ref{710} with an additional dimension obtained by the cochain morphism $\mathcal{I}^\bullet\rightarrow\mathcal{J}^\bullet$ induced on injective resolutions by the sheaf morphism $\phi\colon\mathcal{F}\rightarrow\mathcal{G}$ (see \ref{656}). The natural transformation $\eta\colon\text{id}\rightarrow f_*f^*$ yields the following commutative diagram of sheaves on $Y$
\begin{equation} \label{488}
\begin{tikzcd}[column sep=tiny,row sep=small]
& 0 \arrow[rr] & & \mathcal{G} \arrow[rr] \arrow[dd,"\eta_\mathcal{G}"{yshift=10pt}] & & \mathcal{J}^0 \arrow[rr] \arrow[dd] & & \mathcal{J}^1 \arrow[rr] \arrow[dd] & & \cdots
\\
0 \arrow[rr] & & \mathcal{F} \arrow[rr,crossing over] \arrow[ur,"\phi"{yshift=-2pt}]  & & \mathcal{I}^0 \arrow[rr,crossing over] \arrow[ur] & & \mathcal{I}^1 \arrow[rr,crossing over] \arrow[ur] & & \cdots \\
& 0 \arrow[rr] & & f_*f^*\mathcal{G} \arrow[rr] & & f_*f^*\mathcal{J}^0 \arrow[rr] & & f_*f^*\mathcal{J}^1 \arrow[rr] & & \cdots
\\
0 \arrow[rr] & & f_*f^*\mathcal{F} \arrow[rr] \arrow[ur,swap,"f_*f^*\phi"{yshift=2.5pt}] \arrow[uu,<-,swap,"\eta_\mathcal{F}"{yshift=10pt},crossing over] & & f_*f^*\mathcal{I}^0 \arrow[rr] \arrow[ur] \arrow[uu,<-,crossing over] & & f_*f^*\mathcal{I}^1 \arrow[rr] \arrow[ur] \arrow[uu,<-,crossing over] & & \cdots
\end{tikzcd} \, .
\end{equation}
Moreover, we get the not necessarily commutative diagram (\ref{807}), where the cochain morphism on the top is obtained by applying $f^*$ on the top of (\ref{488}) and the other cochain morphisms are obtain from Theorem \ref{171}. Note that both compositions of cochain morphisms in (\ref{807}) from $f^*\mathcal{I}^\bullet$ to $\overline{\mathcal{J}}^\bullet$ lift $f^*\phi$, hence, by Theorem \ref{171}, they are homotopic.
\begin{equation} \label{807}
\begin{tikzcd}[column sep=small,row sep=small]
& 0 \arrow[rr] & & f^*\mathcal{G} \arrow[rr] \arrow[dd,"\text{id}"{yshift=10pt}] & & f^*\mathcal{J}^0 \arrow[rr] \arrow[dd] & & f^*\mathcal{J}^1 \arrow[rr] \arrow[dd] & & \cdots
\\
0 \arrow[rr] & & f^*\mathcal{F} \arrow[rr,crossing over] \arrow[ur,"f^*\phi"{yshift=-2pt}] & & f^*\mathcal{I}^0 \arrow[rr,crossing over] \arrow[ur] & & f^*\mathcal{I}^1 \arrow[rr,crossing over] \arrow[ur] & & \cdots \\
& 0 \arrow[rr] & & f^*\mathcal{G} \arrow[rr] & & \overline{\mathcal{J}}^0 \arrow[rr] & & \overline{\mathcal{J}}^1 \arrow[rr] & & \cdots
\\
0 \arrow[rr] & & f^*\mathcal{F} \arrow[rr] \arrow[ur,swap,"f^*\phi"{yshift=2pt}] \arrow[uu,<-,swap,"\text{id}"{yshift=10pt},crossing over]  & & \overline{\mathcal{I}}^0 \arrow[rr] \arrow[ur] \arrow[uu,<-,crossing over] & & \overline{\mathcal{I}}^1 \arrow[rr] \arrow[ur] \arrow[uu,<-,crossing over] & & \cdots 
\end{tikzcd} \, .
\end{equation}
By applying the global section functor $\Gamma(Y,-)$ on (\ref{488}) and the global section functor $\Gamma(X,-)$ on (\ref{807}), we can combine the two diagrams to obtain the following diagram

\begin{equation*} 
\begin{tikzcd}[column sep=tiny,row sep=small]
& 0 \arrow[rr] & & \Gamma(Y,\mathcal{J}^0) \arrow[rr] \arrow[dd] & & \Gamma(Y,\mathcal{J}^1) \arrow[rr] \arrow[dd] & & \cdots
\\
0 \arrow[rr]  & & \Gamma(Y,\mathcal{I}^0) \arrow[rr,crossing over] \arrow[ur] & & \Gamma(Y,\mathcal{I}^1) \arrow[rr,crossing over] \arrow[ur] & & \cdots \\
& 0 \arrow[rr] & & \Gamma(X,f^*\mathcal{J}^0) \arrow[rr] \arrow[dd] & & \Gamma(X,f^*\mathcal{J}^1) \arrow[rr] \arrow[dd] & & \cdots
\\
0 \arrow[rr]  & & \Gamma(X,f^*\mathcal{I}^0) \arrow[rr,crossing over] \arrow[ur] \arrow[uu,<-,crossing over] & & \Gamma(X,f^*\mathcal{I}^1) \arrow[rr,crossing over] \arrow[ur] \arrow[uu,<-,crossing over] & & \cdots \\
& 0 \arrow[rr] & &  \Gamma(X,\overline{\mathcal{J}}^0) \arrow[rr] & & \Gamma(X,\overline{\mathcal{J}}^1) \arrow[rr] & & \cdots
\\
0 \arrow[rr] & & \Gamma(X,\overline{\mathcal{I}}^0) \arrow[rr] \arrow[ur] \arrow[uu,<-,crossing over] & & \Gamma(X,\overline{\mathcal{I}}^1) \arrow[rr] \arrow[ur] \arrow[uu,<-,crossing over] & & \cdots 
\end{tikzcd} \, .
\end{equation*}
This diagram is not necessarily commutative but all possible different paths are homotopic since the global section functor is additive and therefore preserves homotopy. By applying the cohomology functor and by the fact that homotopic cochain morphisms induce the same maps in cohomology, we obtain the following commutative diagram
\begin{equation*} 
\begin{tikzcd}
H^k(Y,\mathcal{F}) \arrow[r,"H^k(Y\text{,}\phi)"] \arrow[d] &[10pt] H^k(Y,\mathcal{G}) \arrow[d] \\
H^k\big(\Gamma(X,f^*\mathcal{I}^\bullet)\big) \arrow[r] \arrow[d] & H^k\big(\Gamma(X,f^*\mathcal{J}^\bullet)\big) \arrow[d] \\
H^k(X,f^*\mathcal{F}) \arrow[r,"H^k(X\text{,}f^*\phi)"] & H^k(X,f^*\mathcal{G})
\end{tikzcd}
\end{equation*}
where the compositions of the vertical arrows are the induced morphisms $H^k(f,\mathcal{F})\colon H^k(Y,\mathcal{F})\rightarrow H^k(X,f^*\mathcal{F})$ and $H^k(f,\mathcal{G})\colon H^k(Y,\mathcal{G})\rightarrow H^k(X,f^*\mathcal{G})$ as constructed in Section \ref{710}.
\end{proof} 

\subsection{Construction of \v{C}ech cohomology} \label{967}

In Section \ref{658} we discussed the general definition of sheaf cohomology. The problem with the derived functor construction is that it is not very practical. Although it might be possible to construct injective resolutions explicitly, one might hope for a simpler method to compute sheaf cohomology. \v{C}ech cohomology is an alternative cohomology theory for (pre)sheaves which is much easier to construct in many situations. For a general reference on \v{Cech} cohomology see \cite{bott,cohomology}. \v{C}ech cohomology is defined in two steps. In the first step one defines \v{C}ech cohomology with respect to a fixed open cover of a topological space $X$ and in the second step one defines the \v{C}ech cohomology with respect to $X$ via a colimit over all open covers.

If $\mathcal{U}=(U_i)_{i\in I}$ is an open cover of a space $X$ with an ordered index set $I$ and $k\in\mathbb{N}_0$, we define $N^k_{\mathcal{U}}\coloneqq \{\sigma=(i_0,...,i_k)\text{ }|\text{ }i_0<\ldots <i_k \in I\colon U_{i_0}\cap\ldots\cap U_{i_k}\neq\emptyset\}$ and $U_\sigma\coloneqq U_{i_0}\cap\ldots\cap U_{i_k}$ for $\sigma\in N^k_\mathcal{U}$, i.e.\ $N^k_\mathcal{U}$ is the set of $k$-simplices of the nerve complex $N_\mathcal{U}$ of $\mathcal{U}$. For $\sigma=(i_0,\ldots,i_{k+1})\in N_\mathcal{U}^{k+1}$ we define by  $\partial_j\sigma\coloneqq (i_0,\ldots,\hat{i}_j,\ldots,i_{k+1})\in N^{k}_\mathcal{U}$ the $j$-th simplex in the boundary of $\sigma$.

\begin{definition}[$\check{\text{C}}$ech complex] \label{271} 
Let $X$ be a topological space, $\mathcal{U}=(U_i)_{i\in I}$ an open cover of $X$ with an ordered index set $I$ and $\mathcal{F}$ a (pre)sheaf with values in a category $\mathbf{M}$. For every $k\in \mathbb{N}_0$, define the $k$-dimensional \emph{\v{C}ech-cochains} of $\mathcal{F}$ with respect to the cover $\mathcal{U}$ in the following way
\begin{equation*} 
C^k(\mathcal{U},\mathcal{F})\coloneqq\prod_{\sigma\in N^k_\mathcal{U}} \mathcal{F}(U_\sigma) \quad .
\end{equation*}

\noindent
We define the $k$-th \emph{coboundary morphism} $\delta^k\colon C^k(\mathcal{U},\mathcal{F})\rightarrow C^{k+1}(\mathcal{U},\mathcal{F})$  by
\begin{equation*} 
\delta^k\coloneqq \prod_{\tau\in N^{k+1}_\mathcal{U}}\Big(\sum_{j=0}^{k+1}(-1)^j\mathcal{F}(U_\tau\xhookrightarrow{} U_{\partial_j\tau})\circ p_{\partial_j\tau}\Big) \quad .
\end{equation*}
The \v{C}ech-cochains together with the coboundary morphisms form the cochain complex $C^\bullet(\mathcal{U},\mathcal{F})$ 
\begin{equation*} 
\begin{tikzcd}
0 \arrow[r] & C^0(\mathcal{U},\mathcal{F}) \arrow[r,"\delta^0"] & C^1(\mathcal{U},\mathcal{F}) \arrow[r,"\delta^1"] & C^2(\mathcal{U},\mathcal{F}) \arrow[r,"\delta^2"] & \cdots 
\end{tikzcd} \, .
\end{equation*}
The cohomology of this cochain complex  
\begin{equation*} 
\check{H}^k(\mathcal{U},\mathcal{F})\coloneqq H^k\big(C^\bullet(\mathcal{U},\mathcal{F})\big)
\end{equation*}
is called the $k$-th \emph{\v{C}ech cohomology} of the (pre)sheaf $\mathcal{F}$ with respect to the cover $\mathcal{U}$.
\end{definition} 

\noindent
For our applications of \v{C}ech cohomology in the context of persistence we also need a notion of induced morphisms. 

\begin{definition}[Induced morphisms by (pre)sheaf morphisms] \label{115} 
Let $\phi\colon\mathcal{F}\rightarrow \mathcal{G}$ be a morphism of (pre)sheaves on a topological space $X$ and $\mathcal{U}=(U_i)_{i\in I}$ an open cover of $X$. For all $k\in\mathbb{N}_0$, define the morphism $\phi^k\colon C^k(\mathcal{U},\mathcal{F})\rightarrow C^k(\mathcal{U},\mathcal{G})$ by
\begin{equation*} 
\phi^k\coloneqq\prod_{\sigma\in N_\mathcal{U}^k}  \phi_{U_{\sigma}}\circ p_\sigma  \quad .
\end{equation*}
Then the following diagram commutes 
\begin{equation*} 
\begin{tikzcd}
C^k(\mathcal{U},\mathcal{F}) \arrow[r,"\delta^k"] \arrow[d,swap,"\phi^k"] & C^{k+1}(\mathcal{U},\mathcal{F}) \arrow[d,"\phi^{k+1}"]  \\
C^k(\mathcal{U},\mathcal{G}) \arrow[r,"\delta^k"] & C^{k+1}(\mathcal{U},\mathcal{G})
\end{tikzcd}
\end{equation*}
i.e.\ $\phi^\bullet=(\phi^k)_{k\in\mathbb{N}_0}\colon C^\bullet(\mathcal{U},\mathcal{F})\rightarrow C^\bullet(\mathcal{U},\mathcal{G})$ is a cochain morphism. This cochain morphism induces a morphism $\check{H}^k(\mathcal{U},\phi)\colon\check{H}^k(\mathcal{U},\mathcal{F})\rightarrow \check{H}^k(\mathcal{U},\mathcal{G}) \nonumber $ in cohomology defined by
\begin{equation*} 
\check{H}^k(\mathcal{U},\phi)\coloneqq H^k(\phi^\bullet)\colon H^k\big(C^\bullet(\mathcal{U},\mathcal{F})\big)\rightarrow H^k\big(C^\bullet(\mathcal{U},\mathcal{G})\big) \quad .
\end{equation*}  
\end{definition}  

\noindent
The action of $\check{H}^k(\mathcal{U},-)$ on (pre)sheaves of Definition \ref{271} and the action of $\check{H}^k(\mathcal{U},-)$ on (pre)sheaf morphisms of Definition \ref{115} defines a functor $\check{H}^k(\mathcal{U},-)\colon\mathbf{pShv}(X,\mathbf{M})\rightarrow \mathbf{M}$. 

Now that we have defined \v{C}ech cohomology with respect to a fixed open cover, we still have to relate the \v{C}ech cohomology modules with respect to different open covers. Let $\mathcal{U}=(U_i)_{i\in I}$ and $\mathcal{V}=(V_j)_{j\in J}$ be open covers of $X$. Then $\mathcal{V}$ refines $\mathcal{U}$ (denoted $\mathcal{U}\leq \mathcal{V}$), if there is a map $r\colon J\rightarrow I$ such that $V_j\subseteq U_{r(j)}$ for all $j\in J$. Note that $r$ extends to a map $r\colon N_\mathcal{V}^k\rightarrow N_\mathcal{U}^k$ defined by $\sigma=(j_0,\ldots,j_k)\mapsto \big(r(j_0),\ldots,r(j_k)\big)=r(\sigma)$.

\begin{definition}[Induced morphisms by refinements] \label{709}
Let $X$ be a topological space, $\mathcal{F}$ a (pre)sheaf on $X$ and $\mathcal{U}=(U_i)_{i\in I}$ and $\mathcal{V}=(V_j)_{j\in J}$ open covers of $X$ such that $\mathcal{V}$ is a refinement of $\mathcal{U}$. For $k\in\mathbb{N}_0$, define a morphism $r^k\colon C^k(\mathcal{U},\mathcal{F})\rightarrow C^k(\mathcal{V},\mathcal{F})$ by 
\begin{equation*} 
r^k\coloneqq \prod_{\sigma\in N_\mathcal{V}^k}\big(\mathcal{F}(V_{\sigma}\xhookrightarrow{}U_{r(\sigma)})\circ p_{r(\sigma)}\big) \quad .
\end{equation*}  
Then, for all $k\in\mathbb{N}_0$, the following diagram commutes 
\begin{equation*} 
\begin{tikzcd}
C^k(\mathcal{U},\mathcal{F}) \arrow[r,"\delta^k"] \arrow[d,swap,"r^k"] & C^{k+1}(\mathcal{U},\mathcal{F}) \arrow[d,"r^{k+1}"]  \\
C^k(\mathcal{V},\mathcal{G}) \arrow[r,"\delta^k"] & C^{k+1}(\mathcal{V},\mathcal{G})
\end{tikzcd}
\end{equation*}
i.e.\ $r^\bullet=(r^k)_{k\in\mathbb{N}_0}\colon C^\bullet(\mathcal{U},\mathcal{F})\rightarrow C^\bullet(\mathcal{U},\mathcal{G})$ is a cochain morphism. This cochain morphism induces a morphism $\check{H}^k(r,\mathcal{F})\coloneqq H^k(r^\bullet)\colon\check{H}^k(\mathcal{U},\mathcal{F})\rightarrow \check{H}^k(\mathcal{V},\mathcal{F})$ in cohomology.
\end{definition}

\noindent
The following diagram involving the morphisms of Definition \ref{115} and \ref{709} commutes
\begin{equation*} 
\begin{tikzcd}[column sep=large]
\check{H}^k(\mathcal{U},\mathcal{F}) \arrow[r,"\check{H}^k(r\text{,}\mathcal{F})"] \arrow[d,swap,"\check{H}^k(\mathcal{U}\text{,}\phi)"] & \check{H}^k(\mathcal{V},\mathcal{F}) \arrow[d,"\check{H}^k(\mathcal{V}\text{,}\phi)"]  \\
\check{H}^k(\mathcal{U},\mathcal{G}) \arrow[r,"\check{H}^k(r\text{,}\mathcal{G})"] & \check{H}^k(\mathcal{V},\mathcal{G})
\end{tikzcd}
\end{equation*}
i.e.\ we have a natural transformation $\check{H}^k(r,-)\colon \check{H}^k(\mathcal{U},-)\rightarrow \check{H}^k(\mathcal{V},-)$ in $\mathbf{Fun}\big(\mathbf{pShv}(X,\mathbf{M}),\mathbf{M}\big)$. This implies that the assignments $\mathcal{U}\mapsto \check{H}^k(\mathcal{U},-)$ and $r\mapsto \check{H}^k(r,-)$ define a functor $\mathbf{Cover}(X)\rightarrow\mathbf{Fun}\big(\mathbf{pShv}(X,\mathbf{M}),\mathbf{M}\big)$ where $\mathbf{Cover}(X)$ denotes the category of open covers of $X$ with arrows given by refinement relations. Now we are able to define \v{C}ech cohomology on $X$.

\begin{definition}[$\check{\text{C}}$ech cohomology] \label{824}
Let $X$ be a topological space. For every $k\in\mathbb{N}_0$, define the $k$-th \emph{\v{C}ech cohomology functor} $\check{H}^k(X,-)\colon\mathbf{pShv}(X,\mathbf{M})\rightarrow \mathbf{M}$ by
\begin{equation*} 
\check{H}^k(X,-)\coloneqq\underset{\mathcal{U}}{\text{colim }}\check{H}^k(\mathcal{U},-) \quad .
\end{equation*}
\end{definition}

\subsection{The construction of morphisms induced by continuous maps in \v{C}ech cohomology} \label{809}

Also the previous construction of induced morphisms $H^k(f)\colon H^k(Y,\mathcal{F})\rightarrow H^k(X,f^*\mathcal{F})$ might be a little bit unhandy. Fortunately, we can also construct such morphisms using \v{C}ech cohohomology.

Let $\mathcal{F}$ be a sheaf on $Y$ and $\mathcal{U}=(U_i)_{i\in I}$ an open cover of $Y$. By Theorem \ref{573}, the functor $f^*$ is left adjoint to $f^*$. Hence, the unit $\eta\colon\text{id}\rightarrow f_*f^*$ of the adjunction induces a sheaf morphism $\eta_\mathcal{F}\colon\mathcal{F}\rightarrow f_*f^*\mathcal{F}$. By Definition \ref{115}, for every $k\in\mathbb{N}_0$, such a sheaf morphism induces a cochain morphism $\eta_\mathcal{F}^\bullet\colon C^\bullet(\mathcal{U},\mathcal{F})\rightarrow C^\bullet(\mathcal{U},f_*f^*\mathcal{F})$ and, moreover, a morphism in \v{C}ech cohomology $\check{H}^k(\mathcal{U},\eta_\mathcal{F})\colon\check{H}^k(\mathcal{U},\mathcal{F})\rightarrow\check{H}^k(\mathcal{U},f_*f^*\mathcal{F})$. Note that $f^{-1}(\mathcal{U})\coloneqq(f^{-1}(U_i))_{i\in I}$ is an open cover of $X$ and
\begin{equation*} 
\begin{aligned} 
C^k(\mathcal{U},f_*f^*\mathcal{F})&\cong\prod_{i_0<\ldots<i_k\in I}f_*f^*\mathcal{F}(U_{(i_0,\ldots ,i_k)})\cong\prod_{i_0<\ldots<i_k\in I}f^*\mathcal{F}\big(f^{-1}(U_{(i_0,\ldots ,i_k)})\big) \\
&\cong\prod_{i_0<\ldots<i_k\in I}f^*\mathcal{F}\big(f^{-1}(U_{i_0})\cap\ldots\cap f^{-1}(U_{i_k})\big)\cong C^k(f^{-1}(\mathcal{U}),f^*\mathcal{F}) \quad .
\end{aligned}
\end{equation*}
Hence, for every $k\in\mathbb{N}_0$, we get $\check{H}^k(\mathcal{U},f_*f^*\mathcal{F})\cong\check{H}^k(f^{-1}(\mathcal{U}),f^*\mathcal{F})$
and, therefore, a morphism 
\begin{equation*} 
\check{H}^k_\mathcal{U}(f)\colon\check{H}^k(\mathcal{U},\mathcal{F})\rightarrow \check{H}^k(f^{-1}(\mathcal{U}),f^*\mathcal{F}) \quad .
\end{equation*}
If $\mathcal{V}$ is a refinement of $\mathcal{U}$, by Definition \ref{709}, we get the following commutative diagram
\begin{equation*} 
\begin{tikzcd}[column sep=small,row sep=small]
\check{H}^k(\mathcal{U},\mathcal{F}) \arrow[rr,"\check{H}^k(r\text{,}\mathcal{F})"{xshift=-10pt}] \arrow[dd,swap,"\check{H}^k(\mathcal{U}\text{,}\eta_\mathcal{F})"] \arrow[dr] & & \check{H}^k(\mathcal{V},\mathcal{F}) \arrow[dd,"\check{H}^k(\mathcal{V}\text{,}\eta_\mathcal{F})"] \arrow[dl] \\
& \check{H}^k(Y,\mathcal{F}) \\
\check{H}^k(\mathcal{U},f_*f^*\mathcal{F}) \arrow[rr,"\check{H}^k(r\text{,}f_*f^*\mathcal{F})"{xshift=30pt}] \arrow[dd,swap,"\cong"] \arrow[dr] & & \check{H}^k(\mathcal{V},f_*f^*\mathcal{F}) \arrow[dd,"\cong"] \arrow[dl] \\
& \check{H}^k(Y,f_*f^*\mathcal{F}) \arrow[uu,<-,swap,dashed,"\check{H}^k(Y\text{,}\eta_\mathcal{F})"{yshift=-10pt},crossing over]\\
\check{H}^k(f^{-1}(\mathcal{U}),f^*\mathcal{F}) \arrow[rr,"\check{H}^k(r\text{,}f^*\mathcal{F})"{xshift=25pt}] \arrow[dr] & & \check{H}^k(f^{-1}(\mathcal{V}),f^*\mathcal{F}) \arrow[dl] \\
& \check{H}^k(X,f^*\mathcal{F}) \arrow[uu,<-,swap,dashed,"g"{yshift=-10pt},crossing over] \arrow[uuuu,<-,bend left=62,dashed,"\check{H}^k(f)"{yshift=-10pt},crossing over]
\end{tikzcd}
\end{equation*}
where the triangles represent the respective universal cocones defining \v{C}ech cohomology. Therefore, for every $k\in\mathbb{N}_0$, by the universal property of the colimit, there is a unique induced morphism in \v{C}ech cohomology $\check{H}^k(f)\colon\check{H}^k(Y,\mathcal{F})\rightarrow \check{H}^k(X,f^*\mathcal{F})$ given by
\begin{equation*} 
\check{H}^k(f)\coloneqq g\circ\check{H}^k(Y,\eta_\mathcal{F}) \, .
\end{equation*} 

\begin{remark} \label{851}
Given a continuous map $f\colon X\rightarrow Y$ and a sheaf $\mathcal{F}$ on $Y$. If sheaf and \v{C}ech cohomology agree, then the following diagram commutes \cite[Page 52]{grothendieckega}
\begin{equation*}
\begin{tikzcd}
H^k(Y,\mathcal{F}) \arrow[r,"H^k(f)"] \arrow[d,swap,"\cong"] &[10pt] H^k(X,f^*\mathcal{F}) \arrow[d,"\cong"] \\
\check{H}^k(Y,\mathcal{F}) \arrow[r,"\check{H}^k(f)"] & \check{H}^k(X,f^*\mathcal{F})
\end{tikzcd} \quad .
\end{equation*}
\end{remark}

\subsection{Cohomology of cellular sheaves on simplicial complexes} \label{233}

The cohomology of a sheaf on an abstract simplicial complex, viewed as a topological space with the Alexandrov topology, can be computed from the cochain complex (\ref{437}) \cite{curry}. In this section, we use \v{C}ech cohomology to obtain a derivation of this cochain complex alternative to the one given in \cite{curry}. Viewing the cohomology of cellular sheaves on simplicial complexes as the \v{C}ech cohomology with respect to a certain open cover also allows us to use the construction of Section \ref{809} to derive morphisms induced by simplicial maps. 

We start by showing that sheaf and \v{C}ech cohomology agree on abstract simplicial complexes. We use the following theorem.

\begin{theorem}[Cartan's theorem] \label{197} \cite[Theorem 13.19.]{cohomology} 
Let $X$ be a topological space, $\mathcal{F}$ a sheaf on $X$ and $\mathcal{U}=(U_i)_{i\in I}$ an open cover of $X$ such that $\mathcal{U}$ is a basis for the topology on $X$ closed under finite intersections and $\check{H}^k(U_{\sigma},\mathcal{F})=0$ for all $k>0$ and all $\sigma\in N_\mathcal{U}^k$, then $\check{H}^k(X,\mathcal{F})\cong H^k(X,\mathcal{F})$ for all $k\in\mathbb{N}_0$.
\end{theorem}

\begin{proposition} \label{538} 
For a sheaf on an abstract simplicial complex, viewed as a topological space with the Alexandrov topology, sheaf cohomology agrees with \v{C}ech cohomology.
\end{proposition}

\begin{proof}
Let $X$ be a simplicial complex and $\mathcal{F}$ a sheaf on $X$. We show that the conditions of Theorem \ref{197} are satisfied. Let $\mathcal{U}=\{U_x\text{ }|\text{ } x\in X\}$ be the standard basis of the Alexandrov topology on $X$. Then $\mathcal{U}$ obviously is an open cover of $X$. Let $x,y\in X$ be simplices defined by their sets of vertices and let $z=x\cup y$. Then $U_x\cap U_y=\{v\in X\text{ }|\text{ }x\leq v \text{ and }y\leq v\}=\{v\in X\text{ }|\text{ }x\subseteq v\text{ and }y\subseteq v\}=\{v\in X\text{ }|\text{ }x\cup y\subseteq v\}=\{v\in X\text{ }|\text{ }z\leq v\}=U_z$ if $z\in X$ and $U_x\cap U_y=\emptyset$ otherwise. The iterative application of this argument proves that $\mathcal{U}$ is closed under finite intersections. 

Let $x_0,\ldots,x_k\in X$ and $U_{x_0}\cap\ldots\cap U_{x_k}=U_z$ for a $z\in X$. Let $\mathcal{W}=\{U_z\}$ be the trivial open cover of $U_z$. Every other open cover $\mathcal{V}=(V_j)_{j\in J}$  of $U_z$, viewed as a subspace of $X$ with the subspace topology, obviously refines $\mathcal{W}$. On the other hand, there has to be a $j\in J$ such that $z\in V_j$. Since $U_z$ is open in $X$ also $V_j$ is open in $X$, hence, if $y\in X$ such that $z\leq y$, then $y\in V_j$. Every $y\in U_z$ satisfies $z\leq y$ implying $y\in V_j$ and therefore $U_z\subseteq V_j$. This implies that $\mathcal{W}$ refines $\mathcal{V}$, hence $\mathcal{V}$ and $\mathcal{W}$ are equivalent (i.e.\ $\mathcal{V}\leq \mathcal{W}$ and $\mathcal{W}\leq \mathcal{V}$). Therefore, every open cover $\mathcal{V}$ of $U_z$ is equivalent to the trivial open cover $\mathcal{W}$. Since equivalent open covers induce isomorphisms $\check{H}^k(\mathcal{V},\mathcal{F})\xrightarrow{\cong} \check{H}^k(\mathcal{W},\mathcal{F})$ \cite[Chapter 9.2]{cohomology}, we obtain $\check{H}^k(U_{z},\mathcal{F}|_{U_{z}})\cong\check{H}^k(\mathcal{W},\mathcal{F}|_{U_{z}})$ for all $k\in\mathbb{N}_0$. For the \v{C}ech cochains we get $C^k(\mathcal{W},\mathcal{F}|_{U_{z}})=\prod_{\sigma\in N_\mathcal{W}^k}\mathcal{F}(W_{\sigma})=0$ for every $k>0$ since $N_\mathcal{W}^k=\emptyset$ and the empty product is the terminal object. This implies $\check{H}^k(U_{z},\mathcal{F}|_{U_{z}})=0$ for every $k>0$ and all $x_0,\ldots,x_k\in X$. Hence, by Cartan's theorem, we have $\check{H}^k(X,\mathcal{F})\cong H^k(X,\mathcal{F})$ for every $k\in\mathbb{N}_0$. 
\end{proof}
\noindent
Note that this argument works for sheaves on any poset $X$ such that $\mathcal{U}=\{U_x\text{ }|\text{ } x\in X\}$ is closed under finite intersections. Proposition \ref{538} implies that we can use \v{C}ech cohomology to compute the cohomology of a cellular sheaf on a simplicial complex. 

\begin{theorem} \label{742}
Let $X$ be a finite simplicial complex and $\mathcal{F}$ a sheaf on $X$. If $F:\mathbf{X}\rightarrow \mathbf{M}$ is the corresponding functor on $X$, for all $k\in\mathbb{N}_0$, the sheaf cohomology of $\mathcal{F}$ is given by the cohomology of the cochain complex $(C^\bullet(X,F),\delta^\bullet)$ defined by
\begin{equation} \label{437}
\begin{aligned}
& C^k(X,F)\coloneqq \bigoplus_{\sigma\in X^k} F(\sigma) \\
& \delta^k\coloneqq\bigoplus_{\tau\in X^{k+1}}\big(\sum_{\tau \geq \sigma \in X^k} [\sigma:\tau] F(\sigma\rightarrow \tau)\circ p_\sigma\big)
\end{aligned} 
\end{equation}
\end{theorem}

\begin{proof}
Let $\text{Vert}(X)$ denote the set of vertices of $X$ with an arbitrary linear order. Define $\mathcal{U}\coloneqq(U_x)_{x\in\text{Vert}(X)}$ where $U_x$ denotes the basic open set corresponding to the vertex $x\in\text{Vert}(X)$. Then $\mathcal{U}$ is an open cover of $X$ since every simplex $\sigma\in X$ contains some vertex $x$ and is therefore contained in $U_x$. A simplex $\sigma=(x_0,\ldots,x_k)\in X$ is defined by its set of vertices $x_0<\ldots<x_k\in \text{Vert}(X)$ and it is easy to see that 
\begin{equation*} 
U_{x_0}\cap \ldots\cap U_{x_k}=\begin{cases} \emptyset  &\text{if} \quad \sigma\notin X \\ U_\sigma &\text{if} \quad \sigma\in X  \end{cases} 
\end{equation*}
where $U_\sigma$ on the one hand denotes the Alexandrov basic open set corresponding to $\sigma\in X$ and on the other hand the open set corresponding to $\sigma\in N_\mathcal{U}^k$. Hence, $N_\mathcal{U}=X$, i.e.\ the nerve complex of the open cover $\mathcal{U}$ is the original simplicial complex $X$. We define an orientation coefficient for incidence relations of simplices $\sigma,\tau\in X$:
\begin{equation*} 
[\sigma:\tau]\coloneqq\begin{cases} (-1)^j  &\text{if }\sigma=(x_0,\ldots,\hat{x}_j,\ldots,x_{k+1})<(x_0,\ldots,x_{k+1})=\tau \\ 0 &\text{ else} \end{cases} 
\end{equation*}
where $\tau=(x_0,\ldots,x_{k+1})$ is the representation of the simplex $\tau$ as an ordered list of its vertices. We now compute the \v{C}ech cohomology of $\mathcal{F}$ with respect to the open cover $\mathcal{U}$. By Definition \ref{271}, we obtain 
\begin{equation} \label{784}
C^k(\mathcal{U},\mathcal{F})=\prod_{\sigma\in N_\mathcal{U}^k} \mathcal{F}(U_\sigma)=\prod_{\sigma\in X^k} \mathcal{F}(U_{\sigma})=\prod_{\sigma\in X^k} F(\sigma) 
\end{equation}
and
\begin{equation} \label{852}
\begin{aligned} 
\delta^k&=\prod_{\tau\in N_\mathcal{U}^{k+1}}\Big(\sum_{j=0}^{k+1}(-1)^j\text{ }\mathcal{F}(U_{\tau}\xhookrightarrow{}U_{\partial_j\tau})\circ p_{\partial_j\tau}\Big) \\
&=\prod_{\tau\in X^{k+1}}\Big(\sum_{\tau\geq\sigma\in X^k}[\sigma:\tau]\text{ }\mathcal{F}(U_\tau \xhookrightarrow{}U_\sigma)\circ p_\sigma\Big) \\ &=\prod_{\tau\in X^{k+1}}\Big(\sum_{\tau\geq\sigma\in X^k}[\sigma:\tau]\text{ }F(\sigma\rightarrow \tau)\circ p_\sigma\Big) \quad .
\end{aligned}
\end{equation}
Now suppose $\mathcal{V}=(V_j)_{j\in J}$ is some other open cover of $X$. For every $x\in \text{Vert}(X)$ choose a set $V_{r(x)}$ such that $x\in V_{r(x)}$. This defines a map $r\colon\text{Vert}(X)\rightarrow J$. By the definition of an open set in the Alexandrov topology, $x\in V_{r(x)}$ implies $U_x\subseteq V_{r(x)}$. Therefore, $\mathcal{U}=(U_x)_{x\in \text{Vert}(X)}$ is a refinement of the cover $\mathcal{V}$ by Definition \ref{709}. Since $\mathcal{V}$ was arbitrary, the open cover $\mathcal{U}$ refines every open cover of $X$, i.e.\ it is terminal in the category $\mathbf{Cover}(X)$. This implies that $\check{H}^k(\mathcal{U},\mathcal{F})$ is a colimit of the functor $\check{H}^k(-,\mathcal{F})$ and hence $\check{H}^k(X,\mathcal{F})\cong\check{H}^k(\mathcal{U},\mathcal{F})$. Moreover, by Proposition \ref{538}, $H^k(X,\mathcal{F})\cong \check{H}^k(X,\mathcal{F})$. Since in an abelian category finite products and coproducts are isomorphic, (\ref{784}) and (\ref{852}) imply the result. 
\end{proof}

\noindent
Note that if $\phi\colon\mathcal{F}\rightarrow \mathcal{G}$ is a sheaf morphism on a finite simplicial complex, then, by using Definition \ref{115} and the arguments in the proof of Theorem \ref{742}, we obtain, for every $k\in\mathbb{N}_0$, the following formula for the induced morphism $\phi^k\colon C^k(X,F)\rightarrow C^k(X,G)$
\begin{equation*} 
\phi^k\coloneqq \bigoplus_{\sigma\in X^k} \phi_\sigma\circ p_\sigma \quad .  
\end{equation*}

\subsection{Morphisms induced by simplicial maps} \label{342}

Given a simplicial map $f\colon X\rightarrow Y$ and a sheaf $\mathcal{F}$ on $Y$. In this section, we use the construction of Section \ref{809} to derive an explicit formula of a cochain morphism $f^\bullet$ induced by $f$ on the cochain complexes (\ref{437}) derived in Section \ref{233} such that the maps induced in cohomology are the maps $H^k(f)\colon H^k(Y,\mathcal{F})\rightarrow H^k(X,f^*\mathcal{F})$. 

\begin{theorem} \label{546}
Let $X$ and $Y$ be abstract simplicial complexes, $f:X\rightarrow Y$ a simplicial map and $\mathcal{F}$ a sheaf on $Y$. Let $F\colon\mathbf{X}\rightarrow\mathbf{M}$ be the corresponding functor on $X$. Then the induced morphism $H^k(f)\colon H^k(Y,F)\rightarrow H^k(X,f^*F)$ is given by $H^k(f^\bullet)$, where $f^\bullet\colon C^\bullet(Y,F)\rightarrow C^\bullet(X,f^*F)$ is a cochain morphism on (\ref{437}) such that $f^k\colon C^k(Y,F) \rightarrow C^k(X,f^*F)$ is defined by
\begin{equation*} 
f^k\coloneqq \underset{\sigma\in X}{\bigoplus}\text{ } p_{f(\sigma)} \quad .
\end{equation*}
\end{theorem}

\begin{proof}
We use the constructions of Section \ref{809} and \ref{233}. To construct $\check{H}^k(f)$ we first have to construct $\eta_\mathcal{F}\colon \mathcal{F}\rightarrow f_*f^*\mathcal{F}$. Let $\tau\in Y$, then $\eta_\mathcal{F}$ is defined on the basic open subset $U_\tau\subseteq Y$ as the canonical morphism to the colimit \cite[Section 2.4]{iversen}
\begin{equation} \label{198} 
\mathcal{F}(U_\tau)\rightarrow \underset{f(f^{-1}(U_\tau))\subseteq V}{\text{colim}} \mathcal{F}(V)\cong f_*f^*\mathcal{F}(U_\tau) \quad .
\end{equation}  
By the defining properties of simplicial complexes and simplicial maps, if $f^{-1}(U_\tau)=\{\sigma\in X\text{ }|\text{ }f(\sigma)\geq \tau\}\neq \emptyset$, then there exists a $\sigma\in X$ such that $f(\sigma)=\tau$ and $\tau\in f(f^{-1}(U_\tau))$. Hence, if $f(f^{-1}(U_\tau))\neq\emptyset$, then for every open subset $f(f^{-1}(U_\tau)) \subseteq V\subseteq Y$ we have $U_\tau\subseteq V$. Therefore, either $\emptyset$ or $U_\tau$ is a minimal open subset containing $f(f^{-1}(U_\tau))$, thus the colimit in (\ref{198}) is isomorphic to $\mathcal{F}(\emptyset)$ or $\mathcal{F}(U_\tau)$, respectively. This implies we obtain
\begin{equation*} 
(\eta_\mathcal{F})_{U_\tau}=\begin{cases} \mathcal{F}(U_\tau)\xrightarrow{\text{id}}\mathcal{F}(U_\tau)\cong f_*f^*\mathcal{F}(U_\tau) \quad\text{if} \quad \tau\in f(X) \\ \mathcal{F}(U_\tau)\xrightarrow{\text{ }\text{ }\text{ }\text{ }\text{ }0\text{ }\text{ }\text{ }\text{ }}0\cong f_*f^*\mathcal{F}(U_\tau) \quad\text{if} \quad \tau\notin f(X) \end{cases}
\end{equation*}
for the sheaf morphism induced by $\eta$. As we saw in the proof of Theorem \ref{742}, for the open cover $\mathcal{U}=(U_y)_{y\in\text{Vert}(Y)}$, we obtain $\check{H}^k(Y,\mathcal{F})\cong\check{H}^k(\mathcal{U},\mathcal{F})$. Therefore, we get the following morphism $\eta_\mathcal{F}^\bullet\colon C^k(\mathcal{U},\mathcal{F})\rightarrow C^k(\mathcal{U},f_*f^*\mathcal{F})$ induced on \v{C}ech cochains
\begin{equation*} 
(\eta_\mathcal{F})^k=\prod_{\tau\in f(X)^k} \text{id}\circ p_\tau\colon\prod_{\tau\in Y^k} \mathcal{F}(U_\tau)\rightarrow \prod_{\tau\in f(X)^k} \mathcal{F}(U_\tau) \quad . 
\end{equation*}
The induced morphism $\check{H}^k(Y,\eta_\mathcal{F})$ is the morphism induced in cohomology by $\eta_\mathcal{F}^\bullet$. 

The second step in the construction of $\check{H}^k(f)$ is the construction of the morphism $\check{H}^k(f^{-1}(\mathcal{U}),f^*\mathcal{F})\rightarrow \check{H}^k(X,f^*\mathcal{F})$. The open cover $\overline{\mathcal{U}}=(\overline{U}_x)_{x\in \text{Vert}(X)}$ refines every other open cover of $X$. Let $\sigma\in \text{Vert}(X)$ and $\overline{U}_\sigma$ be the corresponding basic open subset, then $\overline{U}_\sigma\subseteq f^{-1}(U_{f(\sigma)})$. The map $r$, defined by $r(\sigma)\coloneqq f(\sigma)$, corresponding to the refinement of $f^{-1}(\mathcal{U})$ by $\overline{\mathcal{U}}$, induces a cochain morphism $r^\bullet\colon C^\bullet(f^{-1}(\mathcal{U}),f^*\mathcal{F})\rightarrow C^\bullet(\overline{\mathcal{U}},f^*\mathcal{F})$. By Definition \ref{709}, we get
\begin{equation*} 
r^k=\prod_{\sigma\in N_{\overline{\mathcal{U}}}^k} \Big(f^*\mathcal{F}\big(\overline{U}_\sigma\xhookrightarrow{} f^{-1}(U_{f(\sigma)})\big)\circ p_{f(\sigma)}\Big) \quad .
\end{equation*}
Note that $f^*\mathcal{F}\big(f^{-1}(U_{f(\sigma)})\big)=f_*f^*\mathcal{F}\big(U_{f(\sigma)}\big)\cong\mathcal{F}\big(U_{f(\sigma)}\big)$ and $f^*\mathcal{F}(\overline{U}_\sigma)\cong\mathcal{F}(U_{f(\sigma)})$. Hence, for the induced map on the colimits defining $f^*\mathcal{F}$, we obtain $\text{id}=f^*\mathcal{F}\big(\overline{U}_{\sigma}\xhookrightarrow{} f^{-1}(U_{f(\sigma)})\big)\colon\mathcal{F}\big(U_{f(\sigma)}\big)\rightarrow \mathcal{F}\big(U_{f(\sigma)}\big)$. We now define $f^\bullet\coloneqq r^\bullet\circ \eta_\mathcal{F}^\bullet\colon C^\bullet(\mathcal{U},\mathcal{F})\rightarrow C^\bullet(\overline{\mathcal{U}},f^*\mathcal{F})$ and get
\begin{equation*} 
f^k\coloneqq r^k\circ \eta_\mathcal{F}^k=\prod_{\sigma\in N_{\overline{\mathcal{U}}}^k} p_{f(\sigma)}\circ \prod_{\tau\in f(X)^k} p_{\tau}=\prod_{\sigma\in N_{\overline{\mathcal{U}}}^k} p_{f(\sigma)} \quad .
\end{equation*}
The cohomology $H^k(f^\bullet)\colon\check{H}^k(\mathcal{U},\mathcal{F})\rightarrow\check{H}^k(\overline{\mathcal{U}},f^*\mathcal{F})$ of the cochain morphism $f^\bullet$ is the morphism $\check{H}^k(f)\colon\check{H}^k(Y,\mathcal{F})\rightarrow\check{H}^k(X,f^*\mathcal{F})$ induced by $f$. Hence, by Theorem \ref{742} and Remark \ref{851}, we obtain the desired result.
\end{proof}

\end{document}